\documentclass[11pt,a4paper]{article}
\usepackage{amsfonts}

\usepackage{mathrsfs}

\usepackage{stmaryrd}
\usepackage{amsmath}
\usepackage{amssymb}
\usepackage{xcolor}
\usepackage{bbm}
\usepackage{mathrsfs}
\usepackage{graphicx}
\usepackage{lipsum}

\usepackage{enumerate}
\usepackage{theorem}
\usepackage{indentfirst}
\usepackage{hyperref}
\hypersetup{
	colorlinks=true,    
	linkcolor=blue,
	citecolor=red,
}
\makeatletter
\def\tank#1{\protected@xdef\@thanks{\@thanks
		\protect\footnotetext[0]{#1}}}
\def\bigfoot{
	
	\@footnotetext}
\makeatother

\newcommand{\ea}{\end{array}}

\allowdisplaybreaks[4]

\newtheorem{theorem}{Theorem}[section]

\newtheorem{lemma}{Lemma}[section]
\newtheorem{definition}{Definition}[section]
\newtheorem{remark}{Remark}[section]
\numberwithin{equation}{section}

{\theorembodyfont{\rmfamily}
}

\newenvironment{proof}{Proof.}

\begin{document}
\title{{\Large \bf Averaging principle for semilinear slow-fast rough partial differential equations}}
\author{{Miaomiao Li$^{a}$},~~{Yunzhang Li$^{b}$},~~{Bin Pei$^{a,c,d}$}\footnote{Corresponding author: binpei@nwpu.edu.cn},~~{Yong Xu$^{a,d}$}
	\\
	\scriptsize $a.$ School of Mathematics and Statistics, Northwestern Polytechnical University, Xi'an 710072, China \\
    \scriptsize $b.$ Research Institute of Intelligent Complex Systems, Fudan University, Shanghai 200433, China\\
    \scriptsize $c.$ Research and Development Institute of Northwestern Polytechnical University in Shenzhen, Shenzhen 518057, China \\
    \scriptsize $d.$ MOE Key Laboratory of Complexity Science in Aerospace, Northwestern Polytechnical University, Xi'an 710072, China }

\date{}
\maketitle
\begin{center}
	\begin{minipage}{145mm}
		
		{\bf Abstract.} In this paper, we investigate the averaging principle for a class of semilinear slow-fast partial differential equations driven by finite-dimensional rough multiplicative noise.
Specifically, the slow component is driven by a general random $\gamma$-H\"{o}lder rough path for some $\gamma \in (1/3,1/2)$, while the fast component is driven by a Brownian rough path.
Using controlled rough path theory and the classical Khasminskii's time discretization scheme,
we demonstrate that the slow component converges strongly to the solution of the corresponding averaged equation under the H\"{o}lder topology.

\vspace{3mm} {\bf Keywords:} Rough path theory; averaging principle; rough partial differential equations; slow-fast system

	\end{minipage}
\end{center}

\tableofcontents

\section{Introduction}

In this paper, we consider the following semilinear slow-fast rough partial differential equations (RPDEs for short)
\begin{equation}\label{eq1}
	\left\{ \begin{aligned}
		&dX_t^\varepsilon=\big[A X_t^\varepsilon+F_1 (X_t^\varepsilon,Y_t^\varepsilon)\big]dt+G_1 (X_t^\varepsilon)d \mathbf{B}_t,\\
		&dY_t^\varepsilon=\frac 1\varepsilon \big[A Y_t^\varepsilon+F_2 (X_t^\varepsilon,
Y_t^\varepsilon)\big]dt+\frac 1{\sqrt{\varepsilon}}G_2 (X_t^\varepsilon,
Y_t^\varepsilon)d \mathbf{W}_t,\\
		&X_0^\varepsilon=x, ~~ Y_0^\varepsilon=y,
	\end{aligned}  \right.
\end{equation}
where $\varepsilon >0$ is a small parameter that describes the ratio of time-scale between fast component $Y^\varepsilon $ and slow component $X^\varepsilon $,
$\mathbf{B}=(B,B^2)$ is a $d$-dimensional random $\gamma$-H\"{o}lder rough path for every $\gamma \in (1/3,\gamma_0), 1/3< \gamma_0 \leq 1/2$,
$\mathbf{W}=(W,W^2)$ is a $m$-dimensional Brownian rough path (also known as enhanced Brownian motion, BM for short),
and $W, B$ are independent on a complete filtered probability space $(\Omega, \mathscr{F}, \{\mathscr{F}_t\}_{t \geq 0}, \mathbb P)$,
$A :{\rm Dom}(A)\subset \mathcal{H} \to \mathcal{H}$ is the infinitesimal generator of the analytic $C_0$-semigroup $(S_t)_{t \geq 0}$ on the separable Hilbert space $\mathcal{H}$,
the coefficients $F_1,F_2$ and $G_1,G_2$ are nonlinear terms satisfying certain suitable regularity assumptions.
We will establish the strong averaging principle for system (\ref{eq1}), which can be regarded as a functional law of large numbers.
More precisely, we will prove that the slow component $X^\varepsilon$ converges strongly to $\bar{X}$ in the sense of $L^2(\Omega;\hat{C}^\eta([0,T]; \mathcal H)), \eta \in (\gamma-1/4, \gamma)$ (see Definition \ref{space} for $\hat{C}^\eta([0,T]; \mathcal H)$) as $\varepsilon \rightarrow 0$, i.e.
\begin{equation}\label{aim}
\lim_{\varepsilon\rightarrow 0}  \mathbb E\big[\|X^\varepsilon-\bar{X}\|_{\eta, 0}^{2}\big]=0,
\end{equation}
where $\bar{X}$ is the solution of the following averaged equation
\begin{equation*}
\left\{ \begin{aligned}
d\bar{X}_t &=A\bar{X}_t dt +\bar{F}_1 (\bar{X}_t)dt +G_1(\bar{X}_t)d \mathbf{B}_t,\\
\bar{X}_0 &=x,
\end{aligned} \right.
\end{equation*}
with the averaged coefficient given by
$$
\bar{F}_1(x)=\int_{\mathcal H} F_1(x,y) \mu^x (dy),~~~~ x \in \mathcal H,
$$
and $\mu^x$ is the unique invariant measure of the process $Y_t^{x,y}$,
which is the unique solution of the following frozen equation (for a fixed slow component $x \in \mathcal H$):
$$
dY_t= A Y_t dt+F_2(x,Y_t)dt+G_2(x,Y_t) d \widetilde{W}_t,~~~~ Y_0=y \in \mathcal H,
$$
where $\widetilde{W}_t$ is a $m$-dimensional BM on another probability space $(\widetilde{\Omega} ,\widetilde{\mathscr F},\{\widetilde{\mathscr F}_t\}_{t\geq0},\widetilde{\mathbb P})$.

The slow-fast system (also called the two-time scales system) is very common in many real-world dynamical systems such as climate weather
interactions, stochastic volatility in finance, geophysical fluid flows and macro-molecules, see e.g. \cite{BKRP05,FFK12,GD03,KK13,Ki01} and the references therein.
However, it is very difficult to directly analyze or calculate such complex systems due to the two widely separated time scales and the cross-interaction between slow and fast components.
Therefore, a simplified equation that governs the evolution of the original system over a long time scale is highly desirable.
This simplified equation, which captures the essential dynamics of the original system, is often referred to as the averaged equation.
In particular, it no longer depends on the fast component and is therefore more suitable for analysis and simulation.
This theory, known as the averaging principle, demonstrates that the trajectory of the slow component $X^\varepsilon$ can be approximated by the solution $\bar{X}$ of the so-called averaged equation.

The theory of averaging, originated by Laplace and Lagrange, has been the subject of much attention in the past decades.
The first rigorous result was formulated by Bogoliubov and Mitropolsky in \cite{BM61} for ordinary differential equations (ODEs for short).
The further development of the theory in the case of stochastic differential equations (SDEs for short) was established by Khasminskii \cite{K68}.
Subsequently, the averaging principle of different types of slow-fast SDEs has been investigated extensively,
see e.g. \cite{BK04,ELV05,GR16,LDi10,LRSX20} and the references therein.
However, for stochastic partial differential equations (SPDEs for short), the problem becomes even more difficult.
In \cite{C09,CF09}, Cerrai and Freidlin established the averaging principle for slow-fast stochastic reaction-diffusion equations.
Since then, the averaging principle of SPDEs has progressed considerably over the last fifteen years.
For instance, Br\'{e}hier \cite{B12,B20} established the strong and weak orders in averaging for stochastic evolution equations of parabolic type with slow and fast time scales.
Sun et al. \cite{SXX21} demonstrated the averaging principle for slow-fast SPDEs with H\"{o}lder continuous drift coefficients by using the technique of Zvonkin's transformation and Khasminskii's time discretization method.
Moreover, the averaging principle for a class of slow-fast SPDEs with locally monotone coefficients was proved by Liu et al. \cite{LRSX23} using the techniques of time discretization and stopping time.
We refer the interested readers to \cite{BYY18,CL17,CZ22,CL23,DW14,FWL15,HLL22,WR12} and references therein for further results on this subject.

We want to point out that all of the references mentioned above assume that the driving noises are (semi-)martingales.
A natural question is whether the averaging principle still holds true when the driving noise does not have (semi-)martingale property?
A prominent example of such noise is fractional Brownian motion (FBM for short) with the Hurst index $H \in (0, 1/2) \cup (1/2, 1)$.
When $H \in (1/2, 1)$,
Hairer and Li \cite{HL20} considered a slow-fast system where the slow component is driven by FBM and proved that convergence to the averaged solution takes place in probability.
Pei and coauthors \cite{PIX23} further considered the averaging principle for slow-fast SDEs with perturbations involving BM and FBM
and confirmed that it holds in the mean square sense.
One can also refer to \cite{HXPW22,PIX20,WXY22} for further developments.
In these works, the integral with respect to (w.r.t.) FBM is interpreted as a Young integral (i.e. a generalized Riemann-Stieltjes integral).
Moreover, all of these works resort to Markovian because the fast component is driven by BM.
Notably, in recent work \cite{LS22}, Li and Sieber made the first attempt to address non-Markovian fast dynamics.
Later, Pei et al. \cite{PSX24} also considered a slow-fast system of SPDEs with non-Markovian fast components and derived the almost sure averaging principle.

However, the situation for $H < 1/2$ became quite involved and complicated because neither the Young integral nor the It\^{o} integral is applicable.
Rough path theory was originally developed by Lyons \cite{L98,LQ02}. It provides a pathwise approach to stochastic analysis and introduces a new integration theory with which one can solve SDEs driven by irregular signals and, in particular, by a FBM with $H < 1/2$ (see e.g. \cite{CQ02}).
The main feature of this theory is that a path $X$ in the vector space $V$ should not be regarded as determined by a function from an interval $I \subset \mathbb R$ to $V$,
but rather, if this path is not sufficiently regular, some additional information is needed which would play the role of the iterated integrals for regular paths, e.g. quantities like the rank two tensor:
$$X_{t,s}^2=\int_s^t \int_s^u dX_v dX_u $$
and its generalizations.
More precisely, a rough path is the original path together with its iterated integrals, i.e. a rough path $\mathbf{X} =(X, X^2)$.
It is worth highlighting that a $d$-dimensional BM $\{W_t\}_{t \geq 0}$ can also be enhanced with
$W_{t,s}^2:=\int_s^t (W_u-W_s) \otimes d W_u,$
so that $\mathbf W=(W,W^2)$ is almost surely a $\gamma$-H\"{o}lder rough path, for any $\gamma \in (1/3,1/2)$ (see \cite[Section 3]{FH14}).
The integration here is understood either in It\^{o} or Stratonovich sense.
In particular, we can use almost every realisation of $(W,W^2)$ as the driving signal of a rough differential equation (RDE for short)
$$d Y_t=b(Y_t)dt+\sigma (Y_t) d \mathbf W_t, ~~~~Y_0=y.$$
This RDE is then solved ``pathwise'' for a fixed realisation of $(W(\omega), W^2(\omega))$.
In contrast to classical SDE theory, the solution constructed via RDE is immediately well-defined as a flow, i.e. for all $y$ on a common set of
probability one.
The trade-off, however, is that the coefficient $\sigma$ must have $C^3$ regularity, as opposed to the mere Lipschitz regularity required for the classical theory.

Another motivation for developing rough path theory was to explore the dynamical properties of stochastic equations.
It is well-known that an It\^{o} SDE generates a random dynamical system (RDS) due to the flow property (see Kunita \cite{K90}).
However, whether an It\^{o} SPDE can generates a RDS has been a long-standing open problem.
The main reason is that it is not known how to extend Kolmogorov's theorem to an infinite-dimensional parameter range,
which is crucial to obtaining stochastic flows for It\^{o} SPDEs.
For SPDEs driven by additive or linear multiplicative noise,
there are existing results regarding the RDS.
This is because these special noises make it possible to transform an It\^{o} SPDE into a pathwise partial differential equation (PDE for short),
which is straightforward to generate a RDS.
However, for SPDEs driven by nonlinear multiplicative noise, this technique no longer works.
Rough path theory makes it possible to study the generation of a RDS from SPDEs driven by nonlinear multiplicative noise,
due to the fact that the solution exists pathwise, see e.g. \cite{GLS,HN20,HN22,KN23}.

Research on rough path theory has reached a certain level of maturity.
For example, the applications of the rough path analysis to the stochastic calculus w.r.t. BM are explored in \cite{CL05,FV05}.
Ledoux et al. \cite{LQZ02} established the large deviation principle and the support theorem for diffusion processes via rough paths.
On the other hand, Gubinelli \cite{G04} introduced a variant of rough path theory,
known as controlled rough path,
which leads to the same results as those of Lyons but, to some extent, is simpler and more straightforward.
Furthermore, as an application of the theory, the existence and uniqueness of solutions to ODEs driven by irregular paths are also derived in \cite{G04}.
In \cite{GT10}, Gubinelli and Tindel extended the theory of controlled rough paths to solve not only SDEs but also SPDEs driven by the infinite-dimensional Gaussian process.
Moreover, Gerasimovi\v{c}s and Hairer \cite{GH19} studied a class of semilinear SPDEs driven by finite-dimensional Wiener processes
using the theory of controlled rough paths, where the semigroup is incorporated into the definition of controlled rough paths.
Interested readers can also refer to \cite{FH14,GHN21,IXY24,YLZ23,YXP25} and the references therein for further studies.

In this paper, we want to consider the slow-fast system driven by irregular signals within the framework of rough path theory.
To the best of the authors' knowledge, such slow-fast systems have not been studied much yet.
Even in the finite-dimensional case, there are only a few results in this direction.
Pei et al. \cite{PIX21} investigated the averaging principle for a slow-fast system of RDEs driven by mixed fractional Brownian rough paths by using the theory of rough paths,
where the fast component is driven by BM and the slow component is driven by FBM with the Hurst index $H \in (1/3, 1/2]$.
Recently, the strong convergence of slow-fast RDEs in the $L^p(\Omega)$-sense ($p \in [1,\infty)$) was demonstrated by Inahama \cite{In22} using controlled rough path theory,
and the corresponding convergence rates were investigated by Yan et al. \cite{YJYM24} using the Poisson equation technique.
In addition, Pei et al. \cite{PHSX23} addressed SDEs where both slow and fast components are driven by FBM with $H \in (1/3, 1/2]$.
Using controlled rough path theory, they proved that the slow component converges almost surely to the solution of the corresponding averaged equation.

Taking the above discussions into account,
the main purpose of this paper is to study the averaging principle of an infinite-dimensional slow-fast system driven by irregular paths.
Specifically, we will establish the strong averaging principle for semilinear slow-fast RPDEs (\ref{eq1}).
In particular, our result is applicable to slow-fast SPDEs where the slow component is driven by FBM with $H \in (1/3, 1/2]$.
To the best of our knowledge, this seems to be the first paper in which the averaging principle of slow-fast SPDEs driven by irregular paths is studied.

It should be noted that although the averaging principle for SDEs driven by irregular paths has been studied in \cite{In22}, it is quite different from the case of SPDEs.
For example, in \cite{In22} one only needs to define the rough convolution $\int_0^t G(Y_s) d \mathbf{X}_s$.
However, in the infinite-dimensional case, because system (\ref{eq1}) is solved by the semigroup approach,
one of the main problems we encounter is to define a rough convolution of the form $\int_0^t S_{t-s} G(Y_s) d \mathbf{X}_s$.
This is a challenging task due to the lack of regularity of the semigroup $(S_t)_{t \geq 0}$ in zero.
To overcome this difficulty, the semigroup was directly incorporated into the definition of the controlled rough path, as in \cite{GH19}.
Moreover, note that the strong convergence (\ref{aim}) is established in the H\"{o}lder topology with the semigroup $S$,
it turns out to be quite involved to deal with the nonlinear terms that appeared in (\ref{eq1}).
The technique used here is mainly based on the effective scheme of time discretization,
which has been widely used in the literature to study the averaging principle for different types of slow-fast stochastic systems.
We want to emphasize again that once we translate our problem to the language of rough paths, most of the arguments are deterministic.
Therefore, the way in which we finally can establish the existence of a global solution and the averaging principle is based on a concatenation procedure.

The organization of this paper is as follows:
In Section \ref{sect2}, we introduce some notations for functional spaces and controlled rough paths with the semigroup,
and then give the main result of this paper.
In Section \ref{sect3}, we give a rigorous introduction to the slow-fast system of RPDEs from a deterministic and probabilistic perspective, respectively,
and show that the fast component driven by the Brownian rough path is consistent with a SPDE in the sense of It\^{o}.
In Section \ref{Proof}, we present the detailed proof of the averaging principle in the framework of controlled rough path theory.

Throughout this paper, the notation $C$ with or without subscripts will represent a positive constant, whose value may vary from one place to another,
and $C$ with subscripts will be used to emphasize that it depends on certain parameters.

\section{Preliminary and main result}\label{sect2}

In this section, we will review some functional settings of rough path theory and present the main result of this paper.

\subsection{Functional analytic framework}

Throughout this paper we consider a separable Hilbert space $(\mathcal{H},\left\langle \cdot , \cdot \right\rangle_{\mathcal{H}})$.
Let $A$ be a negative definite self-adjoint operator,
and let $S_t$ or $e^{A t}$ be the semigroup generated by $A$.
For $\alpha \geq 0$, the interpolation space $\mathcal{H}_\alpha:= {\rm Dom}((-A)^\alpha)$ is a Hilbert space endowed with the norm
$\|  \cdot  \|_{\mathcal{H}_\alpha}=\|(-A)^\alpha \cdot \|_{\mathcal{H}}$.
Similarly, $ \mathcal{H}_{-\alpha }$ is defined as the completion of $\mathcal{H}$ w.r.t. the norm
$\| \cdot \|_{\mathcal{H}_{-\alpha }}=\| (-A)^{-\alpha} \cdot \|_{\mathcal{H}}$.
Then we obtain a family of separable Banach spaces $(\mathcal{H}_\alpha, \|  \cdot  \|_{\mathcal{H}_\alpha})_{\alpha \in \mathbb R}$
with continuous and dense embedding, i.e. $\mathcal{H}_\alpha \hookrightarrow \mathcal{H}_\beta$ for $\alpha \geq \beta$.
Denote by $\mathcal L(\mathcal{H}_\alpha;\mathcal{H}_\beta)$ the space of all bounded linear operators from $\mathcal{H}_\alpha$ to $\mathcal{H}_\beta$ with the norm $\|\cdot \|_{\mathcal L(\mathcal{H}_\alpha;\mathcal{H}_\beta)}$ for $\alpha, \beta \in \mathbb R$.
If $\alpha=\beta$, we write $\mathcal L(\mathcal{H}_\alpha)=\mathcal L(\mathcal{H}_\alpha; \mathcal{H}_\alpha)$.

With this at hand, we can view the semigroup $(S_t)_{t \geq 0}$ as a linear mapping between these interpolation spaces
and obtain the following standard bounds:
For every $\alpha \geq \beta $ and $\sigma \in [0,1]$, one has the following estimates, which we label for further use, hold true.
\begin{equation}\label{semigroup}
\|S_t u \|_{\mathcal{H}_\alpha } \lesssim t^{\beta-\alpha} \|u\|_{\mathcal{H}_\beta},~~~~
\|(S_t-Id) u \|_{\mathcal{H}_{\beta-\sigma} } \lesssim t^\sigma \|u\|_{\mathcal{H}_\beta}, ~~~~u \in \mathcal{H}_\beta.
\end{equation}
Note that $x \lesssim y $ here means that there exists a constant $C>0$ such that $x \leq Cy$.
A more detailed introduction to analytic semigroup theory can be found in for example \cite{L95,p83}.

\begin{remark}
It is not difficult, according to $(\ref{semigroup})$, to obtain that
\begin{equation}\label{semi}
\|S_t\|_{\mathcal L (\mathcal{H}_\beta; \mathcal{H}_\alpha)} \lesssim t^{\beta-\alpha},~~~~
\|S_t-Id \|_{\mathcal L (\mathcal{H}_\beta; \mathcal{H}_{\beta-\sigma}) } \lesssim t^\sigma.
\end{equation}
\end{remark}

\subsection{Controlled rough path}

We now recall the notion of rough path proposed by Lyons in the 90's (see e.g. \cite{L98,LQ02}).
Since we want to focus on partial differential equations (PDEs) driven by finite-dimensional nonlinear rough multiplicative noise.
Therefore, we will use a local (in time) expansion of the controlled process as in \cite{GH19} to encode the interaction between the class of our integrand and the semigroup,
which is a slight modification of the notion of ``controlled rough path" in \cite{GT10}.

First of all, we define spaces of time increments of functions taking values in some Banach spaces,
which are very similar to those defined in \cite{DGT12,DT13,GH19,GT10}.
Fixed $T>0$ and given a Banach space $(V, \|\cdot \|_V)$,
for $n \in \mathbb N$, define $ C_n(V):=C(\Delta_n; V)$ the space of all continuous functions from $n$-simplex $\Delta_n=\{(t_1, \cdots, t_n): T \geq t_1 \geq t_2 \geq \cdots \geq t_n \geq 0\}$ to $V$.
In general, we write $C([0,T];V)\equiv C_1(V)$ is just the usual space of continuous functions with values in $V$ on $[0,T]$ and endow with the supremum norm $\|f\|_{\infty,V}=\displaystyle \sup_{t\in [0,T]} \|f_t\|_V$.
In this paper, we will only focus on the case where $n=1,2,3 $.
For any $0 \leq s \leq u \leq t \leq T$, $f \in C([0,T];V)$ and $g \in C_2(V)$, we define the increment operator $\delta$ as
$${\delta f}_{t,s}:=f_t-f_s,~~~~{\delta g}_{t,u,s}:=g_{t,s}-g_{t,u}-g_{u,s},$$
and the reduction increment operator $\hat{\delta}$ is given by
$${\hat{\delta} f}_{t,s}:=f_t-S_{t-s}f_s,~~~~{ \hat{\delta} g}_{t,u,s}:=g_{t,s}-g_{t,u}-S_{t-u}g_{u,s}.$$

\begin{definition} \label{space}
Let $(V, \|\cdot \|_V)$ be a Banach space.
For any $T>0$ and $0<\gamma \leq 1$ we set
\begin{align*}
	 |f|_{\gamma,V}:=& \sup_{0 \leq s <t \leq T} \frac {\|{\delta f}_{t,s}\|_V}{|t-s|^\gamma }, ~~f\in C([0,T];V),\\
	 \|f\|_{\gamma,V}:=& \sup_{0 \leq s <t \leq T} \frac {\|{\hat{\delta} f}_{t,s}\|_V}{|t-s|^\gamma }, ~~f\in  C([0,T];V), \\
	 |g|_{\gamma,V}:=& \sup_{0 \leq s <t \leq T} \frac {\|g_{t,s}\|_V}{|t-s|^\gamma }, ~~g\in C_2(V).
\end{align*}
Next we define the spaces as follows
\begin{align*}
	& C^\gamma ([0,T];V):=\{ f \in C([0,T];V) : |f|_{\gamma,V} < \infty \},\\
    & \hat{C}^\gamma ([0,T];V):=\{ f \in C([0,T];V) : \|f\|_{\gamma,V} < \infty \}, \\
	&  C_2^\gamma ([0,T];V):=\{ g \in C_2(V) : |g|_{\gamma,V} < \infty \}.
\end{align*}

\end{definition}

\begin{remark}

\emph{(i)} Note that we have made an abuse of the notation on $\gamma$-H\"{o}lder norms of one-parameter and two-parameter
functions, but we hope that no confusion will arise from this.

\emph{(ii)} In the case of $V=\mathcal{H}_\alpha~ (\alpha \in \mathbb R )$,
we will write $|f|_{\gamma,\mathcal{H}_\alpha}=|f|_{\gamma, \alpha}$, $\|f\|_{\gamma,\mathcal{H}_\alpha}=\|f\|_{\gamma, \alpha}$ and $\|f\|_{\infty,\mathcal{H}_\alpha}=\|f\|_{\infty, \alpha}$ for simplicity.

\emph{(iii)} In the special case $V=\mathbb R^d ~(d \in \mathbb N)$,
the letter $V$ is omitted in the above notations for simplicity,
e.g., when no confusion is possible, we write $ |\cdot |_{\gamma,\mathbb R^d}=|\cdot|_{\gamma}$.

\end{remark}

We can now provide the definition of rough paths on a finite-dimensional space $\mathbb R^d$.
\begin{definition}\label{RP}
	{\rm (Rough Path)} For $\gamma \in (1/3, 1/2]$, a pair $\mathbf{X} =(X, X^2)$ is called $\gamma$-H\"{o}lder rough path $($over $\mathbb R^d)$.
	 If $X \in C^\gamma ([0,T]; \mathbb R^d)$, $X^2 \in C_2^{2\gamma}([0,T]; \mathbb R^d \otimes \mathbb R^d)$ and satisfy Chen's relation
	 \begin{equation}\label{chen}
	 	X^2_{t,s}-X^2_{t,u}-X^2_{u,s}= \delta X_{t,u} \otimes \delta X_{u,s}, ~~\text{for all}~~0\leq s \leq u \leq t \leq T.
	 \end{equation}
Denote by $ \mathscr{C}^\gamma ([0,T]; \mathbb R^d)$ the space of $\gamma$-H\"{o}lder rough paths
and the distance between two $ \gamma$-H\"{o}lder rough paths $\mathbf X=(X, X^2)$ and $\widetilde{\mathbf X}=(\widetilde{X}, \widetilde{X}^2)$ is defined by
$$d_\gamma(\mathbf X,\widetilde{\mathbf X})=|X-\widetilde{X}|_\gamma +|X^2-\widetilde{X}^2|_{2\gamma}.$$
We write $d_\gamma(\mathbf X)=d_\gamma(0, \mathbf X)$ for simplicity.
\end{definition}

This paper is concerned with equations driven by $\sum_{k=1}^d G^k(u_t)d \mathbf X_t^k$,
where $\mathbf X=(\mathbf X^1, \cdots, \mathbf X^d)$ is a rough path on $\mathbb R^d$ and $G^k: \mathcal{H}_\alpha \rightarrow \mathcal{H}_\beta, \alpha, \beta \in \mathbb R$.
We will generally use the shorthand notation $G(u_t)d \mathbf X_t$, where $G=(G^1, \cdots, G^d):\mathcal{H}_\alpha \rightarrow \mathcal L (\mathbb R^d;\mathcal{H}_\beta)$.
For simplicity we write $\mathcal{H}_\beta^d:=\mathcal L (\mathbb R^d;\mathcal{H}_\beta)$ (resp. $\mathcal{H}_\beta^{d \times d}:=\mathcal L (\mathbb R^d \otimes \mathbb R^d;\mathcal{H}_\beta)$)
and simply write $\mathcal{H}^d:=\mathcal{H}_0^d=\mathcal L(\mathbb R^d; \mathcal H)$ (resp. $\mathcal{H}^{d \times d}:=\mathcal{H}_0^{d\times d}$) when $\beta =0$.

\begin{definition}\label{CRP}
{\rm (Controlled Rough Path)}
Let $\mathbf X=(X,X^2) \in \mathscr{C}^\gamma ([0,T]; \mathbb R^d)$ for $\gamma \in (1/3,1/2]$.
We call a pair $(Y, Y')$ a controlled rough path,
if $(Y, Y') \in \hat{C}^\gamma ([0,T];\mathcal{H}_\alpha) \times \hat{C}^\gamma ([0,T]; \mathcal{H}_\alpha^d)$ for some $\alpha \in \mathbb R$
and the remainder term $R^Y$ defined by
	\begin{equation}\label{remainder}
		R_{t,s}^Y=\hat \delta Y_{t,s}-S_{t-s}Y_s' \delta X_{t,s} ,
	\end{equation}
belongs to $C_2^{2\gamma} ([0,T];\mathcal{H}_\alpha)$,
where the component $Y'$ is the Gubinelli derivative of $Y$.
We denote by $\mathscr D_{S, \mathbf X}^{2\gamma} ([0,T]; \mathcal{H}_\alpha )$ the space of all controlled rough paths that are controlled by $\mathbf X$ according to the semigroup $S$.
Define a seminorm on this space by
\begin{equation*}
\|Y,Y'\|_{X, 2\gamma, \alpha}
	=\|Y'\|_{\gamma,\alpha}+|R^Y|_{2\gamma, \alpha},
\end{equation*}
and its norm is given by
\begin{equation*}
\|Y,Y'\|_{\mathscr D_{S, \mathbf X}^{2\gamma}}=
	\|Y_0\|_{ \mathcal{H}_\alpha}
	+ \|Y_0'\|_{\mathcal{H}_\alpha^d} +\|Y,Y'\|_{X, 2\gamma, \alpha}.
\end{equation*}
\end{definition}

\begin{remark}
\emph{(i)} If $\mathcal{H}_\alpha$ is replaced by $\mathcal{H}_\alpha^d$, the above definition still holds,
and we will still use notations $\|\cdot \|_{\gamma,\alpha}$, $|\cdot |_{\gamma,\alpha}$, etc. to represent the corresponding norms, for the sake of notation simplicity.

\emph{(ii)} In the special case where $S=Id$, this is just the usual notion of controlled rough path that was introduced by Gubinelli in \emph{\cite{G04}} \emph{(see also $\cite{LY15}$ from a different perspective)}.

\emph{(iii)}  Note that we do not include the H\"{o}lder seminorm $\|Y\|_{\gamma,\alpha}$ of $Y$ in the definition of the \emph{(semi)}norm of the controlled rough path,
since by using $(\ref{remainder})$ one immediately obtain that
\begin{equation}\label{Y}
\|Y\|_{\gamma,\alpha} \leq |R^Y|_{2\gamma ,\alpha} T^\gamma+\|Y'\|_{\infty,\alpha} |X|_\gamma
\leq C(1+|X|_\gamma)(\|Y_0'\|_{\mathcal{H}_\alpha^d}+\|Y,Y'\|_{X, 2\gamma, \alpha}),
\end{equation}
where the constant $C$ depends only on $\gamma$ and $T$.
\end{remark}

With the definition of controlled rough paths, we can define rough convolution as follows.

\begin{theorem}\label{RoughInt}
\emph{(\cite[Theorem 3.5]{GH19})}
For any $T>0$, let $\mathbf X=(X,X^2)\in \mathscr{C}^\gamma([0,T]; \mathbb R^d)$ for some $\gamma\in(1/3,1/2]$ and let
$(Y,Y^{\prime})\in\mathscr{D}_{S,\mathbf X}^{2\gamma}([0,T];\mathcal{H}_\alpha^d)$.
Then one has that the rough integral
$$\int_{s}^{t}S_{t-u}Y_{u}d \mathbf X_u:=\lim_{|\mathcal{P}|\to 0}\sum_{[u,v]\in\mathcal{P}}S_{t-u}(Y_u \delta X_{v,u}+Y_{u}^{\prime} X^2_{v,u})$$
exists in $\hat{C}^\gamma([0,T];\mathcal{H}_\alpha)$,
where the limit is taken along any partitions $\mathcal P$ of $[s,t]$ such that mesh $|\mathcal P| \rightarrow 0$.
Furthermore, for any $0\leq\beta<3\gamma$ one has the bound
\begin{align} \label{IntBound}
& \Big \|\int_{s}^{t}S_{t-u}Y_{u}d \mathbf X_{u}-S_{t-s}Y_{s} \delta X_{t,s}-S_{t-s}Y_{s}^{\prime} X^2_{t,s} \Big\|_{\mathcal{H}_{\alpha+\beta}} \nonumber\\
 \lesssim &   \big(|R^{Y}|_{2\gamma,\alpha}|X|_{\gamma}+\|Y'\|_{\gamma,\alpha}|X^2|_{2\gamma}\big)|t-s|^{3\gamma-\beta},
\end{align}
and the map
$$(Y,Y')\rightarrow (U,U'):=\Big(\int_0^\cdot S_{\cdot-u}Y_u d \mathbf X_u,Y \Big)$$
is continuous from $\mathscr D_{S,\mathbf X}^{2\gamma} ([0,T]; \mathcal{H}_\alpha^d )$ to $\mathscr D_{S,\mathbf X}^{2\gamma} ([0,T]; \mathcal{H}_\alpha )$.
\end{theorem}

We finish this subsection by presenting a lemma whose main purpose is to demonstrate that rough integrals in system (\ref{eq1}) are meaningful.
Its proof can refer to \cite[Lemma 3.7]{GH19}, we omit it.
This conclusion can also be further extended to allow functions to lose some spatial regularity, see \cite[Lemma 3.11]{GH19}.

For some fixed $\alpha, \beta \in \mathbb R$ and $k,m,n \in \mathbb N$,
let $C_{\alpha, \beta}^k(\mathcal{H}^m;\mathcal{H}^n)$ be the space of $k$-order continuously (Fr\'{e}chet) differentiable functions
$G: \mathcal{H}_\theta^m \rightarrow \mathcal{H}_{\theta+\beta}^n ,\theta \geq \alpha$
with bounded derivatives $D^i G,i=0,1,\cdots,k$.
We simply write $C_{\alpha, \beta}^k(\mathcal{H}^m)=C_{\alpha, \beta}^k(\mathcal{H}^m;\mathcal{H}^m)$ for $m=n$.

\begin{lemma}\label{funCRP}
Let $G\in C_{\alpha,0}^2(\mathcal{H};\mathcal{H}^d)$,
and let $(Y,Y^{\prime})\in \mathscr{D}_{S,\mathbf X}^{2\gamma}([0,T];\mathcal{H}_\alpha)$ for each $T>0$ and $\mathbf X=(X,X^2)\in \mathscr{C}^{\gamma}([0,T];\mathbb R^d)$ with $\gamma \in (1/3,1/2]$.
In addition, we also assume that $Y\in \hat{C}^\eta([0,T];\mathcal{H}_{\alpha+2\gamma})$ with $\eta \in [0,1]$
and $Y^\prime \in  \mathcal L^\infty([0,T];\mathcal{H}_{\alpha+2\gamma}^d)$.
Define $(Z_{t},Z_{t}^{\prime})=(G(Y_t),DG(Y_t)\circ Y_t^{\prime})$,
then one has $(Z,Z^\prime)\in \mathscr{D}_{S,\mathbf X}^{2\gamma}([0,T];\mathcal{H}_\alpha^d)$ and the estimation
$$\|Z,Z'\|_{X,2\gamma,\alpha}\leq C(1+|X|_\gamma)^2(1+\|Y_0\|_{\mathcal{H}_{\alpha+2\gamma}}+\|Y\|_{\eta,\alpha+2\gamma}+\|Y^\prime\|_{\infty,\alpha+2\gamma}+\|Y,Y^\prime\|_{X,2\gamma,\alpha})^2,$$
where the constant $C$ depends on the bounds on $G$ and its derivatives, as well as on time $T$, but is uniform over $T\in(0,1]$.
\end{lemma}

\subsection{Main result}\label{AP}

The main goal of this paper is to study the averaging principle of system (\ref{eq1}).
For this purpose, we need to make some necessary assumptions about coefficients.

\begin{enumerate}
\item [$(\mathbf{H1})$]
     For $F_1 \in C_{-2\gamma, 0}(\mathcal{H} \times \mathcal{H}; \mathcal{H})$,
     $F_2 \in C(\mathcal{H} \times \mathcal{H} ;\mathcal{H})$,
     there exist  constants $C,L_{F_2}>0$ such that for all $x,x_1,x_2 \in  \mathcal{H}_\alpha$ and $y, y_1,y_2\in \mathcal{H}_\alpha$ with $\alpha \geq -2\gamma$, we have
\begin{equation}\label{F1Lip}
\|F_1(x_1,y_1)-F_1(x_2,y_2)\|_{\mathcal{H}_\alpha} \leq  C(\|x_1-x_2\|_{\mathcal{H}_\alpha}+\|y_1-y_2\|_{\mathcal{H}_\alpha}),
\end{equation}
\begin{equation}\label{F1Bound}
\|F_1\|_\infty:= \sup_{(x,y)\in \mathcal{H}_\alpha \times \mathcal{H}_\alpha} \|F_1(x,y)\|_{\mathcal{H}_\alpha} < \infty,
\end{equation}
and for $x_1,x_2,y_1,y_2 \in \mathcal{H}$ we also have
\begin{equation}\label{F2Lip}
\|F_2(x_1,y_1)-F_2(x_2,y_2)\|_{\mathcal{H}}\leq  C\|x_1-x_2\|_{\mathcal{H}}+L_{F_2}\|y_1-y_2\|_{\mathcal{H}}.
\end{equation}

\item [$(\mathbf{H2})$] For diffusion terms $G_1, G_2$, we assume that
		$$G_1 \in C_{-2\gamma, 0}^3(\mathcal{H}; \mathcal{H}^d),~~~~
        G_2 \in C_{-2\gamma, 0}^3(\mathcal{H} \times \mathcal{H}; \mathcal{H}^m),$$
        and there exists a constant $L_{G_2}>0$ such that for any $x_1,x_2,y_1,y_2 \in \mathcal{H}$, we have
     \begin{equation}\label{Glip}
         \|G_2(x_1,y_1)-G_2(x_2,y_2)\|_{\mathcal{H}^m}\leq  C\|x_1-x_2\|_{\mathcal{H}}+L_{G_2}\|y_1-y_2\|_{\mathcal{H}}.
     \end{equation}	

\item [$(\mathbf{H3})$] Let $\{e_n\}_{n \in \mathbb N}$ be a complete orthonormal basis of $\mathcal{H}$, there exists a non-decreasing sequence of real positive numbers $\{ \lambda_n\}_{n \in \mathbb N}$ such that
		\begin{align*}
			A e_n=-\lambda_n e_n.
		\end{align*}

\item[$(\mathbf{H4})$] The smallest eigenvalue $\lambda_1$ of $-A$ and Lipschitz constants $L_{F_2},L_{G_2}$ satisfy
$$ \lambda_1- L_{F_2}-3 L_{G_2}^2 >0.$$

\item[$(\mathbf{H5})$] For any $\gamma \in (1/3, \gamma_0), 1/3<\gamma_0 \leq 1/2$ and $p\in [1,\infty)$ we have
$$\mathbb E\big[\interleave \mathbf{B} \interleave_\gamma^p \big] <\infty,$$
where the homogeneous norm $\interleave \mathbf{B} \interleave_\gamma:=|B|_\gamma +\sqrt{ |B^2|_{2\gamma} }$,
which, although not a norm in the usual sense of normed linear spaces, is a very adequate concept for rough paths.

	\end{enumerate}

\begin{remark}
\emph{(i)} The conditions $(\mathbf{H1})$ and $(\mathbf{H2})$ guarantee the existence of a unique solution to slow-fast RPDEs $(\ref{eq1})$ \emph{(cf. \cite[Theorem 2.3]{MG23})}.

\emph{(ii)}
The assumption $(\mathbf{H4})$ ensures the existence and uniqueness of the invariant probability measure and the exponential ergodicity for the frozen equation $(\text{see} ~(\ref{frozeneq})~\text{below})$ with fixed slow component $x \in \mathcal H$.

\emph{(iii)} The assumption $(\mathbf{H5})$ is mainly used to ensure that the mixed random rough path $\mathbf \Xi \in \mathscr{C}^\gamma ([0,T]; \mathbb R^{d+m}), \mathbb P$-a.s. for $\gamma \in (1/3, \gamma_0) ~(\text{see Lemma}~ \ref{mixrandom} ~\text{in Section}~ \ref{proba})$.

\end{remark}

Next, we give results on the existence and uniqueness of global (in time) solutions to system (\ref{eq1}) from the perspective of controlled rough paths.
To do this we introduce the following space:
$$\mathscr{D}_{S,\mathbf X}^{2\gamma,\beta,\eta}([0,T];\mathcal{H}_{\alpha})=\mathscr{D}_{S,\mathbf X}^{2\gamma}([0,T];\mathcal{H}_{\alpha})\cap \big(\hat{C}^{\eta}([0,T];\mathcal{H}_{\alpha+\beta})\times \mathcal L^{\infty}([0,T];\mathcal{H}_{\alpha+\beta}^d)\big),$$
where $\beta\in\mathbb{R}$ and $\eta\in[0,1]$.
We also write $\hat{C}^0=\mathcal L^\infty$ for $\eta =0$.
Let $(Y,Y')\in\mathscr{D}_{S,\mathbf X}^{2\gamma,\beta,\eta}([0,T];\mathcal{H}_\alpha)$, the seminorm of this space is defined as:
\begin{equation*}
\|Y,Y'\|_{X,2\gamma,\beta,\eta}=\|Y\|_{\eta,\alpha+\beta}+\|Y'\|_{\infty,\alpha+\beta}+\|Y,Y'\|_{X,2\gamma,\alpha}.
\end{equation*}
The norm of this space is given by
\begin{equation*}\label{norm}
\|Y,Y^\prime\|_{\mathscr{D}_{S,\mathbf X}^{2\gamma,\beta,\eta}}=\|Y_0\|_{\mathcal{H}_{\alpha +\beta}}+\|Y_0^\prime\|_{\mathcal{H}_\alpha^d}+\|Y,Y'\|_{X,2\gamma,\beta,\eta}.
\end{equation*}

For simplicity, we also introduce the following notation
$$\mathcal D_{\mathbf X}^{2\gamma, \eta}([0,T];\mathcal{H}_\alpha):=\mathscr{D}_{S,\mathbf X}^{2\gamma,2\gamma,\eta}([0,T];\mathcal{H}_{\alpha-2\gamma}),$$
and solve our equations in the space $\mathcal D_{\mathbf X}^{2\gamma, \eta}([0,T];\mathcal{H})~(\alpha=0)$.

\begin{remark}
For any $(Y,Y^\prime)\in \mathcal D_{\mathbf X}^{2\gamma, \eta}([0,T];\mathcal{H})$ we denote its norm by $\|Y, Y^\prime\|_{\mathcal D_{\mathbf X}^{2\gamma, \eta}}$ unless otherwise stated,
but in the proof of Theorem $\ref{MainRes}$ we sometimes write $\|Y,Y^\prime\|_{\mathcal{D}_{\mathbf X,[0,T]}^{2\gamma,\eta}}$ instead of $\|Y,Y^\prime\|_{\mathcal{D}_{\mathbf X}^{2\gamma,\eta}}$ to emphasise the time horizon.
\end{remark}

Since the coefficients of the fast component in system (\ref{eq1}) satisfy the global Lipschitz condition,
the proof of the existence and uniqueness is similar to that of \cite[Theorem 2.3]{MG23}.
We omit the details here.

\begin{lemma}\label{wellpose}
Assume that $(\mathbf{H1})$-$(\mathbf{H2})$ and $(\mathbf{H5})$ hold.
Then for any initial values $(x,y)\in \mathcal{H} \times \mathcal{H}$,
\emph{Eq.~(\ref{eq1})} has a unique global solution given by the controlled rough path $\big(X_t^\varepsilon,(X^\varepsilon)^\prime_t \big) \times \big( Y_t^\varepsilon,(Y^\varepsilon)^\prime_t \big) \in \mathcal D_{\mathbf B}^{2\gamma, \eta}([0,T];\mathcal{H}) \times  \mathcal D_{\mathbf W}^{2\gamma, \eta}([0,T];\mathcal{H})$ with $\eta \in [0,\gamma), \gamma \in (1/3, \gamma_0)$ and $1/3 < \gamma_0 \leq 1/2$ such that
\begin{equation*}
	\left\{ \begin{aligned}
		&\big(X_t^\varepsilon,(X^\varepsilon)^\prime_t \big)=\Big( S_t x+ \int_0^t S_{t-s} F_1 (X_s^\varepsilon,Y_s^\varepsilon)ds+ \int_0^t S_{t-s} G_1 (X_s^\varepsilon) d\mathbf{B}_s, G_1 (X_t^\varepsilon) \Big), \\
		& \big(Y_t^\varepsilon, (Y^\varepsilon)^\prime_t \big)= \Big( S_{t/\varepsilon} y+ \frac 1\varepsilon \int_0^t S_{(t-s)/{\varepsilon}} F_2 (X_s^\varepsilon,Y_s^\varepsilon)ds +\frac 1{\sqrt{\varepsilon}}\int_0^t S_{(t-s)/{\varepsilon}} G_2 (X_s^\varepsilon,Y_s^\varepsilon) d\mathbf{W}_s,  G_2 (X_t^\varepsilon,Y_t^\varepsilon) \Big).
	\end{aligned}  \right.
\end{equation*}
\end{lemma}

We would now like to present the main result of this paper.
\begin{theorem}\label{MainRes}
	Assume that $(\mathbf{H1})$-$(\mathbf{H5})$ hold. Then for each initial values $x\in \mathcal{H}_\eta$ and $y \in \mathcal{H}$ with $\eta \in (\gamma-1/4,\gamma) $, we have
	\begin{equation}\label{Averaging}
		\lim_{\varepsilon\rightarrow 0}  \mathbb E\big[\|X^\varepsilon-\bar{X}\|_{\eta, 0}^{2}\big]=0,
	\end{equation}
	where $\bar{X}_t$ is the unique solution of the corresponding averaged equation
	\begin{equation*}
		\left\{ \begin{aligned}
			d\bar{X}_t&=\big[ A \bar{X}_t+\bar{F}_1 (\bar{X}_t)\big]dt+G_1(\bar{X}_t)d \mathbf B_t,\\
			\bar{X}_0&=x,
		\end{aligned} \right.
	\end{equation*}
	with the averaged coefficient $\bar{F}_1 (x):=\displaystyle \int_{\mathcal{H}} F_1 (x,y) \mu^x (dy)$ for $x\in\mathcal{H}$, where $\mu^x$ is the unique invariant measure with respect to the Markov semigroup of frozen equation
 \begin{equation}\nonumber
\left\{ \begin{aligned}
	d Y_t&=\big[A Y_t+F_2(x,Y_t)\big]dt+G_2(x,Y_t) d \widetilde{W}_t,\\
	Y_0&= y.
\end{aligned} \right.
\end{equation}
\end{theorem}

\section{Slow-fast rough partial differential equations}\label{sect3}

In this section, we provide definitions of slow-fast system (\ref{eq1}) from both deterministic and probabilistic perspectives.
Moreover, we also show that the fast component $Y^\varepsilon$ driven by the Brownian rough path $\mathbf W$ in system (\ref{eq1}) is consistent with an It\^{o} SPDE.

\subsection{Deterministic perspective} \label{sect3.1}

Let $\mathscr H:=\mathcal{H} \times \mathcal{H}$ be the product Hilbert space. For any $u=(u_1,u_2),v=(v_1,v_2)\in\mathscr{H}$, we denote the scalar product and the induced norm by
$$\langle u,v \rangle_{\mathscr{H}}=\langle u_1,v_1\rangle_{\mathcal{H}}+\langle u_2,
v_2\rangle_{\mathcal{H}},~~~~\|u\|_{\mathscr{H}}=\sqrt{\langle u,u\rangle_{\mathscr{H}}}=\sqrt{\|u_1\|_{\mathcal{H}}^2+\|u_2\|_{\mathcal{H}}^2}.$$
Let $Z^\varepsilon=(X^\varepsilon, Y^\varepsilon)^\top \in \mathscr H$, we will rewrite system (\ref{eq1}) (in a deterministic sense) as
\begin{equation} \label{eq2}
Z_t^\varepsilon= \mathcal S_t Z_0^\varepsilon+\int_0^t \mathcal{S}_{t-s} \mathcal F(Z_s^\varepsilon) ds+\int_0^t \mathcal{S}_{t-s} \mathcal G(Z_s^\varepsilon) d \mathbf \Xi_s,
~~~~Z_0^\varepsilon=( x, y)^\top,
\end{equation}
where $\mathbf \Xi=(\Xi, \Xi^2)$ is a rough path over $\mathbb R^{d+m}:=\mathbb R^d \oplus \mathbb R^m$ and goes as follows
\begin{equation} \label{mixRP}
\Xi_t=\begin{pmatrix} B_t \\ W_t  \end{pmatrix},   ~~~~\Xi^2_{t,s}=\begin{pmatrix} B^2_{t,s} &\displaystyle I[B,W]_{t,s}:=\int_s^t \delta B_{r,s} \otimes d W_r   \\ \displaystyle I[W,B]_{t,s}:=\int_s^t \delta W_{r,s} \otimes d B_r  & W^2_{t,s} \end{pmatrix},
\end{equation}
and
\begin{align*}
\mathcal S_t&:=diag \big( S_t ,S_{t/\varepsilon} \big),\\
\mathcal F(Z_s^\varepsilon)&:= \Big( F_1(X_s^\varepsilon, Y_s^\varepsilon), \frac 1 \varepsilon  F_2 (X_s^\varepsilon, Y_s^\varepsilon)\Big)^\top, \\
\mathcal G(Z_s^\varepsilon)&:= diag \Big( G_1(X_s^\varepsilon), \frac 1 {\sqrt{\varepsilon}}  G_2 (X_s^\varepsilon, Y_s^\varepsilon)\Big).
\end{align*}

Under assumptions $(\mathbf{H1})$ and $(\mathbf{H2})$,
it can be inferred by the concatenation procedure that Eq. (\ref{eq2}) has a unique global solution given by $(Z^\varepsilon, \mathcal G(Z^\varepsilon)) \in \mathcal D_{\mathbf \Xi}^{2\gamma ,\eta}([0,T];\mathscr{H}), \eta \in [0,\gamma)$,
which is controlled by $\mathbf \Xi=(\Xi, \Xi^2) \in \mathscr{C}^\gamma ([0,T]; \mathbb R^{d+m})$ for some $\gamma \in (1/3, 1/2)$ (cf. \cite[Theorem 2.3]{MG23}).

Next, we will verify that if $Z^\varepsilon$ solves Eq. (\ref{eq2}) on $\mathscr{H}$, then the component $X^\varepsilon$ solves a RPDE driven by rough path $\mathbf B$ with parameter $Y^\varepsilon$.

\begin{lemma}
For each $T>0$, let $(Z^\varepsilon, \mathcal G(Z^\varepsilon)) \in \mathcal D_{\mathbf \Xi}^{2\gamma ,\eta}([0,T];\mathscr{H})$ be the unique solution of Eq. \emph{(\ref{eq2})} on $[0, T]$,
then $(X^\varepsilon, G_1(X^\varepsilon)) \in \mathcal D_{\mathbf B}^{2\gamma ,\eta}([0,T];\mathcal{H})$ is the unique solution of the following RPDE
\begin{equation*}
X_t^\varepsilon=S_t x+\int_0^t S_{t-s} F_1 (X_s^\varepsilon,Y_s^\varepsilon)ds+ \int_0^t S_{t-s} G_1 (X_s^\varepsilon) d\mathbf{B}_s, ~~~~(X^\varepsilon)'_t=G_1 (X_t^\varepsilon).
\end{equation*}
\end{lemma}
\begin{proof}
According to $(Z^\varepsilon, \mathcal G(Z^\varepsilon)) \in \mathcal D_{\mathbf \Xi}^{2\gamma ,\eta}([0,T];\mathscr{H})$, we have $(Z^\varepsilon, \mathcal G(Z^\varepsilon)) \in \mathscr D_{\mathcal S, \mathbf \Xi}^{2\gamma}([0,T];\mathscr{H}_{-2 \gamma})$, $\mathscr{H}_{-2 \gamma}:=\mathcal H_{-2\gamma} \times \mathcal H_{-2\gamma}$.
By Definition \ref{CRP}, we have the remainder term
\begin{align*}
R_{t,s}^{Z^\varepsilon}&=\hat{\delta}Z_{t,s}^\varepsilon-\mathcal S_{t-s} \mathcal G(Z_s^\varepsilon) \delta \Xi_{t,s} = \Bigg( \begin{aligned}
&\hat{\delta} X_{t,s}^\varepsilon \\
&\hat{\delta} Y_{t,s}^\varepsilon
\end{aligned}
\Bigg)
-\Bigg ( \begin{aligned}
S_{t-s} G_1(X_s^\varepsilon) &~~~~~~~~~~~~0 \\
0~~~~~~~~~& \displaystyle \frac 1 {\sqrt{\varepsilon}} S_{(t-s)/\varepsilon} G_2(X_s^\varepsilon, Y_s^\varepsilon)
\end{aligned}
\Bigg)
\Bigg ( \begin{aligned}
&\delta B_{t,s} \\
&\delta W_{t,s}
\end{aligned}
\Bigg)
\end{align*}
as a two-parameter function on $\mathscr{H}_{-2\gamma}$ and is $2\gamma$-H\"{o}lder continuous.
So is its component $\hat{\delta} X_{t,s}^\varepsilon-S_{t-s} G_1(X_s^\varepsilon)\delta B_{t,s}$.

Next, it is sufficient to show that the first component of rough integral $\displaystyle\int_0^t \mathcal{S}_{t-s} \mathcal G(Z_s^\varepsilon) d \mathbf \Xi_s $ is equal to $ \displaystyle\int_0^t S_{t-s} G_1 (X_s^\varepsilon) d\mathbf{B}_s $.
Notice that
\begin{align}\label{ZXInt}
&\mathcal S_{t-s} \big(\mathcal G(Z_s^\varepsilon)\delta \Xi_{t,s}+\{\mathcal G(Z^\varepsilon)\}'_s \Xi_{t,s}^2  \big)  \nonumber\\
= & \begin{pmatrix} S_{t-s} G_1(X_s^\varepsilon) \delta B_{t,s}+ S_{t-s} D_x G_1(X_s^\varepsilon)G_1(X_s^\varepsilon) B_{t,s}^2\\
\displaystyle \frac 1 {\sqrt{\varepsilon}} S_{(t-s)/\varepsilon} G_2(X_s^\varepsilon, Y_s^\varepsilon) \delta W_{t,s}
+\frac 1\varepsilon  S_{(t-s)/\varepsilon} D_y G_2(X_s^\varepsilon, Y_s^\varepsilon) G_2(X_s^\varepsilon, Y_s^\varepsilon)W_{t,s}^2
+M_{t,s} \end{pmatrix},
\end{align}
where $ M_{t,s}:= \displaystyle \frac 1 {\sqrt{\varepsilon}} S_{(t-s)/\varepsilon} D_x G_2(X_s^\varepsilon, Y_s^\varepsilon)G_1(X_s^\varepsilon) I[B,W]_{t,s}$, $D_x$ and $ D_y $ represent partial gradient operators w.r.t. $x$ and $y$, respectively.

According to the definition of rough convolution in Theorem \ref{RoughInt}.
The left hand side of (\ref{ZXInt}) is a summand in the modified Riemann sum that approximates the rough integral $\displaystyle\int_0^t \mathcal{S}_{t-s} \mathcal G(Z_s^\varepsilon) d \mathbf \Xi_s$.
The first component of the right hand side of (\ref{ZXInt}) is a summand in the Riemann sum that approximates integral $\displaystyle\int_0^t S_{t-s} G_1 (X_s^\varepsilon) d\mathbf{B}_s$.
Thus the lemma is proven.
\hfill $\square$

\end{proof}

\subsection{Probabilistic perspective} \label{proba}

All our analyses so far have been purely deterministic.
From now on, however, we will work in a stochastic setting,
since our main goal is to prove the stochastic averaging principle for (\ref{eq1}).
Therefore, a probability space $(\Omega, \mathscr F, \mathbb P)$ is given.
Let $\frac 13 < \gamma_0  \leq \frac 12$,
$\mathbf B= ( B, B^2) $ be a $\mathscr{C}^\gamma([0,T];\mathbb{R}^d)$-valued random variable (i.e. random rough path) defined on $(\Omega,\mathscr F,\mathbb P)$ for every $\gamma \in(1/3,\gamma_0)$,
and let $W=\{W_t\}_{t \in [0,T]}$ be a standard $m$-dimensional BM on the probability space $(\Omega, \mathscr F, \mathbb P)$.
We also assume that $\mathbf B$ and $W$ are independent.

Below we give the definition of the random rough path $\mathbf \Xi$, where components of $\mathbf \Xi$ are shown in (\ref{mixRP}).

\begin{definition}\label{Xi}
Let $\mathbf B= ( B, B^2) $ be a $\mathscr{C}^\gamma([0,T];\mathbb{R}^d)$-valued random variable for $\gamma \in (1/3, \gamma_0)$ and $\{W_t\}_{t\in [0,T]}$ is a $m$-dimensional BM.
Define the ``enhancement" of $W$ by
$$W_{t,s}^2:= \int_s^t \delta W_{r,s} \otimes dW_r,  ~~~~ \delta W_{r,s}=W_r-W_s,$$
where the stochastic integration is understood in the sense of It\^{o} integral.
Sometimes we indicate this by writing $W_{t,s}^{It\hat{o}}$.
Then $\mathbf W=(W,W^2)$ is an It\^{o}-type Brownian rough path
and for any $\gamma \in (1/3, 1/2)$, $\mathbf W=(W,W^2) \in \mathscr{C}^\gamma ([0,T]; \mathbb R^m)$, $\mathbb P$-a.s..
Moreover, we define
$$I[B,W]_{t,s}=\int_s^t \delta B_{r,s} \otimes dW_r$$
as an It\^{o} stochastic integral, and then it holds that
$$I[W,B]_{t,s}=\int_s^t \delta W_{r,s} \otimes dB_r=\delta W_{t,s} \otimes \delta B_{t,s} -\int_s^t dW_r \otimes \delta B_{r,s} $$
by imposing integration by parts.
By now we have set all components of the random rough path $\mathbf \Xi=(\Xi, \Xi^2)$.
\end{definition}

Define
$$\bar{\mathscr F}_t:=\sigma \Big\{ \big( ( \delta B_{r,s}, B^2_{r,s}),( \delta W_{r,s}, W^2_{r,s}) \big) : 0\leq s\leq r\leq t \Big \}=\sigma \big\{(B_r,W_r): 0 \leq r \leq  t \big\} ,$$
and denote by $\{ \mathscr F_t \}_{t \in [0,T]}$ the larger augmented filtration of $\{ \bar{\mathscr F}_t \}_{t \in [0,T]}$ and which satisfies the usual conditions. Then we have

(i) $\{W_t\}_{t\in [0,T]}$ is a $m$-dimensional BM w.r.t. the filtration $\{\mathscr F_t \}_{t\in [0,T]} $;

(ii) $\big\{ (\delta B_{t,0}, B_{t,0}^2 ) \big\} _{t\in [0,T]}$ is adapted to the filtration $\{ \mathscr F_t \}_{t\in [0,T]} $.

By Kolmogorov's continuity criterion (see \cite[Theorem 3.1]{FH14}) we obtain the following lemma, and since its proof is similar to  \cite[Lemma 4.6]{In22} or \cite[Lemma 4.3]{YJYM24},
the details of the proof will be omitted.

\begin{lemma} \label{mixrandom}
Assume that condition $(\mathbf{H5})$ holds. For each $\gamma \in (1/3, \gamma_0)$ and $\mathbf \Xi$ introduced in Definition $\ref{Xi}$, we have $\mathbf \Xi(\omega)=\mathbf \Xi=(\Xi, \Xi^2) \in \mathscr{C}^\gamma ([0,T]; \mathbb R^{d+m})$, $\mathbb P$-a.s..
Moreover, we have $\mathbb E \big[\interleave \mathbf \Xi \interleave_\gamma^p \big] <\infty$ for any $p \in [1,\infty)$.
\end{lemma}

Next we want to show that the component $Y^\varepsilon$ defined against the Brownian rough path $\mathbf W=(W,W^2)$ in the slow-fast system (\ref{eq1}) is consistent with a SPDE in the sense of It\^{o}.
To do this, we introduce the space of controlled rough paths that are allowed to explode in finite time.
Define a new Banach space $\bar{V}=V \cup \{\infty\}$ for any Banach space $V$.
The topology on the space $\bar{V}$ is induced by the basis containing open balls of $V$
and the sets of the form $\{v\in V:\|v\|_V\geq N\} \cup \{\infty\}$ for every $N>0$.
Then we define the space of controlled rough paths that can blow up in finite time as follows:
\begin{align*}
\hat{\mathscr{D}}_{S,\mathbf X}^{2\gamma,\beta,\eta}(\mathbb{R}_+;\mathcal{H}_{\alpha})=&\big\{(Y,Y^{\prime})\in \hat{C}^\eta (\mathbb{R}_{+};\bar{\mathcal{H}}_{\alpha+\beta})\times C(\mathbb{R}_{+};\bar{\mathcal{H}}_{\alpha+\beta}^d):\exists ~\tau>0 ~\text{such that}\\
&(Y,Y^{\prime})\restriction_{[0,\tau)} \in \mathscr{D}_{S,\mathbf X}^{2\gamma,\beta,\eta}([0,\tau);\mathcal{H}_{\alpha})~~\text{and}~~(Y_t,Y_t^\prime)=(\infty,\infty),~\forall t\geq\tau \big\},
\end{align*}
where controlled rough paths are controlled by the rough path $\mathbf X$ according to the semigroup $S$, the $\tau$ denotes a blow up time of $(Y,Y^\prime)$ and can be taken $\infty$ if controlled rough paths have finite $\mathscr{D}_{S,\mathbf X}^{2\gamma,\beta,\eta}$ norm on each compact interval.

Let $N>0$, we define the following stopping time
$$\tau_N^\varepsilon:=\inf \big\{ t \geq 0: \|Z_t^\varepsilon\|_{\mathscr H} \geq N \big\}.$$

\begin{lemma}\label{RSPDE}
Assume that $(\mathbf{H5})$ holds. Let $\{W_t\}_{t \geq 0} $ be a $m$-dimensional BM defined on the filtered probability space $(\Omega, \mathscr F, \{\mathscr F_t\}_{t\geq 0},\mathbb P)$,
and let $(Z^\varepsilon,\mathcal G(Z^\varepsilon)) \in \mathcal D_{\mathbf \Xi}^{2\gamma ,\eta}([0,T];\mathscr{H})$ with $\eta \in [0,\gamma)$ and $\gamma \in (1/3, \gamma_0)$ which is the unique solution to Eq. $(\ref{eq2})$
and is adapted to the filtration $\{ \mathscr{F}_t\}_{t \geq 0}$ when viewed as an element of $\hat{\mathscr D}_{\mathcal S,\mathbf \Xi}^{2\gamma ,2\gamma, \eta}(\mathbb R_+;\mathscr{H}_{-2\gamma})$.
Then for each fixed $\varepsilon \in (0,1]$, we have
that the component $Y^\varepsilon$ of $Z^\varepsilon$ satisfies the following It\^{o} SPDE on $[0,T]$,
\begin{equation}\label{eqY}
Y_t^\varepsilon= S_{t/\varepsilon} y+ \frac 1\varepsilon \int_0^{t} S_{(t-s)/\varepsilon} F_2(X_s^\varepsilon, Y_s^\varepsilon)ds
+\frac 1{\sqrt{\varepsilon}} \int_0^{t} S_{(t-s)/\varepsilon} G_2(X_s^\varepsilon, Y_s^\varepsilon) dW_s.
\end{equation}
\end{lemma}

\begin{proof}
Note that the map
$$W \!\restriction_{[0,t]} \mapsto (W, W^2)\!\restriction_{[0,t]} \in \mathscr{C}^\gamma ([0,t]; \mathbb R^m) $$
is measurable.
For every $t < T$, the solution $(Y^\varepsilon,G_2(X^\varepsilon, Y^\varepsilon)) \in \hat{\mathscr D}_{S,\mathbf W}^{2\gamma ,2\gamma, \eta}([0,t];\mathcal{H}_{-2\gamma})$ to the component $Y^\varepsilon$ of $Z^\varepsilon$ is a continuous image of the noise $(W,W^2)\!\restriction_{[0,t]}$ (see \cite[Theorem 4.5]{GH19}).
Considering $(Y_t^\varepsilon(\omega),G_2(X_t^\varepsilon(\omega),Y_t^\varepsilon(\omega)))$ as an element of $ \hat{\mathscr D}_{S,\mathbf W}^{2\gamma ,2\gamma, \eta}(\mathbb R_+;\mathcal{H}_{-2\gamma})$
we deduce that is adapted to the filtration $\{ \mathscr F_t \}_{t \geq 0}$.
Therefore, for the proof of our result, it is sufficient to show that the second component of the random rough integral $\displaystyle \int_0^t \mathcal{S}_{t-s} \mathcal G(Z_s^\varepsilon) d \mathbf \Xi_s$ is equivalent to an It\^{o} integral.

For $t>0$, let
\begin{equation}\label{ZL}
V_t^N= \frac 1{\sqrt{\varepsilon}} \int_0^{t\wedge \tau_N^\varepsilon} S_{(t-s)/\varepsilon} G_2(X_s^\varepsilon, Y_s^\varepsilon) dW_s.
\end{equation}
According to the definition of stopping time $\tau_N^\varepsilon$ and the local boundedness of $G_1$ and $G_2$,
there exists a constant $\hat{N}>0$ such that almost surely for $t < \tau_N^\varepsilon$, one has
\begin{align} \label{G_2Bound}
&\|G_2(X_t^\varepsilon, Y_t^\varepsilon)\|_{\mathcal H^m}
+\|D_x G_2(X_t^\varepsilon, Y_t^\varepsilon)G_1(X_t^\varepsilon)\|_{\mathcal H^{d \times m}} +\|D_y G_2(X_t^\varepsilon, Y_t^\varepsilon)G_2(X_t^\varepsilon, Y_t^\varepsilon) \|_{\mathcal H^{m \times m}} <\hat{N}.
\end{align}
Therefore, It\^{o} integral (\ref{ZL}) exists because $\displaystyle \int_0^t\|G_2(X_s^\varepsilon, Y_s^\varepsilon)\|_{\mathcal H^m}^2 1_{\{s < \tau_N^\varepsilon\}} ds \leq \hat{N}^2 t$, and since $\tau_N^\varepsilon$ is a stopping time, the integrand is adapted to the filtration $\mathscr F_t$.

For all $n,k \in \mathbb{N}$, let $\mathcal{P}_{n}=\{s_{k}^{n}\}_{k=0}^{\infty}$ be a sequence of increasing countable subsets of $\mathbb R_+$ such that $\bigcup_n \mathcal{P}_{n}$ is dense in $\mathbb{R}_+$ and $s_k^n<s_{k+1}^n$.
We pick the partition sequence $\pi_n=\big \{[s_k^n,s_{k+1}^n]:s_k^n\in\mathcal{P}_n,k\in\mathbb{N} \big\}$ that is formed by $\mathcal{P}_n$
and we define $|\pi_n|=\displaystyle \max_{k\geq1}\{|s_{k+1}^n-s_k^n|\}$ as the size of mesh.
Then by the definition of It\^{o} stochastic integration, $V_t^N$ is defined as the limit in probability as follows.

\begin{equation}\label{ZL1}
V_t^N= \lim_{n\to\infty}\sum_{\substack{[u,v]\in\pi_{n}\\u<t}} \frac 1{\sqrt{\varepsilon}} S_{(t-u)/\varepsilon} G_2(X_u^\varepsilon, Y_u^\varepsilon) \delta W_{v,u}1_{\{u<\tau_N^\varepsilon\}}.
\end{equation}
We can extract a subsequence of partitions (which we still denote by $\pi_{n}$ for simplicity) such that the above limit holds almost surely.

On the other hand, since $(Z^\varepsilon,\mathcal G(Z^\varepsilon)) \in \hat{\mathscr{D}}_{\mathcal S,\mathbf \Xi}^{2\gamma,2\gamma,\eta}(\mathbb{R}_+;\mathscr{H}_{-2 \gamma})$,
we conclude from Lemma \ref{funCRP} that
$(\mathcal G(Z^\varepsilon),\mathcal G(Z^\varepsilon)^\prime ) \in \hat{\mathscr{D}}_{\mathcal S,\mathbf \Xi}^{2\gamma,2\gamma,0}(\mathbb{R}_+;\mathscr{H}_{-2\gamma}^{d+m})$,
which implies that the rough integral
$$\int_{0}^{t\wedge \tau_N^\varepsilon(\omega)} \mathcal S_{t-s} \mathcal G (Z_s^\varepsilon) d \mathbf \Xi_s (\omega) $$
exists and one can verify that
$$
\tilde{V}_t^N(\omega):=\displaystyle \frac 1{\sqrt{\varepsilon}} \int_0^{t\wedge \tau_N^\varepsilon(\omega)}  S_{(t-s)/\varepsilon} G_2 (X_s^\varepsilon,Y_s^\varepsilon) d \mathbf W_s (\omega)
$$
is equal to (see (\ref{ZXInt})):
\begin{align}\label{ZL2}
 \tilde{V}_t^N=&\lim_{n\to\infty}\sum_{\substack{[u,v]\in\pi_{n}\\u<t\wedge \tau_N^\varepsilon}} \Big( \frac 1{\sqrt{\varepsilon}} S_{(t-u)/\varepsilon} G_2(X_u^\varepsilon, Y_u^\varepsilon) \delta W_{v,u}
+\frac 1{\sqrt{\varepsilon}} S_{(t-u)/\varepsilon} D_x G_2(X_u^\varepsilon, Y_u^\varepsilon) G_1(X_u^\varepsilon) I[B,W]_{v,u}  \nonumber\\
&~~~~~~~~~~~~~~~~~~+\frac 1{\varepsilon} S_{(t-u)/\varepsilon} D_y G_2(X_u^\varepsilon, Y_u^\varepsilon) G_2(X_u^\varepsilon, Y_u^\varepsilon) W_{v,u}^2 \Big).
\end{align}
By (\ref{ZL1}) and (\ref{ZL2}), it is easy to obtain that the $L^2(\Omega)$ norm of the difference between $V_t^N$ and $\tilde{V}_t^N$ is
\begin{align}\label{Zdif}
\big| V_t^N-\tilde{V}_t^N \big|_{L^2(\Omega)}=&\Bigg| \lim_{n\to\infty}\sum_{\substack{[u,v]\in\pi_{n}\\u<t\wedge \tau_N^\varepsilon}}
\Big(\frac 1{\sqrt{\varepsilon}}
S_{(t-u)/\varepsilon} D_x G_2(X_u^\varepsilon, Y_u^\varepsilon) G_1(X_u^\varepsilon) I[B,W]_{v,u}  \nonumber\\
&+\frac 1{\varepsilon} S_{(t-u)/\varepsilon} D_y G_2(X_u^\varepsilon, Y_u^\varepsilon) G_2(X_u^\varepsilon, Y_u^\varepsilon) W_{v,u}^2 \Big) \Bigg|_{L^2(\Omega)}, ~~~~\text{for any}~t>0.
\end{align}

Next, we will verify that the right hand side of (\ref{Zdif}) is equal to zero.
Firstly, we prove that
$$\Bigg| \lim_{n\to\infty}\sum_{\substack{[u,v]\in\pi_{n}\\u<t\wedge \tau_N^\varepsilon}}
S_{(t-u)/\varepsilon} D_y G_2(X_u^\varepsilon, Y_u^\varepsilon)  G_2(X_u^\varepsilon, Y_u^\varepsilon) W_{v,u}^2 \Bigg|_{L^2(\Omega)}^2=0.$$
Define a (discrete time) martingale $M^n$ started at $M_{0}^{n}=0$ and with increments
$M_{k+1}^{n}-M_{k}^{n}=S_{(t-s_k^n)/\varepsilon} D_y G_2(X_{s_k^n}^\varepsilon, Y_{s_k^n}^\varepsilon)  G_2(X_{s_k^n}^\varepsilon, Y_{s_k^n}^\varepsilon) 1_{\{s_k^n < t\wedge \tau_N^\varepsilon\} } W_{s_{k+1}^n,s_k^n}^2$.
Then by (\ref{G_2Bound}) we have
$$
\begin{aligned}
&\Bigg|\sum_{\substack{[u,v]\in\pi_n \\u<t\wedge \tau_N^\varepsilon}} S_{(t-u)/\varepsilon} D_y G_2(X_u^\varepsilon, Y_u^\varepsilon) G_2(X_u^\varepsilon, Y_u^\varepsilon) W_{v,u}^2 \Bigg|_{L^{2}(\Omega)}^{2}  \\
=&\Bigg|\sum_{k=0}^{\infty}(M_{k+1}^n-M_k^n) \Bigg|_{L^2(\Omega)}^2
=\sum_{k=0}^{\infty} \big|M_{k+1}^{n}-M_{k}^{n} \big|_{L^{2}(\Omega)}^{2}
\lesssim \hat{N}^2 \sum_{k=0}^{\infty} \big|W_{s_{k+1}^n,s_k^n}^2 1_{\{s_k^n<t\}} \big|_{L^2(\Omega)}^2
\lesssim \hat{N}^2 t|\pi_{n}|,
\end{aligned}
$$
where the last inequality is because the Brownian scaling gives $\big| W _{s_{k+1}^n,s_k^n}^2 \big| _{L^2( \Omega ) }^2 \lesssim | s_{k+1}^n- s_k^n |^2$.
We also use the fact that all the above infinite sums are finite
due to the presence of the indicator function.

Since $\bigcup _{n}\mathcal{P} _{n}\text{ is dense in } \mathbb R_+$ we have $| \pi _{n}| \to 0$ as $n\to \infty $.
Therefore, Fatou's lemma implies that
\begin{align}\label{W0}
&\Bigg| \lim_{n\to\infty}\sum_{\substack{[u,v]\in\pi_{n}\\u<t\wedge \tau_N^\varepsilon}}
S_{(t-u)/\varepsilon} D_y G_2(X_u^\varepsilon, Y_u^\varepsilon)  G_2(X_u^\varepsilon, Y_u^\varepsilon) W_{v,u}^2 \Bigg|_{L^2(\Omega)}^2 \nonumber\\
\leq & \lim_{n\to\infty} \Bigg| \sum_{\substack{[u,v]\in\pi_{n}\\u<t\wedge \tau_N^\varepsilon}}
S_{(t-u)/\varepsilon} D_y G_2(X_u^\varepsilon, Y_u^\varepsilon)  G_2(X_u^\varepsilon, Y_u^\varepsilon) W_{v,u}^2 \Bigg|_{L^2(\Omega)}^2 \lesssim  \lim_{n\to\infty} \hat{N}^2 t |\pi _{n}|=0.
\end{align}

Secondly, we prove that
$$\Bigg| \lim_{n\to\infty}\sum_{\substack{[u,v]\in\pi_{n}\\u<t\wedge \tau_N^\varepsilon}}
S_{(t-u)/\varepsilon} D_x G_2(X_u^\varepsilon, Y_u^\varepsilon)  G_1(X_u^\varepsilon) I[B,W]_{v,u} \Bigg|_{L^2(\Omega)}^2=0.$$
Define a new (discrete time) martingale $\tilde{M}^n$ started at $\tilde{M}_0^n=0$
and increments
$\tilde{M}_{k+1}^{n}-\tilde{M}_{k}^{n}=S_{(t-s_k^n)/\varepsilon} D_x G_2(X_{s_k^n}^\varepsilon, Y_{s_k^n}^\varepsilon)  G_1(X_{s_k^n}^\varepsilon) 1_{\{s_k^n < t\wedge \tau_N^\varepsilon\} } I[B,W]_{s_{k+1}^n,s_k^n}$.
Similarly, using (\ref{G_2Bound}) one has the bound
$$
\begin{aligned}
&\Bigg|\sum_{\substack{[u,v]\in\pi_n \\u<t\wedge \tau_N^\varepsilon}} S_{(t-u)/\varepsilon} D_x G_2(X_u^\varepsilon, Y_u^\varepsilon) G_1(X_u^\varepsilon) I[B,W]_{v,u} \Bigg|_{L^{2}(\Omega)}^{2}  \\
= & \Bigg|\sum_{k=0}^{\infty}(\tilde{M}_{k+1}^n-\tilde{M}_k^n) \Bigg|_{L^2(\Omega)}^2
=\sum_{k=0}^{\infty}\big |\tilde{M}_{k+1}^{n}-\tilde{M}_{k}^{n} \big|_{L^{2}(\Omega)}^{2}\\
\lesssim & \hat{N}^2 \sum_{k=0}^{\infty} \big|I[B,W]_{s_{k+1}^n,s_k^n} 1_{\{s_k^n<t\}} \big|_{L^2(\Omega)}^2
\lesssim \hat{N}^2 t|\pi_{n}|^{2\gamma} \mathbb E [\interleave \mathbf{B} \interleave_\gamma^2 ],
\end{aligned}
$$
where the fact that we have used is the following
\begin{align*}
\big|I[B,W]_{s_{k+1}^n,s_k^n}  \big|_{L^2(\Omega)}^2&=\mathbb E \Big (\int_{s_k^n}^{s_{k+1}^n} (B_u-B_{s_k^n}) \otimes d W_u \Big)^2 = \mathbb E \int_{s_k^n}^{s_{k+1}^n} (B_u-B_{s_k^n})^2  d u  \\
&\leq  \int_{s_k^n}^{s_{k+1}^n} \mathbb E [\interleave \mathbf{B} \interleave_\gamma^2] (u-s_k^n)^{2\gamma} du \lesssim  |s_{k+1}^n-s_k^n|^{1+2 \gamma} \mathbb E [\interleave \mathbf{B} \interleave_\gamma^2].
\end{align*}
By condition $(\mathbf{H5})$ we know that $\mathbb E [\interleave \mathbf{B} \interleave_\gamma^2] <\infty$, it is straightforward that
\begin{align}\label{BW0}
&\Bigg| \lim_{n\to\infty}\sum_{\substack{[u,v]\in\pi_{n}\\u<t\wedge \tau_N^\varepsilon}}
S_{(t-u)/\varepsilon} D_x G_2(X_u^\varepsilon, Y_u^\varepsilon)  G_1(X_u^\varepsilon) I[B,W]_{v,u} \Bigg|_{L^2(\Omega)}^2 \nonumber\\
\leq & \lim_{n\to\infty} \Bigg| \sum_{\substack{[u,v]\in\pi_{n}\\u<t\wedge \tau_N^\varepsilon}}
S_{(t-u)/\varepsilon} D_x G_2(X_u^\varepsilon, Y_u^\varepsilon)  G_1(X_u^\varepsilon) I[B,W]_{v,u}  \Bigg|_{L^2(\Omega)}^2 \nonumber\\
\lesssim & \lim_{n\to\infty} \hat{N}^2 t|\pi_{n}|^{2\gamma} \mathbb E [\interleave \mathbf{B} \interleave_\gamma^2]=0.
\end{align}

Finally, combining (\ref{W0}) and (\ref{BW0}) we have $V_t^N=\tilde{V}_t^N$, $\mathbb P$-a.s. for all $t>0$.
Now one can choose a continuous version of It\^{o} integral $V_t^N$ which is still equal almost surely to $\tilde{V}_t^N$ for every $t>0$.
Therefore, we have proven that for any $t(\omega)< \tau_N^\varepsilon$
\begin{equation*}
Y_t^\varepsilon= S_{t/\varepsilon} y+ \frac 1\varepsilon \int_0^{t\wedge \tau_N^\varepsilon} S_{(t-s)/\varepsilon} F_2(X_s^\varepsilon, Y_s^\varepsilon)ds
+\frac 1{\sqrt{\varepsilon}} \int_0^{t\wedge \tau_N^\varepsilon} S_{(t-s)/\varepsilon} G_2(X_s^\varepsilon, Y_s^\varepsilon) dW_s, ~~~~\mathbb P\text{-a.s.}.
\end{equation*}

We know that there is a unique global solution $(Z^\varepsilon,\mathcal G(Z^\varepsilon)) \in \mathcal D_{\mathbf \Xi}^{2\gamma ,\eta}([0,T];\mathscr{H})$ to Eq. (\ref{eq2}) on the time interval $[0,T], 0<T<\infty$, which implies that Eq. (\ref{eq2}) will not blow up in finite time.
More precisely, $\tau_N^\varepsilon \to \infty , ~\mathbb P$-a.s. as $N\to\infty$.
Therefore, whenever $\|Z_t^\varepsilon\|_{\mathscr H}$ is finite we can always restart the Eq. (\ref{eq2}) with initial condition $Z_t^\varepsilon$ and extend the solution further in time.
Moreover, $\tau_N^\varepsilon$ clearly increases as $N$ increases,
thus showing that the component $(Y^\varepsilon, G_2(X^\varepsilon, Y^\varepsilon)) \in \mathcal D_{\mathbf W}^{2\gamma ,\eta}([0,T];\mathcal{H})$ of $Z^\varepsilon$ is the solution to the following It\^{o} SPDE on $[0,T]$,
$$Y_t^\varepsilon= S_{t/\varepsilon} y+ \frac 1\varepsilon \int_0^{t} S_{(t-s)/\varepsilon} F_2(X_s^\varepsilon, Y_s^\varepsilon)ds
+\frac 1{\sqrt{\varepsilon}} \int_0^{t} S_{(t-s)/\varepsilon} G_2(X_s^\varepsilon, Y_s^\varepsilon) dW_s.$$
The proof is complete.
\hfill $\square$
\end{proof}

\section{Proof of averaging principle}\label{Proof}

In this section, we intend to prove the main result, namely Theorem \ref{MainRes}.
Some necessary a priori estimates are given before a concrete proof is given.
Furthermore, inspired by the work \cite{K68} of Khasminskii, the time discretization scheme is used in the proof of the averaging principle.
For this purpose, we construct an auxiliary process $\hat{Y}^\varepsilon$ and give an estimate of the error between the solution $Y^\varepsilon$ to Eq. (\ref{eqY}) and the auxiliary process $\hat{Y}^\varepsilon$, which plays a key role in the main proof.
Throughout this section, as before, $1/3< \gamma_0 \leq 1/2$.

\subsection{Some a priori estimates}\label{XYest}

We now give estimates of the solution $(X_t^\varepsilon, Y_t^\varepsilon)$ to system (\ref{eq1}), which will be used frequently in the sequel.

\begin{lemma}\label{avlemma1}
Assume that the conditions in Theorem $\ref{MainRes}$ hold.  Then for any $(x,y) \in \mathcal H \times \mathcal H$ and $\varepsilon >0$ small enough,
there exists a constant $C_T> 0$ such that
	\begin{equation}\label{Xest}
     \mathbb{E}\big[\|X^{\varepsilon}\|_{\eta ,0}^4 \big] \leq  C_T (1+\|x\|_{\mathcal H}^4+\|y\|_{\mathcal H}^4),
     \end{equation}
and
   \begin{equation}\label{Yest}
    \sup_{t\in[0,T]}\mathbb{E}\big[\|Y_{t}^{\varepsilon}\|_{\mathcal H}^4 \big]
    \leq  C_T(1+\|x\|_{\mathcal H}^4+\|y\|_{\mathcal H}^4).
	\end{equation}
\end{lemma}
\begin{proof}
The proof for estimate (\ref{Xest}) is similar to that in \cite[Lemma 2.6]{MG23}, but we will not go into details here.

Applying It\^{o}'s formula to  $\|Y_t^\varepsilon\|_{\mathcal H}^4$ (cf. \cite[Theorem 4.2.5]{LR15}), we have
    \begin{align*}
	\|Y_t^{\varepsilon}\|_{\mathcal H}^4
    =&\|y\|_{\mathcal H}^4
	+\frac{4}{\varepsilon} \int_{0}^t \|Y_{s}^{\varepsilon}\|_{\mathcal H}^2
   \langle A Y_s^\varepsilon, Y_{s}^{\varepsilon} \rangle_{\mathcal H} ds
+\frac{4}{\varepsilon} \int_{0}^t \|Y_{s}^{\varepsilon}\|_{\mathcal H}^2
    \langle F_2 (X_s^\varepsilon, Y_s^\varepsilon) , Y_{s}^{\varepsilon} \rangle_{\mathcal H} ds
	\nonumber\\
    &+\frac{2}{\varepsilon}\int_{0}^t \|Y_{s}^{\varepsilon}\|_{\mathcal H}^2
    \|G_2(X_{s}^{\varepsilon},Y_{s}^{\varepsilon})\|_{\mathcal H^m}^2ds+\frac{4}{\varepsilon}\int_{0}^t
    \|G_2(X_{s}^{\varepsilon},Y_{s}^{\varepsilon})^*Y_{s}^{\varepsilon}\|_{\mathbb R^m}^2ds
	\nonumber\\
	&+\frac{4}{\sqrt{\varepsilon}}\int_{0}^t
    \|Y_{s}^{\varepsilon}\|_{\mathcal H}^2 \langle Y_{s}^{\varepsilon},
    G_2( X_{s}^{\varepsilon},Y_{s}^{\varepsilon})dW_s \rangle_{\mathcal H} .
    \end{align*}
Taking the expectation for the above equation and using (\ref{F2Lip}), (\ref{Glip}) and $(\mathbf{H3})$, we obtain
\begin{align*}
\frac d{dt} \mathbb E[\|Y_t^\varepsilon\|_{\mathcal H}^4] &\leq  \frac{4}{\varepsilon} \mathbb E \big[ \|Y_t^\varepsilon\|_{\mathcal H}^2
    \langle A Y_t^\varepsilon, Y_t^\varepsilon \rangle_{\mathcal H} \big]+\frac{4}{\varepsilon} \mathbb E \big[ \|Y_t^\varepsilon\|_{\mathcal H}^2
    \langle F_2 (X_t^\varepsilon, Y_t^\varepsilon) , Y_t^\varepsilon \rangle_{\mathcal H} \big]
	\nonumber\\
    &~~~+\frac{6}{\varepsilon}\mathbb E \big[ \|Y_t^\varepsilon\|_{\mathcal H}^2
    \|G_2(X_t^\varepsilon,Y_t^\varepsilon)\|_{\mathcal H^m}^2 \big] \nonumber\\
    & \leq  -\frac{4\lambda_1}{\varepsilon} \mathbb E \|Y_t^\varepsilon\|_{\mathcal H}^4
+\frac{4}{\varepsilon} \mathbb E \big[ \|Y_t^\varepsilon\|_{\mathcal H}^3
    \big( C+C\|X_t^\varepsilon\|_{\mathcal H} +L_{F_2} \|Y_t^\varepsilon\|_{\mathcal H} \big) \big]
	\nonumber\\
    &~~~+\frac{6}{\varepsilon}\mathbb E \big[ \|Y_t^\varepsilon\|_{\mathcal H}^2
     \big( C+C\|X_t^\varepsilon\|_{\mathcal H}^2 +2L_{G_2}^2 \|Y_t^\varepsilon\|_{\mathcal H}^2 \big)  \big] \nonumber\\
    & \leq -\frac {4(\lambda_1- L_{F_2}-3 L_{G_2}^2)-\kappa}\varepsilon \mathbb E \|Y_t^\varepsilon \|_{\mathcal H}^4+\frac C\varepsilon (1+\mathbb E \|X_t^\varepsilon \|_{\mathcal H}^4),
\end{align*}
where we used Young's inequality for a constant $\kappa$ satisfying $4(\lambda_1- L_{F_2}-3 L_{G_2}^2)-\kappa>0$.
Using the comparison theorem we have
\begin{align*}
\mathbb E[\|Y_t^\varepsilon\|_{\mathcal H}^4] \leq &   \|y\|_{\mathcal H}^4 e^{-\frac {4(\lambda_1- L_{F_2}-3 L_{G_2}^2)-\kappa}\varepsilon  t}+\frac C\varepsilon \int_0^t e^{-\frac {4(\lambda_1- L_{F_2}-3 L_{G_2}^2)-\kappa}\varepsilon  (t-s)} (1+\mathbb E \|X_s^\varepsilon \|_{\mathcal H}^4) ds\\
    \leq & \|y\|_{\mathcal H}^4 e^{-\frac {4(\lambda_1- L_{F_2}-3 L_{G_2}^2)-\kappa}\varepsilon  t}\!+\frac C\varepsilon \int_0^t \!e^{-\frac {4(\lambda_1- L_{F_2}-3 L_{G_2}^2)-\kappa}\varepsilon  (t-s)} (1\!+\!\|S_s x\|_{\mathcal H}^4\!+ \mathbb E \|X_s^\varepsilon \!-\!S_s x\|_{\mathcal H}^4) ds \nonumber\\
    \leq & \|y\|_{\mathcal H}^4 e^{-\frac {4(\lambda_1- L_{F_2}-3 L_{G_2}^2)-\kappa}\varepsilon  t}+\frac C\varepsilon \int_0^t e^{-\frac {4(\lambda_1- L_{F_2}-3 L_{G_2}^2)-\kappa}\varepsilon  (t-s)} (1+\|x\|_{\mathcal H}^4+T^{4\eta} \mathbb E \|X^\varepsilon \|_{\eta,0}^4) ds \nonumber\\
    \leq & C_T (1+\|x\|_{\mathcal H}^4+\|y\|_{\mathcal H}^4),
\end{align*}	
where in third inequality, we used the contractive property of the semigroup $(S_t)_{t \geq 0}$, and in the last inequality, we used the estimate (\ref{Xest}).
Hence, estimate (\ref{Yest}) holds.
The proof is complete.
\hfill $\square$

\end{proof}

For the next two lemmas we assume that there exists $M>0$ such that $|B|_\gamma, |B^2|_{2\gamma}$ and $\|X^\varepsilon, G_1(X^\varepsilon)\|_{\mathcal{D}_{\mathbf B}^{2\gamma,\eta}} <M$, $\mathbb P$-a.s..

\begin{lemma}\label{Xetaest}
Assume that conditions $(\mathbf{H1})$ and $(\mathbf{H2})$ hold.
For each $x \in \mathcal{H}_\eta$ and $y \in \mathcal{H}$ with $\eta \in [0,\gamma)$, $\gamma \in (1/3, \gamma_0)$ and $T>0$, there exists a constant $C_{M,T}>0$ such that
\begin{equation*}
\sup_{\varepsilon \in (0,1)} \mathbb E \Big[ \sup_{t \in [0,T]} \|X_t^\varepsilon\|_{\mathcal{H}_\eta}^4 \Big]
\leq C_{M,T} (1+\|x\|_{\mathcal{H}_\eta}^4+\|y\|_{\mathcal{H}}^4 ) .
\end{equation*}

\end{lemma}
\begin{proof}
Recall that
$$ X_t^\varepsilon=S_t x+ \int_0^t S_{t-s} F_1 (X_s^\varepsilon,Y_s^\varepsilon)ds+ \int_0^t S_{t-s} G_1 (X_s^\varepsilon) d\mathbf{B}_s. $$
Clearly then,
\begin{align*}
\|X_t^\varepsilon\|_{\mathcal{H}_\eta}^4
\leq C\|S_t x\|_{\mathcal{H}_\eta}^4+ C\Big\| \int_0^t S_{t-s} F_1 (X_s^\varepsilon,Y_s^\varepsilon)ds \Big\|_{\mathcal{H}_\eta}^4
+ C\Big\| \int_0^t S_{t-s} G_1 (X_s^\varepsilon) d\mathbf{B}_s \Big\|_{\mathcal{H}_\eta}^4.
\end{align*}
By the contractive property of the semigroup $(S_t)_{t \geq 0}$, it is clear that
\begin{equation} \label{initest}
\|S_t x\|_{\mathcal{H}_\eta}^4 \leq \|x\|_{\mathcal{H}_\eta}^4.
\end{equation}
For $\Big\| \displaystyle\int_0^t S_{t-s} F_1 (X_s^\varepsilon,Y_s^\varepsilon)ds \Big\|_{\mathcal{H}_\eta}^4$,
using (\ref{semigroup}), (\ref{F1Lip}) and H\"{o}lder's inequality yields that
\begin{align*}
\Big\| \int_0^t S_{t-s} F_1 (X_s^\varepsilon,Y_s^\varepsilon)ds \Big\|_{\mathcal{H}_\eta}^4 & \lesssim  \Big( \int_0^t |t-s|^{-\eta} \|F_1 (X_s^\varepsilon,Y_s^\varepsilon)\|_{\mathcal{H}} ds \Big)^4 \nonumber\\
& \lesssim \Big( \int_0^t |t-s|^{- \frac 43 \eta} ds \Big)^3 \Big( \int_0^t \|F_1 (X_s^\varepsilon,Y_s^\varepsilon)\|_{\mathcal{H}}^4 ds \Big) \nonumber\\
& \leq  C_T \int_0^t (1+\|X_s^\varepsilon\|_{\mathcal{H}}^4+\|Y_s^\varepsilon\|_{\mathcal{H}}^4) ds \nonumber\\
& \leq  C_T \int_0^t (1+\|S_s x\|_{\mathcal{H}}^4+\|X_s^\varepsilon-S_s x\|_{\mathcal{H}}^4+\|Y_s^\varepsilon\|_{\mathcal{H}}^4) ds \nonumber\\
& \leq  C_T \int_0^t (1+\| x\|_{\mathcal{H}}^4+\|X^\varepsilon\|_{\eta,0}^4 s^{4\eta}+\|Y_s^\varepsilon\|_{\mathcal{H}}^4) ds.
\end{align*}
Thus, using estimates (\ref{Xest}) and (\ref{Yest}), we obtain
\begin{align}\label{drifest}
\mathbb E \Big[\sup_{t\in [0,T]} \Big\| \int_0^t S_{t-s} F_1 (X_s^\varepsilon,Y_s^\varepsilon)ds \Big\|_{\mathcal{H}_\eta}^4 \Big] & \leq  C_T \mathbb E \int_0^T (1+\| x\|_{\mathcal{H}}^4+\|X^\varepsilon\|_{\eta,0}^4 s^{4\eta}+\|Y_s^\varepsilon\|_{\mathcal{H}}^4) ds \nonumber\\
& \leq  C_T (1+\|x\|_{\mathcal H}^4+\|y\|_{\mathcal H}^4).
\end{align}
For $\Big\| \displaystyle\int_0^t S_{t-s} G_1 (X_s^\varepsilon) d\mathbf{B}_s \Big\|_{\mathcal{H}_\eta}$,
using (\ref{IntBound}) and $(\mathbf{H2})$ we have
\begin{align}\label{difest}
\Big\| \displaystyle\int_0^t S_{t-s} G_1 (X_s^\varepsilon) d\mathbf{B}_s \Big\|_{\mathcal{H}_\eta} &\leq  \Big\| \displaystyle\int_0^t S_{t-s} G_1 (X_s^\varepsilon) d\mathbf{B}_s-S_t G_1(x) \delta B_{t,0}-S_t DG_1(x)G_1(x) B_{t,0}^2 \Big\|_{\mathcal{H}_\eta} \nonumber\\
&~~~+\|S_t G_1(x) \delta B_{t,0}\|_{\mathcal{H}_\eta}
+\|S_t DG_1(x)G_1(x) B_{t,0}^2\|_{\mathcal{H}_\eta} \nonumber\\
&\lesssim  \big(|R^{G_1(X^\varepsilon)} |_{2\gamma, -2\gamma} |B|_\gamma+ \|DG_1(X^\varepsilon)G_1(X^\varepsilon)\|_{\gamma, -2\gamma} |B^2|_{2\gamma} \big)|t|^{\gamma-\eta} \nonumber\\
&~~~+t^\gamma |B|_\gamma \|G_1(x)\|_{\mathcal{H}_\eta^d}+t^{2\gamma}|B^2|_{2\gamma} \|DG_1(x)G_1(x)\|_{\mathcal{H}_\eta^{d \times d}} \nonumber\\
&\lesssim  T^{\gamma-\eta} \big(| R^{G_1(X^\varepsilon)} |_{2\gamma, -2\gamma} |B|_\gamma+ \|DG_1(X^\varepsilon)G_1(X^\varepsilon)\|_{\gamma, -2\gamma} |B^2|_{2\gamma} \big) \nonumber\\
&~~~+\big(T^\gamma |B|_\gamma+T^{2\gamma}|B^2|_{2\gamma} \big) (1+\|x\|_{\mathcal{H}_\eta}).
\end{align}

Below,
we will estimate $|R^{G_1(X^\varepsilon)} |_{2\gamma, -2\gamma}$ and $\|DG_1(X^\varepsilon)G_1(X^\varepsilon)\|_{\gamma, -2\gamma}$ in (\ref{difest}) respectively.\\
\textbf{Bound of $DG_1(X^\varepsilon)G_1(X^\varepsilon)$}:
We simply write $Z^\prime:=DG_1(X^\varepsilon)G_1(X^\varepsilon)$ and rewrite $\hat{\delta}Z^{\prime}$ as
$$
\begin{aligned}
(\hat{\delta}Z^{\prime})_{t,s}&=\big(DG_1(X_t^\varepsilon)G_1(X_t^\varepsilon)-DG_1(X_t^\varepsilon)S_{t-s}G_1(X_s^\varepsilon) \big) \\
&~~~+\big(DG_1(X_t^\varepsilon)S_{t-s}G_1(X_s^\varepsilon)-DG_1(X_s^\varepsilon)S_{t-s}G_1(X_s^\varepsilon) \big)  \\
&~~~+\big(DG_1(X_s^\varepsilon)S_{t-s}G_1(X_s^\varepsilon)-DG_1(X_s^\varepsilon)G_1(X_s^\varepsilon) \big) \\
&~~~+\big(DG_1(X_s^\varepsilon)G_1(X_s^\varepsilon)-S_{t-s}DG_1(X_s^\varepsilon)G_1(X_s^\varepsilon) \big)  \\
&=:I_1+I_2+I_3+I_4.
\end{aligned}
$$
Using (\ref{semigroup}) and the boundedness of $G_1$ and its derivatives, we can bound these terms as follows:
\begin{align*}
&\|I_1\|_{\mathcal{H}_{-2\gamma}^{d \times d}}
\leq\|DG_1(X_t^\varepsilon)\|_{\mathcal L(\mathcal{H}_{-2\gamma};\mathcal{H}_{-2\gamma}^d)}\|G_1(X^\varepsilon)\|_{\gamma,-2\gamma}|t-s|^{\gamma}
\leq C\|G_1(X^\varepsilon)\|_{\gamma,-2\gamma}|t-s|^{\gamma},   \\
&\|I_2\|_{\mathcal{H}_{-2\gamma}^{d \times d}}\leq\|DG_1(X_t^\varepsilon)-DG_1(X_s^\varepsilon)\|_{\mathcal L(\mathcal{H}_{-2\gamma};\mathcal{H}_{-2\gamma}^d)}\|G_1(X_s^\varepsilon)\|_{\mathcal{H}_{-2\gamma}^d}  \nonumber\\
&~~~~~~~~~~~~\leq C|X^\varepsilon|_{\gamma,-2\gamma}|t-s|^{\gamma}\|G_1(X^\varepsilon)\|_{\infty,0}  \nonumber\\
&~~~~~~~~~~~~\leq C \big(\|X^\varepsilon\|_{\gamma,-2\gamma}+T^{\gamma+\eta}\|X^\varepsilon\|_{\eta,0}+ T^\gamma \|x\|_{\mathcal{H}} \big)\|G_1(X^\varepsilon)\|_{\infty,0}|t-s|^{\gamma}  \nonumber\\
&~~~~~~~~~~~~\leq C_T \big(\|X^\varepsilon\|_{\gamma,-2\gamma}+\|X^\varepsilon\|_{\eta,0}+  \|x\|_{\mathcal{H}} \big)\|G_1(X^\varepsilon)\|_{\infty,0}|t-s|^{\gamma} , \\
&\|I_3\|_{\mathcal{H}_{-2\gamma}^{d \times d}}
\leq \|DG_1(X_s^\varepsilon)\|_{\mathcal L(\mathcal{H}_{-2\gamma};\mathcal{H}_{-2\gamma}^d)} \|(S_{t-s}-Id)G_1(X_s^\varepsilon)\|_{\mathcal{H}_{-2\gamma}^d} \nonumber\\
&~~~~~~~~~~~~\lesssim C \|G_1(X_s^\varepsilon)\|_{\mathcal{H}^d} |t-s|^{2\gamma}
\leq C_T\|G_1(X^\varepsilon)\|_{\infty,0}|t-s|^{\gamma},  \\
&\|I_4\|_{\mathcal{H}_{-2\gamma}^{d \times d}}\lesssim \|DG_1(X_s^\varepsilon)G_1(X_s^\varepsilon)\|_{\mathcal{H}^{d \times d}}|t-s|^{2\gamma}
\leq C_T \|G_1(X^\varepsilon)\|_{\infty,0}|t-s|^{\gamma},
\end{align*}
where the constant $C$ depends on the bounds of $G_1$ and its derivatives and also on $T$.
We also used estimates $|X^\varepsilon|_{\gamma, -2\gamma} \leq \|X^\varepsilon\|_{\gamma, -2\gamma}+T^\gamma \|X^\varepsilon\|_{\infty,0}$
and $\|X^\varepsilon\|_{\infty,0} \leq \|X^\varepsilon\|_{\eta, 0} T^\eta +\|S_t x\|_{\mathcal H} \leq \|X^\varepsilon\|_{\eta, 0} T^\eta+\|x\|_{\mathcal H}$ for the term $I_2$.

Combining these estimates and using (\ref{Y}) yields that
\begin{equation}\label{Zgamma}
\|Z'\|_{\gamma,-2\gamma}\leq C_T (1+|B|_\gamma)\big(1+\|X^\varepsilon, G_1(X^\varepsilon)\|_{\mathcal{D}_{\mathbf B}^{2\gamma,\eta}} \big)^2.
\end{equation}
\textbf{Bound of $R^{G_1(X^\varepsilon)}$}:
By (\ref{remainder}) we have
\begin{align*}
R^{G_1(X^\varepsilon)}_{t,s}&=G_1(X_t^\varepsilon)-S_{t-s}G_1(X_s^\varepsilon)-S_{t-s}DG_1(X_s^\varepsilon)(X^\varepsilon)_s^\prime \delta B_{t,s} \\
&=G_1(X_t^\varepsilon)-S_{t-s}G_1(X_s^\varepsilon)-DG_1(X_t^\varepsilon)S_{t-s}(X^\varepsilon)_s^\prime \delta B_{t,s}  \\
&~~~+DG_1(X_t^\varepsilon)S_{t-s}(X^\varepsilon)_s^\prime \delta B_{t,s}-S_{t-s}DG_1(X_s^\varepsilon)(X^\varepsilon)_s^\prime \delta B_{t,s} \\
&= \big(G_1(X_t^\varepsilon)-S_{t-s}G_1(X_s^\varepsilon)-DG_1(X_t^\varepsilon)(X_t^\varepsilon-S_{t-s}X_s^\varepsilon) \big)
+DG_1(X_t^\varepsilon) R_{t,s}^{X^\varepsilon} \\
&~~~+ (I_2+I_3+I_4) \delta B_{t,s} \\
&=: I_5+I_6+(I_2+I_3+I_4) \delta B_{t,s}.
\end{align*}
For $I_5$ we have
$$
\begin{aligned}
I_5&=\big(G_1(X_t^\varepsilon)-G_1(S_{t-s}X_s^\varepsilon)-DG_1(X_t^\varepsilon)(X_t^\varepsilon-S_{t-s}X_s^\varepsilon)\big)
+\big(G_1(S_{t-s}X_s^\varepsilon)-S_{t-s}G_1(X_s^\varepsilon) \big)\\
&=:I_{51}+I_{52}.
\end{aligned}
$$
Using Taylor's theorem we get
\begin{equation}\label{I51}
\|I_{51}\|_{\mathcal{H}_{-2\gamma}^d}\leq C \|X^\varepsilon\|_{\gamma,-2\gamma}^2 |t-s|^{2\gamma}.
\end{equation}
For the `commutator' $I_{52}$, by (\ref{semigroup}) one has
\begin{align} \label{I52}
\|I_{52}\|_{\mathcal{H}_{-2\gamma}^d}
&\leq \|G_1(S_{t-s}X_s^\varepsilon)-G_1(X_s^\varepsilon)\|_{\mathcal{H}_{-2\gamma}^d}
+\|G_1(X_s^\varepsilon)-S_{t-s}G_1(X_s^\varepsilon)\|_{\mathcal{H}_{-2\gamma}^d}  \nonumber\\
&\leq C (\|X^\varepsilon\|_{\infty,0}+\|G_1(X_s^\varepsilon)\|_{\mathcal{H}^d})|t-s|^{2\gamma}
\leq C(1+\|X^\varepsilon\|_{\infty,0})|t-s|^{2\gamma} \nonumber\\
&\leq C_T(1+\|x\|_{\mathcal{H}}+\|X^\varepsilon\|_{\eta,0})|t-s|^{2\gamma}.
\end{align}
For $I_6$, it is easy to see that
\begin{equation} \label{I6}
\|I_6\|_{\mathcal{H}_{-2\gamma}^d} \leq C |R^{X^\varepsilon}|_{2\gamma, -2\gamma} |t-s|^{2\gamma}.
\end{equation}
Combined with estimates for terms $I_2$-$I_4$ and (\ref{I51})-(\ref{I6}), we have
\begin{equation}\label{Zrem}
|R^{G_1(X^\varepsilon)}|_{2\gamma, -2\gamma} \leq C_T(1+|B|_\gamma)^2 \big(1+\|X^\varepsilon, G_1(X^\varepsilon)\|_{\mathcal{D}_{\mathbf B}^{2\gamma,\eta}} \big)^2.
\end{equation}

Inserting (\ref{Zgamma}) and (\ref{Zrem}) into (\ref{difest}) and using $|B|_\gamma, |B^2|_{2\gamma}$ and $\|X^\varepsilon,  G_1(X^\varepsilon)\|_{\mathcal{D}_{\mathbf B}^{2\gamma,\eta}} <M$, $\mathbb P$-a.s., we get
\begin{equation}\label{difresult}
\mathbb E \Big[\sup_{t\in [0,T]} \Big\| \displaystyle\int_0^t S_{t-s} G_1 (X_s^\varepsilon) d\mathbf{B}_s \Big\|_{\mathcal{H}_\eta}^4 \Big]
\leq  C_{M,T} (1+\|x\|_{\mathcal{H}_\eta}^4).
\end{equation}
Combining estimates (\ref{initest}), (\ref{drifest}) and (\ref{difresult}), we obtain the desired result.
\hfill $\square$

\end{proof}

\begin{lemma}\label{Xdifflem}
Suppose that conditions $(\mathbf{H1})$ and $(\mathbf{H2})$ hold.
Then for each $T>0$ and $\eta \in [0,\gamma)$ with $\gamma \in (1/3, \gamma_0)$, there are constants $M>0$ and $C_{M,T}>0$ such that for any $x \in \mathcal{H}_\eta, y\in \mathcal{H}$ and $\varepsilon ,h \in(0,1)$, we have
$$\sup_{t \in [0,T]} \mathbb E \|X_{t+h}^\varepsilon-X_t^\varepsilon\|_{\mathcal{H}}^4
\leq h^{4\eta}  C_{M,T} (1+\|x\|_{\mathcal{H}_\eta}^4+\|y\|_{\mathcal{H}}^4).$$
\end{lemma}
\begin{proof}
By some simple calculations we obtain
\begin{equation}\label{Xdiff}
X_{t+h}^\varepsilon-X_t^\varepsilon=(S_h-Id) X_t^\varepsilon+\int_t^{t+h} S_{t+h-s} F_1(X_s^\varepsilon, Y_s^\varepsilon) ds
+\int_t^{t+h} S_{t+h-s} G_1(X_s^\varepsilon) d \mathbf B_s.
\end{equation}
For the first term of the right hand side of (\ref{Xdiff}), using (\ref{semigroup}) and Lemma \ref{Xetaest} we have
\begin{equation}\label{Xt}
\mathbb E \|(S_h-Id) X_t^\varepsilon\|_{\mathcal{H}}^4 \lesssim h^{4\eta} \mathbb E \|X_t^\varepsilon\|_{\mathcal{H}_\eta}^4
\leq h^{4\eta}  C_{M,T} (1+\|x\|_{\mathcal{H}_\eta}^4+\|y\|_{\mathcal{H}}^4 ).
\end{equation}
For the second term of the right hand side of (\ref{Xdiff}), by H\"{o}lder's inequality, (\ref{F1Lip}), (\ref{Yest}) and Lemma \ref{Xetaest} we get
\begin{align}\label{Xth1}
\mathbb E \Big\|\int_t^{t+h} S_{t+h-s} F_1(X_s^\varepsilon, Y_s^\varepsilon) ds \Big\|_{\mathcal{H}}^4 &\leq  \mathbb E \Big[ \Big(\int_t^{t+h} 1^{\frac 43} ds\Big)^{\frac 34} \Big( \int_t^{t+h} \|S_{t+h-s} F_1(X_s^\varepsilon, Y_s^\varepsilon)\|_{\mathcal{H}}^4 ds\Big)^{\frac 14} \Big]^4 \nonumber\\
&\leq  h^3 \mathbb E \int_t^{t+h} \|F_1(X_s^\varepsilon, Y_s^\varepsilon)\|_{\mathcal{H}}^4 ds \nonumber\\
&\leq  C h^3 \mathbb E \int_t^{t+h} (1+\|X_s^\varepsilon\|_{\mathcal{H}}^4+\| Y_s^\varepsilon\|_{\mathcal{H}}^4) ds \nonumber\\
&\leq  C h^3 \mathbb E \int_t^{t+h} (1+\|X_s^\varepsilon\|_{\mathcal H_\eta}^4 +\|Y_s^\varepsilon\|_{\mathcal{H}}^4) ds \nonumber\\
&\leq  C_{M,T} h^4 (1+\|x\|_{\mathcal H_\eta}^4+\|y\|_{\mathcal H}^4),
\end{align}
where we used the fact that the embedding $\mathcal H_\eta \subset \mathcal H$ is continuous in the fourth inequality.

Next by (\ref{IntBound}) we have
\begin{align*}
& \Big\|\int_t^{t+h} S_{t+h-s} G_1(X_s^\varepsilon) d \mathbf B_s  \Big\|_{\mathcal{H}} \\
\leq & \Big\|\int_t^{t+h} S_{t+h-s} G_1(X_s^\varepsilon) d \mathbf B_s-S_h G_1(X_t^\varepsilon) \delta B_{t+h, t}-S_h DG_1(X_t^\varepsilon)G_1(X_t^\varepsilon) B_{t+h,t}^2 \Big\|_{\mathcal{H}} \\
&+ \|S_h G_1(X_t^\varepsilon) \delta B_{t+h, t}\|_{\mathcal{H}}+\|S_h DG_1(X_t^\varepsilon)G_1(X_t^\varepsilon) B_{t+h,t}^2\|_{\mathcal{H}} \\
\lesssim & \big(|R^{G_1(X^\varepsilon)} |_{2\gamma, -2\gamma} |B|_\gamma+ \|DG_1(X^\varepsilon)G_1(X^\varepsilon)\|_{\gamma, -2\gamma} |B^2|_{2\gamma} \big)|h|^\gamma \\
&+\|G_1(X^\varepsilon)\|_{\infty,0} |B|_\gamma h^\gamma+ \|DG_1(X^\varepsilon)G_1(X^\varepsilon)\|_{\infty,0} |B^2|_{2\gamma} h^{2\gamma}\\
\leq & C_T (|B|_\gamma+|B^2|_{2\gamma})(1+|B|_\gamma)^2 (1+\|X^\varepsilon, G_1(X^\varepsilon)\|_{\mathcal{D}_{\mathbf B}^{2\gamma,\eta}})^2 h^\gamma +C (|B|_\gamma+|B^2|_{2\gamma}) \|G_1(X^\varepsilon)\|_{\infty,0}  h^\gamma \\
\leq & C_{M,T} h^\gamma,
\end{align*}
where we used estimates (\ref{Zgamma}) and (\ref{Zrem}) in the third inequality, and for the last inequality we used $|B|_\gamma, |B^2|_{2\gamma},\|X^\varepsilon, G_1(X^\varepsilon)\|_{\mathcal{D}_{\mathbf B}^{2\gamma,\eta}} <M$, $\mathbb P$-a.s..
Therefore, it is easy to obtain that
\begin{equation}\label{Xth2}
\mathbb E \Big\|\int_t^{t+h} S_{t+h-s} G_1(X_s^\varepsilon) d \mathbf B_s  \Big\|_{\mathcal{H}}^4  \leq C_{M,T} h^{4\gamma}.
\end{equation}
Combining the above computations (\ref{Xt})-(\ref{Xth2}), we have
\begin{equation*}
\mathbb E \|X_{t+h}^\varepsilon-X_t^\varepsilon\|_{\mathcal{H}}^4 \leq h^{4\eta}  C_{M,T} (1+\|x\|_{\mathcal{H}_\eta}^4+\|y\|_{\mathcal{H}}^4 ).
\end{equation*}
The proof is complete.               \hfill $\square$

\end{proof}

\subsection{Frozen equation}\label{SubFrozen}

Given the Hilbert space $\mathcal{H}$, we define the following frozen equation with a fixed slow component $x \in \mathcal{H}$
\begin{equation}\label{frozeneq}
\left\{ \begin{aligned}
	d Y_t&=\big[A Y_t+F_2(x,Y_t)\big]dt+G_2(x,Y_t) d \widetilde{W}_t,\\
	Y_0&= y \in \mathcal{H},
\end{aligned} \right.
\end{equation}
where $\{\widetilde{W}_t\}_{t\geq0}$ is a $m$-dimensional BM defined on another probability space $(\widetilde{\Omega} ,\widetilde{\mathscr F},\{\widetilde{\mathscr F}_t\}_{t\geq0},\widetilde{\mathbb P})$ and is independent of $\{W_t\}_{t\geq0}$.

Obviously, according to conditions (\ref{F2Lip}) and (\ref{Glip}), $F_2$ and $G_2$ satisfy the global Lipschitz condition.
It is easy to verify that Eq. (\ref{frozeneq}) admits a unique solution $\{Y^{x,y}_t\}_{t\in[0,T]}$, which is a homogeneous Markov process.

Let $\{P^x_t\}_{t\geq0}$ be a Markov transition semigroup of process $\{Y^{x,y}_t\}_{t\geq 0}$, i.e., for any bounded measurable function $f:\mathcal{H}\rightarrow \mathcal{H} $ we have
$$P^x_t f(y):=\widetilde{\mathbb E} [f(Y^{x,y}_t)],~~~~ y \in \mathcal{H},~t>0,$$
where $\widetilde{\mathbb E}$ is the expectation operator corresponding to $\widetilde{\mathbb P}$.

By \cite[Theorem 2.2]{HLL20}, there exists a unique invariant measure $\mu^x$ for the transition semigroup $\{P_t^x\}_{t\geq0}$ under the condition $(\mathbf{H4})$ and we have the following exponential ergodicity, the proof of which is similar to \cite[Proposition 4.2]{HLL22}, we omit it here.

\begin{lemma}\label{ergo}
For any $(x,y) \in \mathcal H \times \mathcal H$, there exist constants $\rho>0$ and $C> 0$ such that
$$\|\widetilde{\mathbb E}F_1(x, Y_t^{x,y})-\bar{F}_1(x)\|_{\mathcal{H} }\leq C e^{ -\rho t} (1+\|x\|_{\mathcal H}+\|y\|_{\mathcal H}),$$
where $\bar{F}_1 (x)=\displaystyle\int_{\mathcal{H}} F_1 (x,y) \mu^x (dy)$ for $x\in\mathcal{H}$.
\end{lemma}

In the following, estimates of the solution to the frozen equation (\ref{frozeneq}) are given.
\begin{lemma}\label{lem3.4}
For any given $x,x_1,x_2,y\in \mathcal{H}$, there exists a constant $C>0$ such that
\begin{align}
	&\sup_{t\geq0}\widetilde{\mathbb E}\|Y_t^{x_1,y}-Y_t^{x_2,y}\|_{\mathcal{H}}^2\leq C\|x_1-x_2\|_{\mathcal{H} }^2, \label{frozendifference}\\
	&\sup_{t\geq0}\widetilde{\mathbb E}\|Y^{x,y}_t\|_{\mathcal{H}}^2\leq C(1+\|x\|_{\mathcal{H} }^2+\|y\|_{\mathcal{H}}^2).\label{frozenest}
\end{align}
\end{lemma}
\begin{proof}
Note that the proof of (\ref{frozenest}) is very similar to (\ref{frozendifference}), so we only give the proof of (\ref{frozendifference}) here.
Let $Z_t^{x_1,x_2}:=Y_t^{x_1,y}-Y_t^{x_2,y}$, which satisfies the following equation
\begin{equation*}
	\left\{ \begin{aligned}
		dZ_t^{x_1,x_2}=& A Z_t^{x_1,x_2}dt+\big[F_2(x_1,Y_t^{x_1,y})-F_2(x_2,Y_t^{x_2,y})\big]dt +\big[G_2(x_1,Y_t^{x_1,y})-G_2(x_2,Y_t^{x_2,y})\big] d\widetilde{W}_t,\\
		Z_0^{x_1,x_2}=&0.
	\end{aligned} \right.
\end{equation*}
Applying It\^{o}'s formula to $\|Z_t^{x_1,x_2}\|_{\mathcal H}^2$, we have
\begin{align*}
\|Z_t^{x_1,x_2}\|_{\mathcal H}^2=& 2\int_0^t  \langle A Z_s^{x_1,x_2}, Z_s^{x_1,x_2}\rangle_{\mathcal H} d s+2\int_0^t  \langle F_2(x_1,Y_s^{x_1,y})-F_2(x_2,Y_s^{x_2,y}), Z_s^{x_1,x_2} \rangle_{\mathcal H} d s\\
&+\int_0^t \|G_2(x_1,Y_s^{x_1,y})-G_2(x_2,Y_s^{x_2,y})\|_{\mathcal H^m}^2 d s\\
&+2\int_{0}^{t}\langle Z_s^{x_1,x_2}, (G_2(x_1,Y_s^{x_1,y})-G_2(x_2,Y_s^{x_2,y})) d\widetilde{W}_s \rangle_{\mathcal H}.
\end{align*}
By (\ref{F2Lip}), (\ref{Glip}) and Young's inequality, we have
\begin{align*}
	\frac d{dt}\widetilde{\mathbb E}\|Z_t^{x_1,x_2}\|_{\mathcal{H}}^2&= 2\widetilde{\mathbb E}\big[
       \langle A Z_t^{x_1,x_2},Z_t^{x_1,x_2}\rangle_{\mathcal H }\big] +2\widetilde{\mathbb E}\big[\langle F_2(x_1,Y_t^{x_1,y})-F_2(x_2,Y_t^{x_2,y}),Z_t^{x_1,x_2}\rangle_{\mathcal{H}}\big]\\
    &~~~+\widetilde{\mathbb E}\big[\| G_2(x_1,Y_t^{x_1,y})-G_2(x_2,Y_t^{x_2,y}) \|_{\mathcal{H}^m}^2\big]\\
	& \leq -2(\lambda_1-L_{F_2}-L_{G_2}^2)\widetilde{\mathbb E}\|Z_t^{x_1,x_2}\|_{\mathcal{H}}^2 +C\widetilde{\mathbb E}\big[\|x_1-x_2\|_{\mathcal{H} }\|Z_t^{x_1,x_2}\|_{\mathcal{H}}\big]+C\|x_1-x_2\|_{\mathcal H}^2\\
	&\leq -(\lambda_1-L_{F_2}-L_{G_2}^2) \widetilde{\mathbb E}\|Z_t^{x_1,x_2}\|_{\mathcal{H}}^2+C\|x_1-x_2\|_{\mathcal{H} }^2,
\end{align*}
where we used Young's inequality for the last inequality.
The comparison theorem implies that
\begin{equation*}
\widetilde{\mathbb E}\|Z_t^{x_1,x_2}\|_{\mathcal{H}}^2\leq C\|x_1-x_2\|_{\mathcal{H} }^2 \int_0^t e^{-(\lambda_1-L_{F_2}-L_{G_2}^2)(t-s)} ds\leq C\|x_1-x_2\|_{\mathcal{H} }^2.
\end{equation*}
Combining the above computations, we get the desired result.
\hfill $\square$

\end{proof}

\subsection{Averaged equation}\label{sec3.3}

We now consider the following averaged RPDE
\begin{equation}\label{Averagedeq}
\left\{ \begin{aligned}
	d\bar{X}_t&=\big[A \bar{X}_t+\bar{F}_1(\bar{X}_t)\big]dt+G_1(\bar{X}_t)d \mathbf B_t ,\\
	\bar{X}_0&=x,
\end{aligned} \right.
\end{equation}
where the averaged coefficient $\bar{F}_1(x)=\displaystyle\int_{\mathcal{H}} F_1(x,y)\mu^x(dy),\,x\in\mathcal{H},
\mu^x$ is the unique invariant measure of the Markov semigroup $\{P_t^x\}_{t\geq0}$ of the frozen equation (\ref{frozeneq}).

It can be verified that $\bar{F}_1:\mathcal{H} \to \mathcal{H}$ is also Lipschitz continuous, since it is known that $F_1: \mathcal{H} \times \mathcal{H} \to \mathcal{H}$ is Lipschitz continuous by $(\ref{F1Lip})$.
By referring to \cite[Theorem 2.3]{MG23}, we can obtain the existence and uniqueness of solutions to Eq. (\ref{Averagedeq}).

\begin{lemma}\label{AveragedEstimate}
Suppose that conditions $(\mathbf{H1})$ and $(\mathbf{H2})$ hold.
For any $T>0$, $x\in \mathcal H$ and $\mathbf B=(B,B^2)\in \mathscr{C}^\gamma([0,T]; \mathbb R^d)$ with $\gamma \in (1/3, \gamma_0)$,
Eq.~\eqref{Averagedeq} has a unique global solution $(\bar{X},\bar{X}^\prime) \in \mathcal D_{\mathbf B}^{2\gamma, \eta}([0,T]; \mathcal H)$ with $\eta \in [0,\gamma)$ such that
$\bar{X}^\prime=G_1(\bar{X})$ and
$$\bar{X}_t=S_t x +\int_0^t S_{t-s} \bar{F}_1(\bar{X}_s)ds +\int_0^t S_{t-s} G_1 (\bar{X}_s) d \mathbf B_s, ~~~~ t\in [0,T].$$
Moreover, there exists a constant $C>0$ such that
\begin{equation}\label{AveragedBound}
	 \|\bar{X}\|_{\eta, 0} \leq C(1+\|x\|_{\mathcal H}).
\end{equation}
\end{lemma}



\subsection{Auxiliary process}\label{sub3.4}

The time discretization technique proposed by Khasminskii in \cite{K68} will be used in the proof of Theorem \ref{MainRes}.
To be more precise, we first divide the time interval $[0,T]$ into some subintervals of size $\delta$,
where $\delta$ is a fixed positive constant depending on $\varepsilon$, the specific value of which will be given later.
Then, on each subinterval $[k\delta,(k+1)\delta \wedge T]$, $k\in\mathbb{N}$, we construct the following auxiliary process
\begin{equation}\label{Auxeq}
\hat{Y}_t^\varepsilon=\hat{Y}_{k\delta}^\varepsilon+\frac 1\varepsilon\int_{k\delta}^t A \hat{Y}_s^\varepsilon ds+\frac 1\varepsilon\int_{k\delta}^t F_2(X_{k\delta}^\varepsilon,\hat{Y}_s^\varepsilon)ds+\frac 1{\sqrt{\varepsilon}}\int_{k\delta}^tG_2(X_{k\delta}^\varepsilon,\hat{Y}_s^\varepsilon)dW_s.
\end{equation}
Obviously, it is equivalent to
\begin{equation*}
\left\{ \begin{aligned}
	d\hat{Y}_t^\varepsilon&=\frac 1\varepsilon \big[A \hat{Y}_t^\varepsilon+F_2 (X_{t(\delta)}^\varepsilon,\hat{Y}_t^\varepsilon)\big]dt+\frac 1{\sqrt{\varepsilon}} G_2(X_{t(\delta)}^\varepsilon,\hat{Y}_t^\varepsilon)dW_t,\\
	\hat{Y}_0^\varepsilon&=y \in \mathcal{H},
\end{aligned} \right.
\end{equation*}
where $t(\delta):=[\frac{t}{\delta}]\delta$ with $[x]$ being the floor function, i.e. $[x]$ is the largest integer not exceeding $x$.

Now we give an estimate of the error between the solution $Y_t^\varepsilon$ to Eq. (\ref{eqY}) and the auxiliary process $\hat{Y}_t^\varepsilon$, which plays an important role in the proof of Theorem \ref{MainRes}.

\begin{lemma}\label{AuxiliaryEstimate}
For any $T>0$ and $(x,y)\in \mathcal H_\eta \times \mathcal{H}$ with $\eta \in [0,\gamma)$ and $\gamma\in (1/3,\gamma_0)$, there exists a constant $C_{M,T}>0$ such that
\begin{equation}\label{AimAuxdifference}
\sup_{t\in [0,T]}\mathbb E \| Y_t^\varepsilon-\hat{Y}_t^\varepsilon\|_{\mathcal H}^4 \leq  C_{M,T} \delta^{4 \eta} (1+\|x\|_{\mathcal{H}_\eta}^4+\|y\|_{\mathcal{H}}^4 ),
\end{equation}
and
\begin{equation}\label{Auxest}
	\sup_{t\in[0,T]}\mathbb E\|\hat{Y}_t^\varepsilon\|_{\mathcal H}^2\leq C_T(1+\|x\|_{\mathcal H}^2+\|y\|_{\mathcal H}^2).
\end{equation}
\end{lemma}
\begin{proof}
The proof of estimate (\ref{Auxest}) is similar to that of (\ref{Yest}), we only prove (\ref{AimAuxdifference}).
It is easy to see that
\begin{align*}
	 d ( Y_t^\varepsilon-\hat{Y}_t^\varepsilon)=&\frac 1\varepsilon  A (Y_t^\varepsilon-
	\hat{Y}_t^\varepsilon) dt
	+\frac 1\varepsilon \big[F_2(X_t^\varepsilon, Y_t^\varepsilon)-
	F_2(X_{t(\delta)}^\varepsilon,\hat{Y}_t^\varepsilon)\big]dt \\
&+ \frac 1{\sqrt{\varepsilon}}  \big[G_2(X_t^\varepsilon, Y_t^\varepsilon)-
	G_2(X_{t(\delta)}^\varepsilon,\hat{Y}_t^\varepsilon)\big]dW_t.
\end{align*}
Using It\^{o}'s formula we have
\begin{align*}
\| Y_t^\varepsilon-\hat{Y}_t^\varepsilon \|_{\mathcal{H}}^4
=&\frac 4\varepsilon \int_0^t \| Y_s^\varepsilon-\hat{Y}_s^\varepsilon \|_{\mathcal{H}}^2 \langle A (Y_s^\varepsilon-\hat{Y}_s^\varepsilon), Y_s^\varepsilon-\hat{Y}_s^\varepsilon \rangle_{\mathcal{H}} ds \\
&+\frac 4\varepsilon \int_0^t \| Y_s^\varepsilon-\hat{Y}_s^\varepsilon \|_{\mathcal{H}}^2 \langle F_2 (X_s^\varepsilon, Y_s^\varepsilon)-F_2(X_{s(\delta)}^\varepsilon,\hat{Y}_s^\varepsilon), Y_s^\varepsilon-\hat{Y}_s^\varepsilon \rangle_{\mathcal{H}} ds \\
&+\frac 4{\sqrt{\varepsilon}} \int_0^t \| Y_s^\varepsilon-\hat{Y}_s^\varepsilon \|_{\mathcal{H}}^2 \langle  [ G_2 (X_s^\varepsilon, Y_s^\varepsilon)-G_2(X_{s(\delta)}^\varepsilon,\hat{Y}_s^\varepsilon) ]dW_s, Y_s^\varepsilon-\hat{Y}_s^\varepsilon \rangle_{\mathcal{H}}  \\
&+\frac 2\varepsilon \int_0^t \| Y_s^\varepsilon-\hat{Y}_s^\varepsilon \|_{\mathcal{H}}^2 \| G_2 (X_s^\varepsilon, Y_s^\varepsilon)-G_2(X_{s(\delta)}^\varepsilon,\hat{Y}_s^\varepsilon) \|_{\mathcal{H}^m}^2 ds \\
&+\frac 4\varepsilon \int_0^t \| (G_2 (X_s^\varepsilon, Y_s^\varepsilon)-G_2(X_{s(\delta)}^\varepsilon,\hat{Y}_s^\varepsilon) )^* (Y_s^\varepsilon-\hat{Y}_s^\varepsilon)\|_{\mathbb R^m}^2 ds .
\end{align*}
It is straightforward that
\begin{align*}
	\frac d{dt}\mathbb E\| Y_t^\varepsilon-\hat{Y}_t^\varepsilon \|_{\mathcal{H}}^4
\leq &\frac 4\varepsilon \mathbb E \big[ \| Y_t^\varepsilon-\hat{Y}_t^\varepsilon \|_{\mathcal{H}}^2  \langle A (Y_t^\varepsilon-\hat{Y}_t^\varepsilon), Y_t^\varepsilon-\hat{Y}_t^\varepsilon \rangle_{\mathcal{H}} \big] \\
&+\frac 4\varepsilon \mathbb E \big[ \| Y_t^\varepsilon-\hat{Y}_t^\varepsilon \|_{\mathcal{H}}^2 \langle F_2 (X_t^\varepsilon, Y_t^\varepsilon)-F_2(X_{t(\delta)}^\varepsilon,\hat{Y}_t^\varepsilon), Y_t^\varepsilon-\hat{Y}_t^\varepsilon \rangle_{\mathcal{H}} \big] \\
&+\frac 6\varepsilon \mathbb E \big[ \| Y_t^\varepsilon-\hat{Y}_t^\varepsilon \|_{\mathcal{H}}^2 \| G_2 (X_t^\varepsilon, Y_t^\varepsilon)-G_2(X_{t(\delta)}^\varepsilon,\hat{Y}_t^\varepsilon) \|_{\mathcal{H}^m}^2 \big].
\end{align*}
By conditions (\ref{F2Lip}), (\ref{Glip}) and Young's inequality, we get
\begin{align*}
	\frac d{dt}\mathbb E\|  Y_t^\varepsilon-\hat{Y}_t^\varepsilon\|_{\mathcal{H}}^4
	\leq& -\frac {4 \lambda_1}\varepsilon \mathbb E \|Y_t^\varepsilon-\hat{Y}_t^\varepsilon \|_{\mathcal{H}}^4 +\frac 4\varepsilon \mathbb E \big[ \| Y_t^\varepsilon-\hat{Y}_t^\varepsilon \|_{\mathcal{H}}^3 \big(C \|X_t^\varepsilon-X_{t(\delta)}^\varepsilon\|_{\mathcal{H}} +L_{F_2} \| Y_t^\varepsilon-\hat{Y}_t^\varepsilon \|_{\mathcal{H}} \big) \big] \\
&+\frac 6\varepsilon \mathbb E \big[ \| Y_t^\varepsilon-\hat{Y}_t^\varepsilon \|_{\mathcal{H}}^2 \big(C \|X_t^\varepsilon-X_{t(\delta)}^\varepsilon\|_{\mathcal{H}}^2 +2L_{G_2}^2 \| Y_t^\varepsilon-\hat{Y}_t^\varepsilon \|_{\mathcal{H}}^2 \big) \big] \\
\leq& -\frac {4(\lambda_1- L_{F_2}-3 L_{G_2}^2) }\varepsilon \mathbb E \|Y_t^\varepsilon-\hat{Y}_t^\varepsilon \|_{\mathcal{H}}^4 + \frac C\varepsilon \mathbb E \big[\| Y_t^\varepsilon-\hat{Y}_t^\varepsilon \|_{\mathcal{H}}^3\|X_t^\varepsilon-X_{t(\delta)}^\varepsilon\|_{\mathcal{H}}  \big] \\
&+ \frac C\varepsilon \mathbb E \big[\| Y_t^\varepsilon-\hat{Y}_t^\varepsilon \|_{\mathcal{H}}^2\|X_t^\varepsilon-X_{t(\delta)}^\varepsilon\|_{\mathcal{H}}^2  \big]\\
\leq & -\frac {4(\lambda_1- L_{F_2}-3 L_{G_2}^2)-\kappa }\varepsilon \mathbb E \|Y_t^\varepsilon-\hat{Y}_t^\varepsilon \|_{\mathcal{H}}^4
+\frac C\varepsilon \mathbb E \|X_t^\varepsilon-X_{t(\delta)}^\varepsilon\|_{\mathcal{H}}^4,
\end{align*}
where we used Young's inequality for a constant $\kappa$ satisfying $4(\lambda_1- L_{F_2}-3 L_{G_2}^2)-\kappa>0$.
The comparison theorem implies that
\begin{align*}
\mathbb E\|  Y_t^\varepsilon-\hat{Y}_t^\varepsilon\|_{\mathcal{H}}^4
&\leq  \frac C\varepsilon \int_0^t e^{-\frac {4(\lambda_1- L_{F_2}-3 L_{G_2}^2)-\kappa }\varepsilon (t-s)} \mathbb E \|X_s^\varepsilon-X_{s(\delta)}^\varepsilon\|_{\mathcal{H}}^4 ds  \\
&\leq \sup_{s\in [0,T] }\mathbb E\|X_s^\varepsilon-X_{s(\delta)}^\varepsilon\|_{\mathcal{H}}^4 \cdot \frac C\varepsilon \int_0^t e^{-\frac {4(\lambda_1- L_{F_2}-3 L_{G_2}^2)-\kappa }\varepsilon (t-s)}ds  \\
&\leq  \delta^{4\eta}  C_{M,T} (1+\|x\|_{\mathcal{H}_\eta}^4+\|y\|_{\mathcal{H}}^4),
\end{align*}
where we used Lemma \ref{Xdifflem} for the last inequality. The proof is complete.     \hfill $\square$
\end{proof}

\subsection{Proof of Theorem \ref{MainRes}}\label{sub3.5}

We are now going to prove Theorem \ref{MainRes}.
Before giving the specific proof, we give the following two lemmas, the proofs of which are similar to \cite[Lemma 2.2 and Lemma 2.3]{MG23},
and we omit the details of the proof.

\begin{lemma}\label{Glem1}
 Let $T> 0$ and $\mathbf B \in \mathscr{C}^\gamma([0,T]; \mathbb R^d)$ for some $\gamma \in (1/3, \gamma_0)$. Then for any $(Y, Y^\prime) \in \mathcal{D} _{\mathbf B}^{2\gamma , \eta }([ 0, T] ; \mathcal{H} )$ with $\eta \in [0,\gamma)$ and $ G_1 \in C_{- 2\gamma , 0}^3( \mathcal{H} ; \mathcal{H}^d)$, we have
$$\Big(\int_0^\cdot S_{\cdot-u} G_1 (Y_u) d\mathbf{B}_u, G_1(Y) \Big) \in \mathcal{D}_{\mathbf B}^{2\gamma,\eta}([0,T];\mathcal{H}),$$
and
\begin{align}\label{IntG}
&\Big\| \int_0^{\cdot} S_{\cdot-u} G_1 (Y_u) d\mathbf{B}_u , G_1 (Y) \Big\|_{\mathcal{D}_{\mathbf B}^{2\gamma,\eta}}  \nonumber\\
 \lesssim & \|G_1 (Y)_0^\prime\|_{\mathcal H_{-2\gamma}^{d \times d}}+\|G_1 (Y)\|_{\infty,0}
+T^{\gamma-\eta}(1+|B|_\gamma+|B^2|_{2\gamma})\|G_1 (Y),G_1 (Y)^{\prime}\|_{\mathscr{D}_{S,\mathbf B}^{2\gamma,2\gamma,0}}.
\end{align}
\end{lemma}

\begin{lemma}\label{Glem2}
Let $T> 0$, $G_1 \in C_{- 2\gamma , 0}^3( \mathcal{H} ; \mathcal{H}^d)$ and $ (Y, Y^{\prime}), ( V, V^{\prime }) \in \mathcal{D} _{\mathbf B}^{2\gamma , \eta }( [ 0, T] ; \mathcal{H} ) $ for some $\mathbf{B}\in \mathscr{C}^{\gamma }( [ 0, T] ; \mathbb{R}^d)$ with $\gamma \in (1/3, \gamma_0)$
and there exists $M > 0$ such that $|B| _{\gamma }, |B^2| _{2\gamma}, \|Y, Y^{\prime}\|_{\mathcal{D}_{\mathbf B}^{2\gamma,\eta}}$, $\|V, V^{\prime }\|_{\mathcal{D}_{\mathbf B}^{2\gamma,\eta}} <M$. Then the following estimate holds true
$$\|G_1(Y)-G_1(V),(G_1(Y)-G_1(V))^\prime\|_{\mathscr{D}_{S,\mathbf B}^{2\gamma,2\gamma,0}}\leq C_{M,T}(1+|B|_{\gamma})^2 \|Y-V,(Y-V)^\prime\|_{\mathcal{D}_{\mathbf B}^{2\gamma,\eta}}.$$
\end{lemma}

\begin{theorem}\label{th3}
Assume that the conditions in Theorem \emph{\ref{MainRes}} hold. For any $x\in \mathcal{H}_\eta$ and $y\in \mathcal{H}$ with $\eta \in (\gamma-1/4, \gamma)$, $\gamma \in (1/3, \gamma_0)$ and $1/3 <\gamma_0 \leq 1/2$, we have
\begin{equation*}
	\lim_{\varepsilon\rightarrow0}\mathbb E\big[\|X^\varepsilon-\bar{X}\|_{\eta, 0}^{2}\big]=0.
\end{equation*}
\end{theorem}
\begin{proof}
First, referring to \cite[Corollary 2.1]{MG23},
from the perspective of rough paths, we can show that $\|X^\varepsilon,G_1(X^\varepsilon)\|_{\mathcal{D}_{\mathbf B}^{2\gamma, \eta}}, \|\bar{X},G_1(\bar{X})\|_{\mathcal{D}_{\mathbf B}^{2\gamma, \eta}} \leq C(1+\|x\|_{\mathcal{H}}), \mathbb P$-a.s., when $T > 0$ is sufficiently small.
Therefore, based on a concatenation procedure (cf. \cite{GLS,HN20}),
whenever $\|X^\varepsilon,G_1(X^\varepsilon)\|_{\mathcal{D}_{\mathbf B,[0,t]}^{2\gamma, \eta}}$ is finite we can restart the slow equation $X^\varepsilon$ with initial condition $X_t^\varepsilon$,
and extend the bound further in time such that for any $T>0$ there exists a constant $M>0$ that depends on the initial value $x$, one has
$$\|X^\varepsilon,G_1(X^\varepsilon)\|_{\mathcal{D}_{\mathbf B}^{2\gamma, \eta}} <M, ~~\mathbb P \text{-a.s.}.$$
Similarly, for the averaged equation (\ref{Averagedeq}), we also have $\|\bar{X},G_1(\bar{X})\|_{\mathcal{D}_{\mathbf B}^{2\gamma, \eta}} <M$, $\mathbb P$-a.s. for any $T>0$.
This guarantees that all bounds with respect to $M$ in Lemma \ref{Xetaest}, Lemma \ref{Xdifflem} and Lemma \ref{Glem2} are satisfied.

Next, we will divide the proof into the following three steps:\\
\textbf{Step1:} From Eq. (\ref{eq1}) and Eq. (\ref{Averagedeq}), we have
\begin{equation*}
\left\{ \begin{aligned}
d(X_t^\varepsilon-\bar{X}_t)&=\big[ A (X_t^\varepsilon-\bar{X}_t)+ F_1(X_t^\varepsilon, Y_t^\varepsilon)-\bar{F}_1(\bar{X}_t) \big]dt
+\big[G_1(X_t^\varepsilon)-G_1(\bar{X}_t) \big] d \mathbf B_t, \\
X_0^\varepsilon-\bar{X}_0& = 0.
\end{aligned} \right.
\end{equation*}
Then it is easy to see that
\begin{align*}
X_t^\varepsilon-\bar{X}_t
=& \int_0^t S_{t-s} \big[ F_1(X_s^\varepsilon,Y_s^\varepsilon)-F_1(X_{s(\delta)}^\varepsilon,\hat{Y}_s^\varepsilon) \big] ds +\int_0^t S_{t-s} \big[F_1(X_{s(\delta)}^\varepsilon,\hat{Y}_s^\varepsilon)-\bar{F}_1(X_{s(\delta)}^\varepsilon) \big]ds \nonumber\\
&+\int_0^t S_{t-s} \big[\bar{F}_1(X_{s(\delta)}^\varepsilon)-\bar{F}_1(X_s^\varepsilon) \big]ds  \nonumber+\int_0^t S_{t-s} \big[\bar{F}_1(X_s^\varepsilon)-\bar{F}_1(\bar{X}_s) \big]ds  \nonumber\\
&+ \int_0^t S_{t-s} \big[G_1(X_s^\varepsilon)-G_1(\bar{X}_s) \big] d \mathbf B_s \\
=:& \mathcal M_t^\varepsilon+\int_0^t S_{t-s} \big[\bar{F}_1(X_s^\varepsilon)-\bar{F}_1(\bar{X}_s) \big]ds + \int_0^t S_{t-s} \big[G_1(X_s^\varepsilon)-G_1(\bar{X}_s) \big] d \mathbf B_s,
\end{align*}
where
\begin{align*}
\mathcal M_t^\varepsilon
:=& \int_0^t S_{t-s} \big[F_1(X_s^\varepsilon,Y_s^\varepsilon)-F_1(X_{s(\delta)}^\varepsilon,\hat{Y}_s^\varepsilon) \big]ds +\int_0^t S_{t-s} \big[F_1(X_{s(\delta)}^\varepsilon,\hat{Y}_s^\varepsilon)-\bar{F}_1(X_{s(\delta)}^\varepsilon) \big]ds\nonumber\\
&+\int_0^t S_{t-s} \big[\bar{F}_1(X_{s(\delta)}^\varepsilon)-\bar{F}_1(X_s^\varepsilon) \big]ds.
\end{align*}
Since the Gubinelli derivative of the deterministic integral is zero, we have
\begin{align}\label{J12}
&\big\|(X^\varepsilon-\bar{X}, G_1(X^\varepsilon)-G_1(\bar{X}))-(\mathcal M^\varepsilon,0) \big\|_{\mathcal D_{\mathbf B}^{2\gamma, \eta}} \nonumber\\
= &\Big\| \int_0^\cdot S_{\cdot-s} \big[\bar{F}_1(X_s^\varepsilon)-\bar{F}_1(\bar{X}_s) \big]ds+\int_0^\cdot S_{\cdot-s} \big[G_1(X_s^\varepsilon)-G_1(\bar{X}_s) \big] d \mathbf B_s, G_1(X^\varepsilon)-G_1(\bar{X})  \Big\|_{\mathcal D_{\mathbf B}^{2\gamma, \eta}} \nonumber\\
\leq & \Big\| \int_0^\cdot S_{\cdot-s} \big[\bar{F}_1(X_s^\varepsilon)-\bar{F}_1(\bar{X}_s) \big]ds,0 \Big\|_{\mathcal D_{\mathbf B}^{2\gamma, \eta}} \nonumber\\
&+\Big\|\int_0^\cdot S_{\cdot-s} \big[G_1(X_s^\varepsilon)-G_1(\bar{X}_s) \big] d \mathbf B_s, G_1(X^\varepsilon)-G_1(\bar{X})  \Big\|_{\mathcal D_{\mathbf B}^{2\gamma, \eta}} \nonumber\\
=:& J_1+J_2.
\end{align}
For $J_1$ we have
\begin{equation}\label{J1}
J_1
= \Big\| \int_0^\cdot S_{\cdot-s} \big[\bar{F}_1(X_s^\varepsilon)-\bar{F}_1(\bar{X}_s) \big]ds \Big\|_{\eta,0}
+\Big| R^{\int_0^\cdot S_{\cdot-s} [\bar{F}_1(X_s^\varepsilon)-\bar{F}_1(\bar{X}_s)]ds} \Big|_{2\gamma, -2\gamma}.
\end{equation}
For the first term of the right hand side of (\ref{J1}), by the Lipschitz continuity of $\bar{F}_1$, we have
\begin{align*}
&\Big\| \int_0^t S_{t-u}  \big[\bar{F}_1(X_u^\varepsilon)-\bar{F}_1(\bar{X}_u) \big]du-
S_{t-s} \int_0^s S_{s-u} \big[\bar{F}_1(X_u^\varepsilon)-\bar{F}_1(\bar{X}_u)\big] du  \Big\|_{\mathcal H} \\
= &\Big\| \int_s^t S_{t-u}  \big[\bar{F}_1(X_u^\varepsilon)-\bar{F}_1(\bar{X}_u) \big]du  \Big\|_{\mathcal H} \leq \int_s^t  \big\|S_{t-u}  \big[\bar{F}_1(X_u^\varepsilon)-\bar{F}_1(\bar{X}_u)\big]  \big\|_{\mathcal H} du \\
\leq & \int_s^t  \big\|\bar{F}_1(X_u^\varepsilon)-\bar{F}_1(\bar{X}_u) \big\|_{\mathcal H} du \leq C \|X^\varepsilon-\bar{X}\|_{\infty,0} |t-s|.
\end{align*}
Therefore, it is easy to deduce that
\begin{equation}\label{Feta}
\Big\| \int_0^\cdot S_{\cdot-s} \big[\bar{F}_1(X_s^\varepsilon)-\bar{F}_1(\bar{X}_s) \big]ds \Big\|_{\eta,0}
\leq C \|X^\varepsilon-\bar{X}\|_{\infty,0} T^{1-\eta}.
\end{equation}
For the second term of the right hand side of (\ref{J1}), by (\ref{remainder}) and the Lipschitz continuity of $\bar{F}_1$ we have
\begin{align*}
\Big\|R^{\int_0^\cdot S_{\cdot-u} [\bar{F}_1(X_u^\varepsilon)-\bar{F}_1(\bar{X}_u) ]du }_{t,s} \Big\|_{\mathcal{H}_{-2\gamma}} &=\Big\| \int_s^t S_{t-u}  \big[\bar{F}_1(X_u^\varepsilon)-\bar{F}_1(\bar{X}_u) \big]du \Big\|_{\mathcal{H}_{-2\gamma}} \\
& \leq  \int_s^t  \big\|S_{t-u}  \big[\bar{F}_1(X_u^\varepsilon)-\bar{F}_1(\bar{X}_u)\big]  \big\|_{\mathcal{H}_{-2\gamma}} du \\
& \leq  \int_s^t \|\bar{F}_1(X_u^\varepsilon)-\bar{F}_1(\bar{X}_u)\|_{\mathcal{H}_{-2\gamma}} du \\
& \leq  C\int_s^t \|\bar{F}_1(X_u^\varepsilon)-\bar{F}_1(\bar{X}_u)\|_{\mathcal{H}} du \\
& \leq  C \|X^\varepsilon-\bar{X}\|_{\infty,0} |t-s|,
\end{align*}
where for the second inequality we used the contractive property of the semigroup $(S_t)_{t\geq 0}$,
and for the third inequality we used the fact that $\mathcal H \subset \mathcal H_{-2\gamma}$ is continuous.
Thus we get that
\begin{equation}\label{Frem}
\Big| R^{\int_0^\cdot S_{\cdot-s} [\bar{F}_1(X_s^\varepsilon)-\bar{F}_1(\bar{X}_s) ]ds} \Big|_{2\gamma, -2\gamma}
\leq C \|X^\varepsilon-\bar{X}\|_{\infty,0} T^{1-2\gamma}.
\end{equation}

For $J_2$, by (\ref{IntG}) and Lemma \ref{Glem2} we have
\begin{align}\label{J2}
J_2&=\Big\|\int_0^\cdot S_{\cdot-s} \big[G_1(X_s^\varepsilon)-G_1(\bar{X}_s) \big] d \mathbf B_s, G_1(X^\varepsilon)-G_1(\bar{X}) \Big\|_{\mathcal D_{\mathbf B}^{2\gamma, \eta}}  \nonumber\\
& \lesssim  \|(G_1(X^\varepsilon)-G_1(\bar{X}))_0^\prime\|_{\mathcal H_{-2\gamma}^{d \times d}}
+\|G_1(X^\varepsilon)-G_1(\bar{X})\|_{\infty,0}  \nonumber\\
&~~~+T^{\gamma-\eta}(1+|B|_\gamma+|B^2|_{2\gamma})\|G_1(X^\varepsilon)-G_1(\bar{X}),(G_1(X^\varepsilon)-G_1(\bar{X}))^\prime\|_{\mathscr{D}_{S,\mathbf B}^{2\gamma,2\gamma,0}} \nonumber\\
& \leq  C\|X_0^\varepsilon-\bar{X}_0\|_{\mathcal{H}_{-2\gamma}}+C \|X_0^\varepsilon-\bar{X}_0\|_{\mathcal H}+C T^\eta \|X^\varepsilon-\bar{X}\|_{\eta,0} \nonumber\\
&~~~+T^{\gamma-\eta} (1+|B|_\gamma+|B^2|_{2\gamma}) (1+|B|_\gamma)^2 C_{M,T} \|X^\varepsilon-\bar{X},G_1(X^\varepsilon)-G_1(\bar{X})\|_{\mathcal D_{\mathbf B}^{2\gamma, \eta}} \nonumber\\
& \leq  T^{\gamma-\eta} C_{M,T}\|X^\varepsilon-\bar{X},G_1(X^\varepsilon)-G_1(\bar{X})\|_{\mathcal D_{\mathbf B}^{2\gamma, \eta}}
+C T^\eta \|X^\varepsilon-\bar{X}\|_{\eta,0} +C \|X_0^\varepsilon-\bar{X}_0\|_{\mathcal H},
\end{align}
where we used the fact that the embedding $\mathcal{H} \subset \mathcal H_{-2\gamma}$ is continuous in the last inequality.
Substituting (\ref{Feta})-(\ref{J2}) into (\ref{J12}), we have
\begin{align*}
&\|(X^\varepsilon-\bar{X}, G_1(X^\varepsilon)-G_1(\bar{X}))-(\mathcal M^\varepsilon,0)\|_{\mathcal D_{\mathbf B}^{2\gamma, \eta}} \nonumber\\
\leq&  C\|X^\varepsilon-\bar{X}\|_{\infty,0} T^{1-\eta}+ C\|X^\varepsilon-\bar{X}\|_{\infty,0} T^{1-2\gamma}+C \|X_0^\varepsilon-\bar{X}_0\|_{\mathcal H} \nonumber\\
&+T^{\gamma-\eta} C_{M,T} \|X^\varepsilon-\bar{X},G_1(X^\varepsilon)-G_1(\bar{X})\|_{\mathcal D_{\mathbf B}^{2\gamma, \eta}}
+C T^\eta  \|X^\varepsilon-\bar{X}\|_{\eta,0} \nonumber\\
\leq &  C (T^1 +T^\eta +T^{1+\eta-2\gamma})\|X^\varepsilon-\bar{X}\|_{\eta,0} +C_T \|X_0^\varepsilon-\bar{X}_0\|_{\mathcal H} \nonumber\\
&+T^{\gamma-\eta} C_{M,T} \|X^\varepsilon-\bar{X},G_1(X^\varepsilon)-G_1(\bar{X})\|_{\mathcal D_{\mathbf B}^{2\gamma, \eta}} \nonumber\\
\leq & C_{M,T} (T^1 +T^\eta +T^{1+\eta-2\gamma}+T^{\gamma-\eta}) \|X^\varepsilon-\bar{X},G_1(X^\varepsilon)-G_1(\bar{X})\|_{\mathcal D_{\mathbf B}^{2\gamma, \eta}}+C_T \|X_0^\varepsilon-\bar{X}_0\|_{\mathcal H},
\end{align*}
where in the second inequality, we used the estimate $\|X^\varepsilon-\bar{X}\|_{\infty,0} \leq T^\eta \|X^\varepsilon-\bar{X}\|_{\eta,0} +\|X_0^\varepsilon-\bar{X}_0\|_{\mathcal H}$.

Let $0<T<1$ be small enough such that $C_{M,T} T^\sigma <1/2$, where $\sigma=1 \wedge \eta \wedge (1+\eta-2\gamma) \wedge (\gamma-\eta)$, then we have
\begin{align*}
&\|X^\varepsilon-\bar{X}, G_1(X^\varepsilon)-G_1(\bar{X})\|_{\mathcal D_{\mathbf B}^{2\gamma, \eta}} \nonumber\\
\leq &\|\mathcal M^\varepsilon,0\|_{\mathcal D_{\mathbf B}^{2\gamma, \eta}}
+ C_{M,T} T^\sigma \|X^\varepsilon-\bar{X},G_1(X^\varepsilon)-G_1(\bar{X})\|_{\mathcal D_{\mathbf B}^{2\gamma, \eta}} +C_T \|X_0^\varepsilon-\bar{X}_0\|_{\mathcal H}  \nonumber\\
\leq &\|\mathcal M^\varepsilon,0\|_{\mathcal D_{\mathbf B}^{2\gamma, \eta}}
+ \frac 12 \|X^\varepsilon-\bar{X},G_1(X^\varepsilon)-G_1(\bar{X})\|_{\mathcal D_{\mathbf B}^{2\gamma, \eta}} +C_T \|X_0^\varepsilon-\bar{X}_0\|_{\mathcal H}.
\end{align*}
It is easy to obtain that
$$\|X^\varepsilon-\bar{X}, G_1(X^\varepsilon)-G_1(\bar{X})\|_{\mathcal D_{\mathbf B}^{2\gamma, \eta}} \leq 2 \|\mathcal M^\varepsilon,0\|_{\mathcal D_{\mathbf B}^{2\gamma, \eta}} +C_T \|X_0^\varepsilon-\bar{X}_0\|_{\mathcal H}.$$

For any $T>0$, we divide the time interval $[0,T]$ into $N \in \mathbb N$ identical subintervals such that $ {(\frac T N)}^\sigma C_{M,T} <1/2$,
then on each subinterval $[t_k,t_{k+1}]$ with $t_k=\frac k N T ,k=0, \cdots , N-1$, we have
\begin{equation}\label{Xk}
\|X^\varepsilon-\bar{X}, G_1(X^\varepsilon)-G_1(\bar{X})\|_{\mathcal D_{\mathbf B, [t_k,t_{k+1}]}^{2\gamma, \eta}} \leq 2 \|\mathcal M^\varepsilon,0\|_{\mathcal D_{\mathbf B}^{2\gamma, \eta}}+C_T \|X_k^\varepsilon-\bar{X}_k\|_{\mathcal H},
\end{equation}
where we write $X_k^\varepsilon=X_{t_k}^\varepsilon$ and $\bar{X}_k=\bar{X}_{t_k}$ for simplicity.

Next, we will estimate $\|X_k^\varepsilon-\bar{X}_k\|_{\mathcal H}$ on the subinterval $[t_{k-1},t_k]$ for $k=1, \cdots, N-1$.
Using (\ref{Xk}) one has
\begin{align}\label{xik}
\|X_k^\varepsilon-\bar{X}_k\|_{\mathcal H}
&\leq \|X_k^\varepsilon-\bar{X}_k-S_{t_k-t_{k-1}} (X^\varepsilon_{k-1}-\bar{X}_{k-1})\|_{\mathcal H}
+\|S_{t_k-t_{k-1}} (X^\varepsilon_{k-1}-\bar{X}_{k-1})\|_{\mathcal H}  \nonumber\\
&\leq  \|X^\varepsilon-\bar{X}\|_{\eta,0,[t_{k-1},t_k]} \Big(\frac T N \Big)^\eta+\|X^\varepsilon_{k-1}-\bar{X}_{k-1}\|_{\mathcal H}  \nonumber\\
&\leq  \|X^\varepsilon-\bar{X},G_1(X^\varepsilon)-G_1(\bar{X})\|_{\mathcal D_{\mathbf B, [t_{k-1},t_k]}^{2\gamma, \eta}} \Big(\frac T N \Big)^\sigma
+\|X^\varepsilon_{k-1}-\bar{X}_{k-1}\|_{\mathcal H}  \nonumber\\
&\leq   \big(2 \|\mathcal M^\varepsilon,0\|_{\mathcal D_{\mathbf B}^{2\gamma, \eta}}+C_T \|X^\varepsilon_{k-1}-\bar{X}_{k-1}\|_{\mathcal H} \big) \Big(\frac T N \Big)^\sigma +\|X^\varepsilon_{k-1}-\bar{X}_{k-1}\|_{\mathcal H}  \nonumber\\
&\leq \|\mathcal M^\varepsilon,0\|_{\mathcal D_{\mathbf B}^{2\gamma, \eta}}+ 2\|X^\varepsilon_{k-1}-\bar{X}_{k-1}\|_{\mathcal H},
\end{align}
where in the last inequality, we used $C_{M,T} (\frac T N)^\sigma <1/2$.
Therefore, for $k=1$, according to (\ref{xik}) and $X_0^\varepsilon=\bar{X}_0=x$ we obtain
\begin{equation*}
\|X^\varepsilon_1-\bar{X}_1\|_{\mathcal H}
\leq  \|\mathcal M^\varepsilon,0\|_{\mathcal D_{\mathbf B}^{2\gamma, \eta}}.
\end{equation*}
For $k=2$ we have
\begin{equation*}
\|X^\varepsilon_2-\bar{X}_2\|_{\mathcal H}
\leq  \|\mathcal M^\varepsilon,0\|_{\mathcal D_{\mathbf B}^{2\gamma, \eta}} +2\|X^\varepsilon_1-\bar{X}_1\|_{\mathcal H}
\leq  \|\mathcal M^\varepsilon,0\|_{\mathcal D_{\mathbf B}^{2\gamma, \eta}}+2\|\mathcal M^\varepsilon,0\|_{\mathcal D_{\mathbf B}^{2\gamma, \eta}}.
\end{equation*}
For $k=3$ we have
\begin{align*}
\|X^\varepsilon_3-\bar{X}_3\|_{\mathcal H}
&\leq  \|\mathcal M^\varepsilon,0\|_{\mathcal D_{\mathbf B}^{2\gamma, \eta}} +2\|X^\varepsilon_2-\bar{X}_2\|_{\mathcal H} \\
&\leq  \|\mathcal M^\varepsilon,0\|_{\mathcal D_{\mathbf B}^{2\gamma, \eta}}+2\big (\|\mathcal M^\varepsilon,0\|_{\mathcal D_{\mathbf B}^{2\gamma, \eta}}+2\|\mathcal M^\varepsilon,0\|_{\mathcal D_{\mathbf B}^{2\gamma, \eta}} \big) \\
&= (1+2+2^2)\|\mathcal M^\varepsilon,0\|_{\mathcal D_{\mathbf B}^{2\gamma, \eta}} .
\end{align*}
By recursively using (\ref{xik}) we conclude that
\begin{equation}\label{xiest}
\|X_k^\varepsilon-\bar{X}_k\|_{\mathcal H} \leq (1+2^1+2^2+\cdots+2^{k-1})\|\mathcal M^\varepsilon,0\|_{\mathcal D_{\mathbf B}^{2\gamma, \eta}}
=(2^k-1)\|\mathcal M^\varepsilon,0\|_{\mathcal D_{\mathbf B}^{2\gamma, \eta}}, ~~k=1,\cdots, N-1.
\end{equation}
Substituting (\ref{xiest}) into (\ref{Xk}) yields that
\begin{align} \label{Xkest}
&\|X^\varepsilon-\bar{X}, G_1(X^\varepsilon)-G_1(\bar{X})\|_{\mathcal D_{\mathbf B, [t_k,t_{k+1}]}^{2\gamma, \eta}} \leq 2 \|\mathcal M^\varepsilon,0\|_{\mathcal D_{\mathbf B}^{2\gamma, \eta}}+C_T (2^k-1)\|\mathcal M^\varepsilon,0\|_{\mathcal D_{\mathbf B}^{2\gamma, \eta}} \nonumber\\
\leq & C_T 2^N \|\mathcal M^\varepsilon,0\|_{\mathcal D_{\mathbf B}^{2\gamma, \eta}}.
\end{align}

Consequently, using (\ref{Xkest}) one has
\begin{align}\label{Xdifeta}
\|X^\varepsilon-\bar{X}\|_{\eta, 0} &\leq \sum_{k=0}^{N-1} \|X^\varepsilon-\bar{X}\|_{\eta, 0, [t_k, t_{k+1}]} \leq  \sum_{k=0}^{N-1}\|X^\varepsilon-\bar{X}, G_1(X^\varepsilon)-G_1(\bar{X})\|_{\mathcal D_{\mathbf B, [t_k,t_{k+1}]}^{2\gamma, \eta}} \nonumber\\
&\leq  C_T N 2^N \|\mathcal M^\varepsilon,0\|_{\mathcal D_{\mathbf B}^{2\gamma, \eta}},
\end{align}
where the first inequality is due to: For any $0\leq s <t \leq T$, if $s \in [t_j, t_{j+1}], t\in [t_k, t_{k+1}]$ and $j<k, j,k=0,1, \cdots, N-1$, then we have
\begin{align*}
&\|X_t^\varepsilon-\bar{X}_t-S_{t-s}(X_s^\varepsilon-\bar{X}_s)\|_{\mathcal H}   \\
=& \|X_t^\varepsilon-\bar{X}_t-S_{t-t_k}(X_k^\varepsilon-\bar{X}_k)+S_{t-t_k}(X_k^\varepsilon-\bar{X}_k)-S_{t-t_{k-1}}(X_{k-1}^\varepsilon-\bar{X}_{k-1})+\cdots \\
&+ S_{t-t_{j+2}}(X_{j+2}^\varepsilon-\bar{X}_{j+2})-S_{t-t_{j+1}}(X_{j+1}^\varepsilon-\bar{X}_{j+1})+S_{t-t_{j+1}}(X_{j+1}^\varepsilon-\bar{X}_{j+1})       -S_{t-s}(X_s^\varepsilon-\bar{X}_s)\|_{\mathcal H} \\
\leq& \|X_t^\varepsilon-\bar{X}_t-S_{t-t_k}(X_k^\varepsilon-\bar{X}_k)\|_{\mathcal H}
+\big\|S_{t-t_k} \big[(X_k^\varepsilon-\bar{X}_k)-S_{t_k-t_{k-1}}(X_{k-1}^\varepsilon-\bar{X}_{k-1}) \big] \big\|_{\mathcal H} \\
&+ \cdots + \big \|S_{t-t_{j+2}} \big[ (X_{j+2}^\varepsilon-\bar{X}_{j+2})-S_{t_{j+2}-t_{j+1}}(X_{j+1}^\varepsilon-\bar{X}_{j+1}) \big] \big\|_{\mathcal H}  \\
&+\big\|S_{t-t_{j+1}} \big[ (X_{j+1}^\varepsilon-\bar{X}_{j+1})-S_{t_{j+1}-s}(X_s^\varepsilon-\bar{X}_s) \big] \big\|_{\mathcal H} \\
\leq&  \|X_t^\varepsilon-\bar{X}_t-S_{t-t_k}(X_k^\varepsilon-\bar{X}_k)\|_{\mathcal H}
+\|X_k^\varepsilon-\bar{X}_k-S_{t_k-t_{k-1}}(X_{k-1}^\varepsilon-\bar{X}_{k-1}) \|_{\mathcal H} \\
&+ \cdots + \|X_{j+2}^\varepsilon-\bar{X}_{j+2}-S_{t_{j+2}-t_{j+1}}(X_{j+1}^\varepsilon-\bar{X}_{j+1}) \|_{\mathcal H}  \\
&+\| X_{j+1}^\varepsilon-\bar{X}_{j+1}-S_{t_{j+1}-s}(X_s^\varepsilon-\bar{X}_s) \|_{\mathcal H},
\end{align*}
where we used the contractive property of semigroup $(S_t)_{t\geq 0}$ for the last inequality.\\
\textbf{Step2:}
We estimate $\|\mathcal M^\varepsilon,0\|_{\mathcal D_{\mathbf B}^{2\gamma, \eta}}$. Set
\begin{align}\label{z}
\mathcal M_t^\varepsilon&= \int_0^t S_{t-s} \big[F_1(X_s^\varepsilon,Y_s^\varepsilon)
	-F_1(X_{s(\delta)}^\varepsilon,\hat{Y}_s^\varepsilon) \big]ds+\int_0^t S_{t-s} \big[F_1(X_{s(\delta)}^\varepsilon,\hat{Y}_s^\varepsilon)-\bar{F}_1(X_{s(\delta)}^\varepsilon) \big]ds\nonumber\\
    &~~~+\int_0^t S_{t-s} \big[\bar{F}_1(X_{s(\delta)}^\varepsilon)
	-\bar{F}_1(X_s^\varepsilon) \big]ds\nonumber\\
	&=:\sum_{i=1}^{3} K_i(t).
\end{align}
Below we will estimate the terms $K_i(t)$, $i=1,2,3$, in (\ref{z}), respectively.

It is straightforward that
\begin{equation*}
\mathbb E\|K_1(\cdot),0\|_{\mathcal D_{\mathbf B}^{2 \gamma, \eta}}^2
\leq  2\mathbb E\|K_1(\cdot)\|_{\eta,0}^2+2\mathbb E |R^{K_1(\cdot)}|_{2\gamma,-2\gamma}^2.
\end{equation*}
For $\mathbb E\|K_1(\cdot)\|_{\eta,0}^2$, using H\"{o}lder's inequality, Lemma \ref{Xdifflem} and (\ref{AimAuxdifference}), we have
\begin{align}\label{K_1eta}
\mathbb E\|K_1(\cdot)\|_{\eta,0}^2&=\mathbb E \Big[ \sup_{s<t} \frac {\|K_1(t)-S_{t-s}K_1(s)\|_{\mathcal H}^2} { |t-s|^{2\eta} } \Big]\nonumber\\
&=\mathbb E \Big[ \sup_{s<t} |t-s|^{-2\eta}
 \Big\| \int_s^t S_{t-r} \big[F_1(X_r^\varepsilon,Y_r^\varepsilon)-F_1(X_{r(\delta)}^\varepsilon,\hat{Y}_r^\varepsilon) \big]dr \Big\|_{\mathcal H}^2 \Big] \nonumber\\
&\leq \mathbb E \Big[ \sup_{s<t} |t-s|^{-2\eta}  \int_s^t 1^2 dr \cdot
\int_s^t \big\|  S_{t-r} \big[F_1(X_r^\varepsilon,Y_r^\varepsilon)-F_1(X_{r(\delta)}^\varepsilon,\hat{Y}_r^\varepsilon) \big] \big\|_{\mathcal H}^2 dr\Big] \nonumber\\
&\leq T^{1-2\eta} \mathbb E \Big[ \sup_{s<t} \int_s^t \big\| F_1(X_r^\varepsilon,Y_r^\varepsilon)-F_1(X_{r(\delta)}^\varepsilon,\hat{Y}_r^\varepsilon)  \big\|_{\mathcal H}^2 dr\Big] \nonumber\\
&\leq  T^{1-2\eta} \mathbb E \int_0^T \| F_1(X_r^\varepsilon,Y_r^\varepsilon)-F_1(X_{r(\delta)}^\varepsilon,\hat{Y}_r^\varepsilon) \|_{\mathcal H}^2 dr \nonumber\\
&\leq C T^{1-2\eta}  \int_0^T  \mathbb E \big[\| X_r^\varepsilon
	-X_{r(\delta)}^\varepsilon \|_{\mathcal H}^2+\|Y_r^\varepsilon-\hat{Y}_r^\varepsilon\|_{\mathcal H}^2 \big] dr \nonumber\\
&\leq C_{M,T} (1+\|x\|_{\mathcal{H}_\eta}^2+\|y\|_{\mathcal{H}}^2)\delta^{2\eta}.
\end{align}
For $\mathbb E |R^{K_1(\cdot)}|_{2\gamma,-2\gamma}^2$, making use of (\ref{remainder}) we have
\begin{align*}
\mathbb E|R^{K_1(\cdot)}|_{2\gamma,-2\gamma}^2&=\mathbb E \Big[ \sup_{s<t} \frac {\|K_1(t)-S_{t-s}K_1(s)\|_{\mathcal{H}_{-2\gamma}}^2} { |t-s|^{4\gamma} } \Big]\nonumber\\
&\leq \mathbb E \Big[ \sup_{s<t} |t-s|^{-4\gamma} \Big\| \int_s^t S_{t-r}\big[ F_1(X_r^\varepsilon,Y_r^\varepsilon)-F_1(X_{r(\delta)}^\varepsilon,\hat{Y}_r^\varepsilon) \big]dr \Big\|_{\mathcal{H}_{-2\gamma}}^2  1_{ \{|t-s|\leq \delta\} }\Big] \nonumber\\
&~~~+\mathbb E \Big[ \sup_{s<t} |t-s|^{-4\gamma} \Big\| \int_s^t S_{t-r}\big[F_1(X_r^\varepsilon,Y_r^\varepsilon)-F_1(X_{r(\delta)}^\varepsilon,\hat{Y}_r^\varepsilon) \big]dr \Big\|_{\mathcal{H}_{-2\gamma}}^2  1_{ \{|t-s|> \delta\} }\Big] \nonumber\\
&=: K_{11}+K_{12}.
\end{align*}
For $K_{11}$, by H\"{o}lder's inequality and (\ref{F1Bound}), we obtain
\begin{align}\label{K11}
K_{11}
&\leq \mathbb E \Big[ \sup_{s<t} |t-s|^{-4\gamma}  \int_s^t 1^2 dr  \cdot \int_s^t \big\|S_{t-r}\big[F_1(X_r^\varepsilon,Y_r^\varepsilon)-F_1(X_{r(\delta)}^\varepsilon,\hat{Y}_r^\varepsilon) \big] \big\|_{\mathcal{H}_{-2\gamma}}^2dr   1_{ \{|t-s|\leq \delta\} }\Big] \nonumber\\
&\leq C\mathbb E \Big[ \sup_{s<t} |t-s|^{1-4\gamma} \int_s^t (\|F_1(X_r^\varepsilon,Y_r^\varepsilon)\|_{\mathcal{H}_{-2\gamma}}^2
	+\|F_1(X_{r(\delta)}^\varepsilon,\hat{Y}_r^\varepsilon) \|_{\mathcal{H}_{-2\gamma}}^2 )dr   1_{ \{|t-s|\leq \delta\} }\Big] \nonumber\\
&\leq C\sup_{s<t} \big\{ |t-s|^{2(1-2\gamma)} 1_{ \{|t-s|\leq \delta\} } \big\}  \leq C \delta^{2(1-2\gamma)}.
\end{align}
For $K_{12}$, by H\"{o}lder's inequality, (\ref{F1Lip}), (\ref{AimAuxdifference}) and Lemma \ref{Xdifflem}, we get
\begin{align}\label{K12}
K_{12}
& \leq \mathbb E \Big[ \sup_{s<t} |t-s|^{-4\gamma}  \int_s^t 1^2 dr  \cdot \int_s^t \big\|S_{t-r}\big[F_1(X_r^\varepsilon,Y_r^\varepsilon)-F_1(X_{r(\delta)}^\varepsilon,\hat{Y}_r^\varepsilon) \big] \big\|_{\mathcal{H}_{-2\gamma}}^2dr   1_{ \{|t-s|> \delta\} }\Big] \nonumber\\
& \leq  T^{1-2\gamma}\mathbb E \Big[ \sup_{s<t} |t-s|^{-2\gamma}  \int_s^t \|F_1(X_r^\varepsilon,Y_r^\varepsilon)-F_1(X_{r(\delta)}^\varepsilon,\hat{Y}_r^\varepsilon) \|_{\mathcal{H}_{-2\gamma}}^2dr   1_{ \{|t-s|> \delta\} }\Big] \nonumber\\
&\leq C_T \mathbb E \Big[ \sup_{s<t} |t-s|^{-2\gamma}  \Big(\int_s^t 1^2 dr \Big)^{\frac 12}
     \Big(\int_s^t \|F_1(X_r^\varepsilon,Y_r^\varepsilon)-F_1(X_{r(\delta)}^\varepsilon,\hat{Y}_r^\varepsilon) \|_{\mathcal{H}_{-2\gamma}}^4 dr \Big)^{\frac 12}   1_{ \{|t-s|> \delta\} }\Big] \nonumber\\
&\leq C_T \mathbb E \Big[ \sup_{s<t} |t-s|^{\frac 12-2\gamma}
     \Big(\int_s^t \|F_1(X_r^\varepsilon,Y_r^\varepsilon)-F_1(X_{r(\delta)}^\varepsilon,\hat{Y}_r^\varepsilon) \|_{\mathcal{H}_{-2\gamma}}^4 dr \Big)^{\frac 12}   1_{ \{|t-s|> \delta\} }\Big] \nonumber\\
&\leq C_T \delta^{\frac 12-2\gamma}
  \mathbb E \Big[ \sup_{s<t}
     \Big(\int_s^t \|F_1(X_r^\varepsilon,Y_r^\varepsilon)-F_1(X_{r(\delta)}^\varepsilon,\hat{Y}_r^\varepsilon) \|_{\mathcal{H}_{-2\gamma}}^4 dr \Big)^{\frac 12}   1_{ \{|t-s|> \delta\} }\Big] \nonumber\\
&\leq C_T \delta^{\frac 12-2\gamma}
   \Big\{ \mathbb E
     \int_0^T \|F_1(X_r^\varepsilon,Y_r^\varepsilon)-F_1(X_{r(\delta)}^\varepsilon,\hat{Y}_r^\varepsilon) \|_{\mathcal{H}}^4 dr   \Big\}^{\frac 12}\nonumber\\
&\leq C_T \delta^{\frac 12-2\gamma}
   \Big\{
     \int_0^T \mathbb E\big[\|X_r^\varepsilon-X_{r(\delta)}^\varepsilon\|_{\mathcal H}^4+\|Y_r^\varepsilon-\hat{Y}_r^\varepsilon\|_{\mathcal H}^4 \big] dr    \Big\}^{\frac 12}\nonumber\\
&\leq C_{M,T} (1+\|x\|_{\mathcal{H}_\eta}^2+\|y\|_{\mathcal{H}}^2) \delta^{\frac 12+2\eta-2\gamma},
\end{align}
where $\frac 12+2\eta-2\gamma>0$ because $\eta >\gamma-1/4$.
Combining estimates (\ref{K_1eta})-(\ref{K12}), we have
\begin{equation}\label{K_1est}
\mathbb E\|K_1(\cdot),0\|_{\mathcal D_{\mathbf B}^{2 \gamma, \eta}}^2
\leq  C_{M,T} (1+\|x\|_{\mathcal{H}_\eta}^2+\|y\|_{\mathcal{H}}^2)(\delta^{2\eta}+\delta^{2(1-2\gamma)}+\delta^{\frac 12+2\eta-2\gamma}).
\end{equation}

From the definition of $\bar{F}_1$, it is easy to obtain that $\|\bar{F}_1\|_\infty \leq \|F_1\|_\infty$.
Thus, similar to the estimation of $K_1(t)$, using Lemma \ref{Xdifflem} we get
\begin{equation}\label{K_3est}
\mathbb E\|K_3(\cdot),0\|_{\mathcal D_{\mathbf B}^{2 \gamma, \eta}}^2
\leq  C_{M,T} (1+\|x\|_{\mathcal{H}_\eta}^2+\|y\|_{\mathcal{H}}^2) (\delta^{2\eta}+\delta^{2(1-2\gamma)}+\delta^{\frac 12+2\eta-2\gamma}).
\end{equation}
\textbf{Step3:} Next, let's estimate $K_2(t)$.

For $\mathbb E\|K_2(\cdot)\|_{\eta,0}^2$, we have
\begin{align}\label{K2eta}
\mathbb E\|K_2(\cdot)\|_{\eta,0}^2&=\mathbb E \Big[ \sup_{s<t} |t-s|^{-2\eta}
 \Big\| \int_s^t S_{t-r}\big[F_1(X_{r(\delta)}^\varepsilon,\hat{Y}_r^\varepsilon)-\bar{F}_1(X_{r(\delta)}^\varepsilon) \big]dr \Big\|_{\mathcal H}^2 \Big] \nonumber\\
&\leq 2\mathbb E \Big[ \sup_{s<t} |t-s|^{-2\eta}
 \Big\| \int_s^t \big( S_{t-r}-S_{t-r(\delta)} \big) \big[F_1(X_{r(\delta)}^\varepsilon,\hat{Y}_r^\varepsilon)-\bar{F}_1(X_{r(\delta)}^\varepsilon) \big]dr \Big\|_{\mathcal H}^2 \Big] \nonumber\\
&~~~+ 2\mathbb E \Big[ \sup_{s<t} |t-s|^{-2\eta}
 \Big\| \int_s^t  S_{t-r(\delta)}  \big[F_1(X_{r(\delta)}^\varepsilon,\hat{Y}_r^\varepsilon)-\bar{F}_1(X_{r(\delta)}^\varepsilon) \big]dr \Big\|_{\mathcal H}^2 \Big] \nonumber\\
&=: K_{21}+K_{22} .
\end{align}
For $K_{21}$, by H\"{o}lder's inequality, (\ref{semi}), (\ref{F1Bound}) and $\|\bar{F}_1\|_{\infty} \leq \|F_1\|_\infty $, we obtain
\begin{align}\label{K2eta1}
K_{21}&\leq 4\mathbb E \Big[ \sup_{s<t} |t-s|^{-2\eta}
  \int_s^t \| S_{t-r}-S_{t-r(\delta)} \|_{\mathcal L(\mathcal H;\mathcal H)}^2 dr
  \int_s^t \big(\|F_1(X_{r(\delta)}^\varepsilon,\hat{Y}_r^\varepsilon)\|_{\mathcal H}^2+\|\bar{F}_1(X_{r(\delta)}^\varepsilon)\|_{\mathcal H}^2 \big)dr \Big] \nonumber\\
&\leq C \mathbb E \Big[ \sup_{s<t} |t-s|^{1-2\eta}
  \int_s^t \| S_{t-r}-S_{t-r} \cdot S_{r-r(\delta)} \|_{\mathcal L(\mathcal H;\mathcal H)}^2 dr  \Big] \nonumber\\
&\leq C \mathbb E \Big[ \sup_{s<t} |t-s|^{1-2\eta}
  \int_s^t \| S_{t-r} \|_{\mathcal L(\mathcal H_{-\sigma};\mathcal H)}^2  \| S_{r-r(\delta)}-Id \|_{\mathcal L(\mathcal H;\mathcal H_{-\sigma})}^2 dr  \Big] \nonumber\\
&\leq C \delta^{2\sigma} \mathbb E \Big[ \sup_{s<t} |t-s|^{1-2\eta}
  \int_s^t |t-r|^{-2\sigma}  dr  \Big] \nonumber\\
&\leq C \delta^{2\sigma} \cdot \sup_{s<t} |t-s|^{2(1-\eta-\sigma)} \leq C T^{2(1-\eta-\sigma)} \delta^{2\sigma},
\end{align}
where $0<\sigma< \min{\{ \frac 12 , 1-\eta \} }$.
For $K_{22}$, we note that
\begin{align*}
K_{22}&\leq 2\mathbb E \Big[ \sup_{s<t} |t-s|^{-2\eta}
 \Big\| \int_s^t  S_{t-r(\delta)}  \big[F_1(X_{r(\delta)}^\varepsilon,\hat{Y}_r^\varepsilon)-\bar{F}_1(X_{r(\delta)}^\varepsilon) \big]dr \Big\|_{\mathcal H}^2 1_{\{|t-s|\leq \delta \}}\Big] \nonumber\\
&~~~+ 2\mathbb E \Big[ \sup_{s<t} |t-s|^{-2\eta}
 \Big\| \int_s^t  S_{t-r(\delta)}  \big[F_1(X_{r(\delta)}^\varepsilon,\hat{Y}_r^\varepsilon)-\bar{F}_1(X_{r(\delta)}^\varepsilon) \big]dr \Big\|_{\mathcal H}^2 1_{ \{|t-s|>\delta \}} \Big] \nonumber\\
&\leq 2\mathbb E \Big[ \sup_{s<t} |t-s|^{-2\eta}
 \Big\| \int_s^t  S_{t-r(\delta)}  \big[F_1(X_{r(\delta)}^\varepsilon,\hat{Y}_r^\varepsilon)-\bar{F}_1(X_{r(\delta)}^\varepsilon) \big]dr \Big\|_{\mathcal H}^2 1_{\{|t-s|\leq \delta \}}\Big] \nonumber\\
&~~~+ 6\mathbb E \Big[ \sup_{s<t} |t-s|^{-2\eta}
 \Big\| \int_s^{([s/\delta]+1)\delta}  S_{t-r(\delta)}  \big[F_1(X_{r(\delta)}^\varepsilon,\hat{Y}_r^\varepsilon)-\bar{F}_1(X_{r(\delta)}^\varepsilon) \big]dr \Big\|_{\mathcal H}^2 1_{ \{|t-s|>\delta \}} \Big] \nonumber\\
&~~~+ 6\mathbb E \Big[ \sup_{s<t} |t-s|^{-2\eta}
 \Big\| \int_{t(\delta)}^t  S_{t-r(\delta)}  \big[F_1(X_{r(\delta)}^\varepsilon,\hat{Y}_r^\varepsilon)-\bar{F}_1(X_{r(\delta)}^\varepsilon) \big]dr \Big\|_{\mathcal H}^2 1_{ \{|t-s|>\delta \}} \Big] \nonumber\\
&~~~+ 6\mathbb E \Big[ \sup_{s<t} |t-s|^{-2\eta}
 \Big\| \sum_{k=[s/\delta]+1}^{[t/\delta]-1} \int_{k\delta}^{(k+1)\delta} S_{t-k\delta} \big[F_1(X_{k\delta}^\varepsilon,\hat{Y}_r^\varepsilon)-\bar{F}_1(X_{k\delta}^\varepsilon) \big]dr \Big\|_{\mathcal H}^2 1_{ \{|t-s|>\delta \}} \Big] \nonumber\\
&=:L_1+L_2+L_3+L_4.
\end{align*}
Firstly, using H\"{o}lder's inequality, (\ref{F1Bound}) and $\|\bar{F}_1\|_{\infty} \leq \|F_1\|_\infty $, the term $L_1$ is controlled by
\begin{align}\label{L1}
L_1 &\leq  4\mathbb E \Big[ \sup_{s<t} |t-s|^{-2\eta}
  \int_s^t \|S_{t-r(\delta)} \|_{\mathcal L(\mathcal H;\mathcal H)}^2 dr  \int_s^t   \big(\| F_1(X_{r(\delta)}^\varepsilon,\hat{Y}_r^\varepsilon)\|_{\mathcal H}^2+\|\bar{F}_1(X_{r(\delta)}^\varepsilon) \|_{\mathcal H}^2 \big) dr 1_{\{|t-s|\leq \delta \}}\Big] \nonumber\\
&\leq  C \sup_{s<t} \big\{ |t-s|^{2(1-\eta)} 1_{\{|t-s|\leq \delta \}} \big\} \leq  C \delta^{2(1-\eta)}.
\end{align}
Secondly, using the H\"{o}lder's inequality and the uniform boundedness of $F_1$ and $\bar{F}_1$, we have
\begin{align}\label{L2}
L_2 & \leq  C\delta^{-2\eta} \mathbb E \Big[ \sup_{s<t}
  \int_s^{([s/\delta]+1)\delta} \|S_{t-r(\delta)} \|_{\mathcal L(\mathcal H;\mathcal H)}^2 dr \nonumber\\
  &~~~~~~~~~~~~~~~~~\times \int_s^{([s/\delta]+1)\delta}  \big( \| F_1(X_{r(\delta)}^\varepsilon,\hat{Y}_r^\varepsilon)\|_{\mathcal H}^2+\|\bar{F}_1(X_{r(\delta)}^\varepsilon) \|_{\mathcal H}^2 \big) dr \Big] \leq  C \delta^{2(1-\eta)}.
\end{align}
For $L_3$, similar to the estimation of $L_2$, we have
\begin{equation}\label{L3}
L_3 \leq C \delta^{2(1-\eta)}.
\end{equation}
Finally, we estimate the term $L_4$.
\begin{align}\label{Psi}
L_4 &\leq  6\delta^{-2\eta}\mathbb E\Big[\sup_{s<t} \Big\| \sum_{k=[s/\delta]+1}^{[t/\delta]-1} \int_{k\delta}^{(k+1)\delta} S_{t-k\delta} \big[F_1(X_{k\delta}^\varepsilon,\hat{Y}_r^\varepsilon)-\bar{F}_1(X_{k\delta}^\varepsilon) \big]dr \Big\|_{\mathcal H}^2 1_{ \{|t-s|>\delta \}} \Big] \nonumber\\
&\leq  6\delta^{-2\eta}\mathbb E\Big[\sup_{s<t} \Big(\Big[\frac t\delta \Big]-\Big[\frac s\delta \Big]-1 \Big) \nonumber\\
    &~~~~~~~~~~~~~~\times \sum_{k=[s/\delta]+1}^{[t/\delta]-1}\Big\|  \int_{k\delta}^{(k+1)\delta} S_{t-k\delta} \big[F_1(X_{k\delta}^\varepsilon,\hat{Y}_r^\varepsilon)-\bar{F}_1(X_{k\delta}^\varepsilon) \big]dr \Big\|_{\mathcal H}^2 1_{ \{|t-s|>\delta \}} \Big] \nonumber\\
&\leq  C_T\delta^{-1-2\eta} \mathbb E \Big[\sum_{k=0}^{[T/\delta]-1}\Big\|  \int_{k\delta}^{(k+1)\delta} S_{t-k\delta} \big [F_1(X_{k\delta}^\varepsilon,\hat{Y}_r^\varepsilon)-\bar{F}_1(X_{k\delta}^\varepsilon) \big]dr \Big\|_{\mathcal H}^2 1_{ \{|t-s|>\delta \}} \Big] \nonumber\\
&\leq \frac {C_T}{\delta^{2(1+\eta)}} \max_{0\leq k\leq[T/\delta]-1}\mathbb E\Big\| S_{t-k\delta} \int_{k\delta}^{(k+1)\delta}  \big[F_1(X_{k\delta}^\varepsilon,\hat{Y}_r^\varepsilon)-\bar{F}_1(X_{k\delta}^\varepsilon) \big]dr \Big\|_{\mathcal H}^2  \nonumber\\
&\leq \frac {C_T}{\delta^{2(1+\eta)}} \max_{0\leq k\leq[T/\delta]-1}\mathbb E\Big\| \int_{k\delta}^{(k+1)\delta} \big[F_1(X_{k\delta}^\varepsilon,\hat{Y}_r^\varepsilon)-\bar{F}_1(X_{k\delta}^\varepsilon) \big]dr \Big\|_{\mathcal H}^2 \nonumber\\
&\leq \frac {C_T\varepsilon^2}{\delta^{2(1+\eta)}} \max_{0\leq k\leq[T/\delta]-1}
	\mathbb E\Big\|\int_0^{\frac \delta\varepsilon} \big[F_1(X_{k\delta}^\varepsilon,\hat{Y}_{r\varepsilon+k\delta}^\varepsilon)-\bar{F}_1(X_{k\delta}^\varepsilon) \big]dr \Big\|_{\mathcal H}^2  \nonumber\\
&\leq \frac {C_T\varepsilon^2}{\delta^{2(1+\eta)}} \max_{0\leq k\leq[T/\delta]-1}
\int_0^{\frac \delta\varepsilon}\int_r^{\frac \delta\varepsilon}\Psi_k(s,r)dsdr,
\end{align}
where for any $0\leq r\leq s\leq \frac \delta\varepsilon$,
\begin{equation}\label{psi}
\Psi_k(s,r):=\mathbb E\big[\langle F_1(X_{k\delta}^\varepsilon,\hat{Y}_{s\varepsilon+k\delta}^\varepsilon)-\bar{F}_1(X_{k\delta}^\varepsilon),
F_1(X_{k\delta}^\varepsilon,\hat{Y}_{r\varepsilon+k\delta}^\varepsilon)-\bar{F}_1(X_{k\delta}^\varepsilon)\rangle_{\mathcal H}\big].
\end{equation}

Now we estimate $\Psi_k(s,r)$.
For any $s>0$ and $\mathscr F_s$-measurable $\mathcal H$-valued random variables $\xi$ and $Y$, we consider the following equation
$$
\left\{ \begin{aligned}
	dY_t&=\frac 1\varepsilon \big[A Y_t+F_2(\xi,Y_t)\big]dt+\frac 1{\sqrt{\varepsilon}}G_2(\xi,Y_t)dW_t, ~~t\geq s,\\
	Y_s&=Y,
\end{aligned} \right.
$$
which has a unique solution $\{\widetilde{Y}^{\varepsilon,s,\xi,Y}_t\}_{t\geq0}$, since $F_2$ and $G_2$ are globally Lipschitz continuous by conditions (\ref{F2Lip}) and (\ref{Glip}).
Looking back at the construction of the auxiliary process (\ref{Auxeq}), for each $k\in\mathbb N$ and $t\in[k\delta,(k+1)\delta]$, we have $\mathbb P$-a.s.,
\begin{center}
	$\hat{Y}^\varepsilon_t=\widetilde{Y}^{\varepsilon,k\delta,X^\varepsilon_{k\delta},\hat{Y}^\varepsilon_{k\delta}}_t$.
\end{center}
Therefore, $\Psi_k(s,r)$ can be rewritten as
\begin{align*}
	\Psi_k(s,r)
	=&\mathbb E\Big[ \big\langle F_1(X^\varepsilon_{k\delta},\widetilde{Y}^{\varepsilon,k\delta,X^\varepsilon_{k\delta},
		\hat{Y}^\varepsilon_{k\delta}}_{s\varepsilon+k\delta})-\bar{F}_1(X^\varepsilon_{k\delta}),F_1(X^\varepsilon_{k\delta},\widetilde{Y}^{\varepsilon,k\delta,X^\varepsilon_{k\delta},
		\hat{Y}^\varepsilon_{k\delta}}_{r\varepsilon+k\delta})-
	\bar{F}_1(X^\varepsilon_{k\delta}) \big\rangle_{\mathcal H}\Big] \\
	=&\int_{\Omega}\mathbb E\Big[ \big\langle F_1(X^\varepsilon_{k\delta},\widetilde{Y}^{\varepsilon,k\delta,X^\varepsilon_{k\delta},
		\hat{Y}^\varepsilon_{k\delta}}_{s\varepsilon+k\delta})-\bar{F}_1(X^\varepsilon_{k\delta}), \\
&~~~~~~~~~~F_1(X^\varepsilon_{k\delta},\widetilde{Y}^{\varepsilon,k\delta,X^\varepsilon_{k\delta},
		\hat{Y}^\varepsilon_{k\delta}}_{r\varepsilon+k\delta})-
	\bar{F}_1(X^\varepsilon_{k\delta}) \big \rangle_{\mathcal H}\Big|\mathscr F_{k\delta}\Big](\omega)\mathbb P(d\omega)\\
	=&\int_{\Omega}\mathbb E\Big[ \big\langle F_1(X^\varepsilon_{k\delta}(\omega),\widetilde{Y}^{\varepsilon,k\delta,X^\varepsilon_{k\delta}(\omega),
		\hat{Y}^\varepsilon_{k\delta}(\omega)}_{s\varepsilon+k\delta})-
	\bar{F}_1(X^\varepsilon_{k\delta}(\omega)),\\
	&~~~~~~~~~~ F_1(X^\varepsilon_{k\delta}(\omega),\widetilde{Y}^{\varepsilon,k\delta,X^\varepsilon_{k\delta}(\omega),
		\hat{Y}^\varepsilon_{k\delta}(\omega)}_{r\varepsilon+k\delta})-
	\bar{F}_1(X^\varepsilon_{k\delta}(\omega)) \big\rangle_{\mathcal H}\Big]\mathbb P(d\omega).
\end{align*}
Here, going from second to the third equation, we used the fact that $X^\varepsilon_{k\delta}$ and $\hat{Y}^\varepsilon_{k\delta}$ are $\mathscr F_{k\delta}$-measurable, and for any fixed $(x,y)\in \mathcal H\times \mathcal H$, $\{\widetilde{Y}^{\varepsilon,k\delta,x,y}_{s\varepsilon+k\delta}\}_{s\geq0}$ is independent of $\mathscr F_{k\delta}$.

According to the definition of process $\{\widetilde{Y}^{\varepsilon,k\delta,
	x,y}_t\}_{t\geq0}$, for each $k\in\mathbb N$, using a shift transformation, we have $\mathbb P$-a.s.,
\begin{align}\label{Step3equ}
&\widetilde{Y}^{\varepsilon,k\delta,x,y}_{s\varepsilon+k\delta}\nonumber\\
=&y+\frac 1\varepsilon \int_{k\delta}^{s\varepsilon+k\delta} A \widetilde{Y}^{\varepsilon,k\delta,x,y}_r dr+\frac 1\varepsilon\int_{k\delta}^{s\varepsilon+k\delta}F_2(x,\widetilde{Y}^{\varepsilon,k\delta,x,y}_r)dr+
	\frac 1{\sqrt{\varepsilon}}\int_{k\delta}^{s\varepsilon+k\delta}G_2(x,\widetilde{Y}^{\varepsilon,k\delta,x,y}_r)dW_r\nonumber\\
=&y+\frac 1\varepsilon\int_0^{s\varepsilon}A \widetilde{Y}^{\varepsilon,k\delta,x,y}_{r+k\delta} dr+\frac 1\varepsilon\int_0^{s\varepsilon}F_2(x,\widetilde{Y}^{\varepsilon,k\delta,x,y}_{r+k\delta})dr+\frac 1{\sqrt{\varepsilon}}\int_0^{s\varepsilon}G_2(x,\widetilde{Y}^{\varepsilon,k\delta,x,y}_{r+k\delta})dW^{k\delta}_r\nonumber\\
=&y+\int_0^sA \widetilde{Y}^{\varepsilon,k\delta,x,y}_{r\varepsilon+k\delta}dr+\int_0^s F_2(x,\widetilde{Y}^{\varepsilon,k\delta,x,y}_{r\varepsilon+k\delta})dr
+\int_0^sG_2(x,\widetilde{Y}^{\varepsilon,k\delta,x,y}_{r\varepsilon+k\delta})d\bar{W}^{k\delta}_r,
\end{align}
where $W^{k\delta}_r:=W_{r+k\delta}-W_{k\delta}$ and $\bar{W}^{k\delta}_r:=\frac 1{\sqrt{\varepsilon}}W^{k\delta}_{r\varepsilon}$.

It is recalled that the solution of the frozen equation (\ref{frozeneq}) is given by
$$Y_s^{x,y}=y + \int_0^s A Y_r^{x,y} dr +\int_0^s F_2(x, Y_r^{x,y}) dr+ \int_0^s G_2(x, Y_r^{x,y}) d \widetilde{W}_r, ~~\text{for any}~s >0.$$
Therefore, the uniqueness of the solutions to Eq.~(\ref{Step3equ}) and Eq.~(\ref{frozeneq}) implies that the distribution of $\{\widetilde{Y}^{\varepsilon,k\delta,x,y}_{s\varepsilon+k\delta}\}_{0\leq s\leq \frac \delta\varepsilon}$ coincides with the distribution of $\{Y^{x,y}_s\}_{0\leq s\leq \frac \delta\varepsilon}$. So we have
\begin{align*}
\Psi_k(s,r)=&\int_\Omega \widetilde{\mathbb E} \Big[ \big\langle F_1(X^\varepsilon_{k\delta}(\omega),Y^{X^\varepsilon_{k\delta}(\omega),
		\hat{Y}^\varepsilon_{k\delta}(\omega)}_s)-\bar{F}_1(X^\varepsilon_{k\delta}(\omega)),\\
&\qquad\quad F_1(X^\varepsilon_{k\delta}(\omega),Y^{X^\varepsilon_{k\delta}(\omega),
		\hat{Y}^\varepsilon_{k\delta}(\omega)}_r)-\bar{F}_1(X^\varepsilon_{k\delta}(\omega)) \big\rangle_{\mathcal H}\Big] \mathbb P(d\omega)\\
=&\int_\Omega\int_{\widetilde{\Omega}}\Big\langle \widetilde{\mathbb E} \Big[F_1(X^\varepsilon_{k\delta}(\omega),Y^{X^\varepsilon_{k\delta}(\omega),
		Y^{X^\varepsilon_{k\delta}(\omega),\hat{Y}^\varepsilon_{k\delta}(\omega)}_r(\omega')}_{s-r}-
	\bar{F}_1(X^\varepsilon_{k\delta}(\omega))\Big],\\
&\qquad\qquad F_1(X^\varepsilon_{k\delta}(\omega),Y^{X^\varepsilon_{k\delta}(\omega),
		\hat{Y}^\varepsilon_{k\delta}(\omega)}_r(\omega'))-\bar{F}_1(X^\varepsilon_{k\delta}(\omega))\Big\rangle_{\mathcal H}
\widetilde{\mathbb P}(d\omega')\mathbb P(d\omega),	
\end{align*}
where in the second equality, we used the Markov property of process $\{Y^{x,y}_t\}$.
Using Lemma \ref{ergo} and estimate (\ref{frozenest}) we have
\begin{align}\label{Psiest}
\Psi_k(s,r)
&\leq C\int_\Omega \int_{\widetilde{\Omega}}\Big\{\Big[1+\|X^\varepsilon_{k\delta}(\omega)\|_{\mathcal H}+\|Y^{X^\varepsilon_{k\delta}(\omega),
		\hat{Y}^\varepsilon_{k\delta}(\omega)}_r(\omega')\|_{\mathcal H}\Big]  e^{-(s-r)\rho}  \nonumber\\
&~~~\times\Big[1+\|X^\varepsilon_{k\delta}(\omega)\|_{\mathcal H}+\|Y^{X^\varepsilon_{k\delta}(\omega),
		\hat{Y}^\varepsilon_{k\delta}(\omega)}_r(\omega')\|_{\mathcal H}\Big]\Big\}\widetilde{\mathbb
	P}(d\omega')\mathbb P(d\omega)  \nonumber\\
&\leq C\int_\Omega\Big[1+\|X^\varepsilon_{k\delta}(\omega)\|_{\mathcal H}^2+\|\hat{Y}^\varepsilon_{k\delta}(\omega)\|_{\mathcal H}^2\Big]\mathbb P(d\omega) e^{-(s-r)\rho}   \nonumber\\
&\leq C\int_\Omega\Big[1+\|X^\varepsilon_{k\delta}(\omega)\|_{\mathcal H_\eta}^2+\|\hat{Y}^\varepsilon_{k\delta}(\omega)\|_{\mathcal H}^2\Big]\mathbb P(d\omega) e^{-(s-r)\rho}   \nonumber\\
&\leq C(1+\|x\|_{\mathcal H_\eta}^2+\|y\|_{\mathcal H}^2) e^{-(s-r)\rho} .
\end{align}
In the third inequality, we used the embedding $\mathcal H_\eta \subset \mathcal H$ continuously, and in the last inequality, we used Lemma \ref{Xetaest} and estimate (\ref{Auxest}).

By substituting (\ref{Psiest}) in (\ref{Psi}), we obtain
\begin{align}\label{L4}
L_4 & \leq \frac {C_T\varepsilon^2}{\delta^{2(1+\eta)}} (1+\|x\|_{\mathcal H_\eta}^2+\|y\|_{\mathcal H}^2) \int_0^{\frac \delta\varepsilon}\int_r^{\frac \delta\varepsilon} e^{-(s-r)\rho}  dsdr\nonumber\\
&\leq C_T (1+\|x\|_{\mathcal H_\eta}^2+\|y\|_{\mathcal H}^2) \frac {\varepsilon^2}{\delta^{2(1+\eta)}} \Big(\frac \delta{\varepsilon\rho}-\frac 1{\rho^2}+\frac 1{\rho^2}e^{-\frac {\rho\delta}{\varepsilon}}\Big).
\end{align}
Substituting (\ref{K2eta1})-(\ref{L3}) and (\ref{L4}) into (\ref{K2eta}), we have
\begin{align*}
\mathbb E\|K_2(\cdot)\|_{\eta,0}^2
\leq C_T(1+\|x\|_{\mathcal H_\eta}^2+\|y\|_{\mathcal H}^2)
\Big(\delta^{2\sigma}+\delta^{2(1-\eta)}+\frac \varepsilon{\delta^{1+2\eta}}+\frac {\varepsilon^2}{\delta^{2(1+\eta)}} \Big).
\end{align*}

Now, we estimate the remainder term $\mathbb E |R^{K_2(\cdot)}|_{2\gamma,-2\gamma}^2$ of $K_2(t)$.
By means of (\ref{remainder}), it is easy to see that
\begin{align*}
\mathbb E |R^{K_2(\cdot)}|_{2\gamma,-2\gamma}^2&=\mathbb E \Big[ \sup_{s<t} |t-s|^{-4\gamma}
 \Big\| \int_s^t S_{t-r}\big[F_1(X_{r(\delta)}^\varepsilon,\hat{Y}_r^\varepsilon)-\bar{F}_1(X_{r(\delta)}^\varepsilon) \big]dr \Big\|_{\mathcal{H}_{-2\gamma}}^2 \Big] \nonumber\\
&\leq 2\mathbb E \Big[ \sup_{s<t} |t-s|^{-4\gamma}
 \Big\| \int_s^t \big( S_{t-r}-S_{t-r(\delta)} \big) \big[F_1(X_{r(\delta)}^\varepsilon,\hat{Y}_r^\varepsilon)-\bar{F}_1(X_{r(\delta)}^\varepsilon) \big]dr \Big\|_{\mathcal{H}_{-2\gamma}}^2 \Big] \nonumber\\
&~~~+ 2\mathbb E \Big[ \sup_{s<t} |t-s|^{-4\gamma}
 \Big\| \int_s^t  S_{t-r(\delta)}  \big[F_1(X_{r(\delta)}^\varepsilon,\hat{Y}_r^\varepsilon)-\bar{F}_1(X_{r(\delta)}^\varepsilon) \big]dr \Big\|_{\mathcal{H}_{-2\gamma}}^2 \Big] \nonumber\\
&=: \widetilde{K}_{21}+\widetilde{K}_{22} .
\end{align*}
For $\widetilde{K}_{21}$, using H\"{o}lder's inequality, (\ref{semi}), (\ref{F1Bound}) and $\|\bar{F}_1\|_\infty \leq \|F_1\|_\infty$, we have
\begin{align}\label{K2rem1}
\widetilde{K}_{21}&\leq  4\mathbb E \Big[ \sup_{s<t} |t-s|^{-4\gamma}
  \int_s^t \| S_{t-r}-S_{t-r(\delta)} \|_{\mathcal L(\mathcal{H}_{-2\gamma};\mathcal{H}_{-2\gamma})}^2 dr  \nonumber\\
&~~~~~~~~~~~~~~~~~~~~~~~\times \int_s^t \big(\|F_1(X_{r(\delta)}^\varepsilon,\hat{Y}_r^\varepsilon)\|_{\mathcal{H}_{-2\gamma}}^2+\|\bar{F}_1(X_{r(\delta)}^\varepsilon)\|_{\mathcal{H}_{-2\gamma}}^2 \big) dr \Big] \nonumber\\
&\leq  C\mathbb E \Big[ \sup_{s<t} |t-s|^{-4\gamma}
  \int_s^t \| S_{t-r}-S_{t-r(\delta)} \|_{\mathcal L(\mathcal{H}_{-2\gamma};\mathcal{H}_{-2\gamma})}^2 dr  \nonumber\\
&~~~~~~~~~~~~~~~~~~~~~~~~\times \int_s^t \big(\|F_1(X_{r(\delta)}^\varepsilon,\hat{Y}_r^\varepsilon)\|_{\mathcal{H}}^2+\|\bar{F}_1(X_{r(\delta)}^\varepsilon)\|_{\mathcal{H}}^2 \big) dr \Big] \nonumber\\
&\leq C   \sup_{s<t}\Big[ |t-s|^{1-4\gamma}
  \int_s^t \| S_{t-r}-S_{t-r} \cdot S_{r-r(\delta)} \|_{\mathcal L(\mathcal{H}_{-2\gamma};\mathcal{H}_{-2\gamma})}^2 dr  \Big] \nonumber\\
&\leq C   \sup_{s<t} \Big[ |t-s|^{1-4\gamma}
  \int_s^t \| S_{t-r} \|_{\mathcal L(\mathcal H_{-2\gamma-\sigma^{\prime}};\mathcal{H}_{-2\gamma})}^2  \| S_{r-r(\delta)}-Id \|_{\mathcal L(\mathcal{H}_{-2\gamma};\mathcal H_{-2\gamma-\sigma^{\prime}})}^2 dr  \Big] \nonumber\\
&\leq C \delta^{2\sigma^{\prime}}   \sup_{s<t} \Big[ |t-s|^{1-4\gamma}
  \int_s^t |t-r|^{-2\sigma^{\prime}}  dr  \Big] \nonumber\\
&\leq C \delta^{2\sigma^{\prime}}  \sup_{s<t} |t-s|^{2(1-2\gamma-\sigma^{\prime})} \leq C T^{2(1-2\gamma-\sigma^{\prime})} \delta^{2\sigma^{\prime}},
\end{align}
where $0<\sigma^{\prime}<1-2\gamma$.
With regard to $\widetilde{K}_{22}$, we notice that
\begin{align*}
\widetilde{K}_{22}&\leq 2\mathbb E \Big[ \sup_{s<t} |t-s|^{-4\gamma}
 \Big\| \int_s^t  S_{t-r(\delta)}  \big[F_1(X_{r(\delta)}^\varepsilon,\hat{Y}_r^\varepsilon)-\bar{F}_1(X_{r(\delta)}^\varepsilon) \big]dr \Big\|_{\mathcal{H}_{-2\gamma}}^2 1_{\{|t-s|\leq \delta \}}\Big] \nonumber\\
&~~~+ 2\mathbb E \Big[ \sup_{s<t} |t-s|^{-4\gamma}
 \Big\| \int_s^t  S_{t-r(\delta)}  \big[F_1(X_{r(\delta)}^\varepsilon,\hat{Y}_r^\varepsilon)-\bar{F}_1(X_{r(\delta)}^\varepsilon) \big]dr \Big\|_{\mathcal{H}_{-2\gamma}}^2 1_{ \{|t-s|>\delta \}} \Big] \nonumber\\
&\leq 2\mathbb E \Big[ \sup_{s<t} |t-s|^{-4\gamma}
 \Big\| \int_s^t  S_{t-r(\delta)}  \big[F_1(X_{r(\delta)}^\varepsilon,\hat{Y}_r^\varepsilon)-\bar{F}_1(X_{r(\delta)}^\varepsilon) \big]dr \Big\|_{\mathcal{H}_{-2\gamma}}^2 1_{\{|t-s|\leq \delta \}}\Big] \nonumber\\
&~~~+ 6\mathbb E \Big[ \sup_{s<t} |t-s|^{-4\gamma}
 \Big\| \int_s^{([s/\delta]+1)\delta}  S_{t-r(\delta)}  \big[F_1(X_{r(\delta)}^\varepsilon,\hat{Y}_r^\varepsilon)-\bar{F}_1(X_{r(\delta)}^\varepsilon) \big]dr \Big\|_{\mathcal{H}_{-2\gamma}}^2 1_{ \{|t-s|>\delta \}} \Big] \nonumber\\
&~~~+ 6\mathbb E \Big[ \sup_{s<t} |t-s|^{-4\gamma}
 \Big\| \int_{t(\delta)}^t  S_{t-r(\delta)}  \big[F_1(X_{r(\delta)}^\varepsilon,\hat{Y}_r^\varepsilon)-\bar{F}_1(X_{r(\delta)}^\varepsilon) \big]dr \Big\|_{\mathcal{H}_{-2\gamma}}^2 1_{ \{|t-s|>\delta \}} \Big] \nonumber\\
&~~~+ 6\mathbb E \Big[ \sup_{s<t} |t-s|^{-4\gamma}
 \Big\| \sum_{k=[s/\delta]+1}^{[t/\delta]-1} \int_{k\delta}^{(k+1)\delta} S_{t-k\delta} \big[F_1(X_{k\delta}^\varepsilon,\hat{Y}_r^\varepsilon)-\bar{F}_1(X_{k\delta}^\varepsilon) \big]dr \Big\|_{\mathcal{H}_{-2\gamma}}^2 1_{ \{|t-s|>\delta \}} \Big] \nonumber\\
&=:\widetilde{L}_1+\widetilde{L}_2+\widetilde{L}_3+\widetilde{L}_4.
\end{align*}

For $\widetilde{L}_1, \widetilde{L}_2$ and $\widetilde{L}_3$, similar to estimates for $L_1, L_2$ and $L_3$, using H\"{o}lder's inequality and the embedding $\mathcal H \subset \mathcal H_{-2\gamma}$ continuously, we have
\begin{align}
\widetilde{L}_1 &\leq  C \mathbb E \Big[ \sup_{s<t} |t-s|^{-4\gamma}
  \int_s^t \|S_{t-r(\delta)} \|_{\mathcal L(\mathcal{H}_{-2\gamma};\mathcal{H}_{-2\gamma})}^2 dr  \nonumber\\
  &~~~~~~~~~~~~~~~~~\times \int_s^t  \big( \| F_1(X_{r(\delta)}^\varepsilon,\hat{Y}_r^\varepsilon)\|_{\mathcal{H}}^2
  +\|\bar{F}_1(X_{r(\delta)}^\varepsilon) \|_{\mathcal{H}}^2 \big) dr 1_{\{|t-s|\leq \delta \}}\Big] \leq  C \delta^{2(1-2\gamma)}, \label{tildeL1}\\
\widetilde{L}_2
&\leq  C\delta^{-4\gamma} \mathbb E \Big[ \sup_{s<t}
  \int_s^{([s/\delta]+1)\delta} \|S_{t-r(\delta)}\|_{\mathcal L(\mathcal{H}_{-2\gamma};\mathcal{H}_{-2\gamma})}^2 dr \nonumber\\
  &~~~~~~~~~~~~\times \int_s^{([s/\delta]+1)\delta}   \big(\| F_1(X_{r(\delta)}^\varepsilon,\hat{Y}_r^\varepsilon)\|_{\mathcal{H}}^2+\|\bar{F}_1(X_{r(\delta)}^\varepsilon) \|_{\mathcal{H}}^2 \big) dr \Big] \leq  C \delta^{2(1-2\gamma)}, \label{tildeL2} \\
\widetilde{L}_3 &\leq  C\delta^{-4\gamma} \mathbb E \Big[ \sup_{s<t}
  \int_{t(\delta)}^t \|S_{t-r(\delta)}\|_{\mathcal L(\mathcal{H}_{-2\gamma};\mathcal{H}_{-2\gamma})}^2 dr \nonumber\\
  &~~~~~~~~~~~~~~~~~\times \int_{t(\delta)}^t   \big(\| F_1(X_{r(\delta)}^\varepsilon,\hat{Y}_r^\varepsilon)\|_{\mathcal{H}}^2+\|\bar{F}_1(X_{r(\delta)}^\varepsilon) \|_{\mathcal{H}}^2 \big) dr \Big] \leq  C \delta^{2(1-2\gamma)}. \label{tildeL3}
\end{align}
For $\widetilde{L}_4$, making use of H\"{o}lder's inequality, the compressibility of semigroup $(S_t)_{t\geq 0}$ and the fact that the embedding $\mathcal H \subset \mathcal H_{-2\gamma}$ is continuous, it can be obtained that
\begin{align*}
\widetilde{L}_4 &\leq  6\delta^{-4\gamma}\mathbb E\Big[\sup_{s<t} \Big\| \sum_{k=[s/\delta]+1}^{[t/\delta]-1} \int_{k\delta}^{(k+1)\delta} S_{t-k\delta} \big[F_1(X_{k\delta}^\varepsilon,\hat{Y}_r^\varepsilon)-\bar{F}_1(X_{k\delta}^\varepsilon) \big]dr \Big\|_{\mathcal{H}_{-2\gamma}}^2 1_{ \{|t-s|>\delta \}} \Big] \\
&\leq  6\delta^{-4\gamma}\mathbb E\Big[\sup_{s<t} \Big(\Big[\frac t\delta \Big]-\Big[\frac s\delta \Big]-1 \Big) \\
    &~~~~~~~~~~~~~~\times \sum_{k=[s/\delta]+1}^{[t/\delta]-1}\Big\|  \int_{k\delta}^{(k+1)\delta} S_{t-k\delta} \big[F_1(X_{k\delta}^\varepsilon,\hat{Y}_r^\varepsilon)-\bar{F}_1(X_{k\delta}^\varepsilon) \big]dr \Big\|_{\mathcal{H}_{-2\gamma}}^2 1_{ \{|t-s|>\delta \}} \Big] \\
&\leq  C_T\delta^{-1-4\gamma} \mathbb E \Big[\sum_{k=0}^{[T/\delta]-1}\Big\|  \int_{k\delta}^{(k+1)\delta} S_{t-k\delta} \big [F_1(X_{k\delta}^\varepsilon,\hat{Y}_r^\varepsilon)-\bar{F}_1(X_{k\delta}^\varepsilon) \big]dr \Big\|_{\mathcal{H}_{-2\gamma}}^2 1_{ \{|t-s|>\delta \}} \Big]\\
&\leq \frac {C_T}{\delta^{2(1+2\gamma)}} \max_{0\leq k\leq[T/\delta]-1}\mathbb E\Big\| S_{t-k\delta} \int_{k\delta}^{(k+1)\delta}  \big[F_1(X_{k\delta}^\varepsilon,\hat{Y}_r^\varepsilon)-\bar{F}_1(X_{k\delta}^\varepsilon) \big]dr \Big\|_{\mathcal{H}_{-2\gamma}}^2\\
&\leq \frac {C_T}{\delta^{2(1+2\gamma)}} \max_{0\leq k\leq[T/\delta]-1}\mathbb E\Big\| \int_{k\delta}^{(k+1)\delta} \big[F_1(X_{k\delta}^\varepsilon,\hat{Y}_r^\varepsilon)-\bar{F}_1(X_{k\delta}^\varepsilon) \big]dr \Big\|_{\mathcal{H}_{-2\gamma}}^2\\
&\leq \frac {C_T\varepsilon^2}{\delta^{2(1+2\gamma)}} \max_{0\leq k\leq[T/\delta]-1}
	\mathbb E\Big\|\int_0^{\frac \delta\varepsilon} \big[F_1(X_{k\delta}^\varepsilon,\hat{Y}_{r\varepsilon+k\delta}^\varepsilon)-\bar{F}_1(X_{k\delta}^\varepsilon) \big]dr \Big\|_{\mathcal H}^2\\
&\leq \frac {C_T\varepsilon^2}{\delta^{2(1+2\gamma)}} \max_{0\leq k\leq[T/\delta]-1}
\int_0^{\frac \delta\varepsilon}\int_r^{\frac \delta\varepsilon}\Psi_k(s,r)dsdr.
\end{align*}
Here $\Psi_k(s,r)$ is shown in (\ref{psi}), so similarly we have
\begin{equation}\label{tildeL4}
\widetilde{L}_4
\leq C_T (1+\|x\|_{\mathcal H_\eta}^2+\|y\|_{\mathcal H}^2) \frac {\varepsilon^2}{\delta^{2(1+2\gamma)}} \Big(\frac \delta{\varepsilon\rho}-\frac 1{\rho^2}+\frac 1{\rho^2}e^{-\frac {\rho\delta}{\varepsilon}}\Big).
\end{equation}
Combining estimates (\ref{K2rem1})-(\ref{tildeL4}) we have
\begin{align*}\label{Step4Res}
\mathbb E |R^{K_2(\cdot)}|_{2\gamma,-2\gamma}^2
\leq C_T(1+\|x\|_{\mathcal H_\eta}^2+\|y\|_{\mathcal H}^2)\Big(\delta^{2\sigma^{\prime}}+\delta^{2(1-2\gamma)}+\frac \varepsilon{\delta^{1+4\gamma}}+\frac {\varepsilon^2}{\delta^{2(1+2\gamma)}} \Big).
\end{align*}
Finally, according to the proof of \textbf{Step2} and \textbf{Step3}, we take $\delta=\varepsilon^{\frac 1{2(1+2\gamma)}}$, then the following formula holds.
\begin{align*}
&\mathbb E\|\mathcal M^\varepsilon,0\|_{\mathcal D_{\mathbf B}^{2 \gamma, \eta}}^2\\
\leq & C_{M,T} (1+\|x\|_{\mathcal H_\eta}^2+\|y\|_{\mathcal H}^2)  \\
&\times \Big(\delta^{2\eta}+\delta^{2(1-2\gamma)}+\delta^{2(1-\eta)}+\delta^{\frac 12+2\eta-2\gamma}+ \delta^{2\sigma}+\delta^{2\sigma^{\prime}} +\frac \varepsilon{\delta^{1+4\gamma}}+\frac {\varepsilon^2}{\delta^{2(1+2\gamma)}}+\frac \varepsilon{\delta^{1+2\eta}}+\frac {\varepsilon^2}{\delta^{2(1+\eta)}} \Big)\\
\leq& C_{M,T} (1+\|x\|_{\mathcal H_\eta}^2+\|y\|_{\mathcal H}^2) \Big(\delta^{\frac 12+2\eta-2\gamma}+\delta^{2(1-2\gamma)}+ \delta^{2\sigma}+\delta^{2\sigma^{\prime}} +\frac \varepsilon{\delta^{1+4\gamma}}+\frac {\varepsilon^2}{\delta^{2(1+2\gamma)}} \Big) \\
\leq& C_{M,T} (1+\|x\|_{\mathcal H_\eta}^2+\|y\|_{\mathcal H}^2)  \Big(\varepsilon^{\frac {1/2+2\eta-2\gamma} {2(1+2\gamma)}}+\varepsilon^{\frac {1-2\gamma}{1+2\gamma}}+ \varepsilon^{ \frac \sigma{1+2\gamma}}+\varepsilon^{ \frac {\sigma^{\prime}} {1+2\gamma} } +\varepsilon^{ \frac 1{2(1+2\gamma)} }+\varepsilon \Big).
\end{align*}
Letting $\varepsilon\rightarrow0$, we have
\begin{equation*}
	\lim_{\varepsilon\rightarrow0} \mathbb E\|\mathcal M^\varepsilon,0\|_{\mathcal D_{\mathbf B}^{2 \gamma, \eta}}^2=0.
\end{equation*}
Therefore, by estimate (\ref{Xdifeta}), it is straightforward that
$$\lim_{\varepsilon \to 0} \mathbb E \big[\|X^\varepsilon-\bar{X}\|_{\eta, 0}^2 \big]=0.$$
The proof is complete.
\hfill $\square$

\end{proof}

\section*{Acknowledgements}
M. Li and B. Pei were partially supported by National Natural Science Foundation of China (NSFC) under Grant No.12172285, Guangdong Basic and Applied Basic Research Foundation under Grant No. 2024A1515012164, Shaanxi
Fundamental Science Research Project for Mathematics and Physics under Grant No.22JSQ027 and Fundamental
Research Funds for the Central Universities. 
Y. Li was partially supported by NSFC under Grant No. 12301566, Science and Technology Commission of Shanghai Municipality under Grant No. 23JC1400300, and Chenguang Program of Shanghai Education Development Foundation and Shanghai Municipal Education Commission under Grant No.22CGA01.
Y. Xu was partially supported by Key International (Regional) Joint Research Program of NSFC under Grant No.12120101002.

\end{document}